\documentclass[12pt]{article}

\usepackage{graphicx}

\usepackage[margin=1in]{geometry}

\usepackage{amsmath,amssymb,amsthm}
\usepackage{xcolor,setspace,adjustbox}
\usepackage{booktabs,multirow,makecell,threeparttable}
\usepackage{verbatim}
\usepackage{csquotes}
\usepackage{siunitx}
\usepackage[colorlinks=true,linkcolor=blue,citecolor=blue,urlcolor=blue]{hyperref}
\usepackage{apacite}

\newcommand{\bfTheta}{\boldsymbol{\Theta}}
\newcommand{\bfalpha}{\boldsymbol{\alpha}}
\newcommand{\bfepsilon}{\boldsymbol{\epsilon}}
\newcommand{\bfpi}{\boldsymbol{\pi}}
\newcommand{\bftheta}{\boldsymbol{\theta}}
\newcommand{\bfxi}{\boldsymbol{\xi}}
\newtheorem{theorem}{Theorem}[section]
\newtheorem{lemma}[theorem]{Lemma}
\newtheorem{remark}[theorem]{Remark}
\newtheorem{assumption}{Assumption}
\newtheorem{condition}{Condition}[section]

\begin{document}

\title{On the Asymptotics of Item Selection in Multidimensional Computerized Adaptive Testing}
\author{
  Seungwon Lee\\
  School of Statistics, University of Minnesota\vspace{0.5cm}\\
  Xiaoou Li\\
  School of Statistics, University of Minnesota}
\date{}
\maketitle

\onehalfspacing
\begin{abstract}
We study Fisher-information-based item-selection rules for multidimensional computerized adaptive testing (MCAT) with intentional and nuisance abilities. Although such rules are widely used, rigorous asymptotic theory for the coupled dynamics of adaptive selection and latent-trait estimation remains limited. We consider a weighted A-optimality criterion that prioritizes intentional abilities while controlling accuracy in nuisance dimensions. For the multidimensional two-parameter logistic model, we first establish asymptotic optimality under an item-type reuse regime: paired with the maximum likelihood estimator, the weighted A-optimal rule attains the minimum asymptotic weighted mean squared error. We then prove consistency and asymptotic normality of the maximum likelihood estimator and express its asymptotic weighted mean squared error through a limiting Fisher information matrix. We also extend the analysis to the operational setting in which each item is administered at most once. These results justify normal approximations and information-based standard errors in MCAT. A simulation study examines finite-sample behavior at practical test lengths and shows favorable weighted mean squared error, with item-selection patterns consistent with the predicted limiting behavior.
\end{abstract}


\section{Introduction}
Computerized adaptive testing (CAT) is a testing framework in which items are selected sequentially and adaptively based on a test-taker’s previous responses. By tailoring the test to the individual’s latent trait, CAT typically achieves more accurate ability estimation with shorter test lengths than non-adaptive testing designs. Test responses in CAT are commonly analyzed using item response theory (IRT) models, in which the probability of a correct response is expressed as a function of an unobserved latent trait (i.e., ability parameter) associated with the test-taker.

When the latent trait is unidimensional, an individual’s ability is represented by a single scalar parameter, leading to unidimensional CAT (UCAT). In contrast, multidimensional CAT (MCAT) extends this framework by modeling ability as a vector of multiple latent dimensions through multidimensional IRT (MIRT) models \cite{luecht1996multidimensional, mcdonald1997normal, reckase1997linear}. Central to both UCAT and MCAT is the item selection rule, which determines the next item to be administered given the test-taker's history. Consequently, a variety of item selection methods have been proposed and studied in the literature.

One line of work selects items using the Fisher information matrix, which quantifies the information an item contributes to the estimation of latent ability parameters. This approach includes A-optimal and D-optimal criteria, both rooted in the optimal experimental design literature \cite{silvey2013optimal,atkinson2007optimum}. Under A-optimality, the next item is chosen to minimize the trace of the inverse Fisher information matrix, whereas under D-optimality, the next item is chosen to minimize the determinant of the inverse Fisher information matrix. Because these criteria summarize the information matrix in different ways, they generally lead to different item selections. Their empirical performance in MCAT was examined by \citeA{mulder2009multidimensional}. In the same work, the authors introduced the distinction between intentional and nuisance abilities in MCAT, where intentional abilities represent the primary traits of interest and nuisance abilities capture secondary dimensions affecting responses. This distinction enables more targeted and efficient measurement of the abilities of interest. To this end, they proposed the $\mathrm{D}_s$-optimality criterion, which focuses on efficient estimation of the intentional abilities while accounting for the presence of nuisance dimensions. Beyond Fisher information-based approaches, several alternative criteria have been proposed for both UCAT and MCAT. These include Kullback–Leibler information-based criteria \cite{chang1996global} and Bayesian mutual information-based criteria \cite{veldkamp2002multidimensional,mulder2009kullback}. Comparative studies of Fisher information-based, KL-based, and mutual information-based approaches can be found in \citeA{chang2015psychometrics}, \citeA{wang2009kullback}, and \citeA{wang2011item}.

Despite the substantial literature on item selection rules for CAT and MCAT, most existing work evaluates these rules primarily through empirical simulation. While the asymptotic theoretical properties of ability estimators have been rigorously established for certain UCAT designs (e.g., \citeNP{chang2009nonlinear,wang2017computerized}), to the best of our knowledge, rigorous theoretical results regarding the consistency and asymptotic normality of latent trait estimation and the limiting behavior of item selection rules in MCAT are largely absent from the psychometric literature. A primary challenge in establishing this theoretical support is that standard selection rules are typically myopic (i.e., optimizing for the immediate next item), whereas consistent estimation in a multidimensional setting requires ensuring that the selected items sufficiently span the directions of the latent trait space.

In this paper, we provide a theoretical analysis that bridges this gap, linking common empirical practice to a rigorous asymptotic theory for MCAT. Focusing on Fisher information-based approaches in the presence of intentional and nuisance abilities, we establish the asymptotic normality and optimality of item selection criteria under specific performance metrics, including mean squared error. These results offer a rigorous justification for using normal approximations in constructing confidence intervals and for quantifying the efficiency of different selection rules. We also complement this asymptotic theory with simulation studies, demonstrating that our theoretical results provide valuable insights for practical test settings.

We build upon and extend the general active learning framework established by \citeA{li2025globally}, who analyzed a general class of active estimation problems with applications to MCAT. However, their theoretical results were limited to settings where items may be selected repeatedly in the same test. In contrast, our framework accommodates the operational MCAT constraint that each item is administered at most once to a given examinee, a setting that is more representative of realistic MCAT applications. Furthermore, unlike \citeA{li2025globally}, our analysis distinguishes between intentional and nuisance abilities, thereby addressing a more general problem and enabling more accurate psychometric measurement.

The rest of the paper is organized as follows. In Section \ref{sec:problem}, we formulate the problem of interest, including the modeling of intentional and nuisance abilities. Section \ref{sec:methods} introduces the estimation procedure, the item-selection rules, and the evaluation metrics. Section \ref{sec:theory} presents the theoretical analysis. Section \ref{sec:simulation} reports simulation studies that assess the theoretical results under practical test lengths. Finally, Section \ref{sec:discuss} summarizes the main findings and discusses potential directions for future research. Technical proofs and additional simulation results are given in the Appendix.

\section{Problem Formulation}
\label{sec:problem}
Consider an MCAT with item pool $\mathcal{J}$. Let $Y \in \{0,1\}$ be the response to item $j \in \mathcal{J}$, where $Y=1$ indicates a correct answer and $Y=0$ otherwise. Under a MIRT model, a test-taker is characterized by an ability parameter $\bftheta \in \mathbb{R}^d$, and the response probability is given by an item response function (IRF) $p_j(\bftheta) = \Pr(Y=1|\bftheta)$.
In this study, we focus on the multidimensional two-parameter logistic (M2PL) model. Each item $j \in \mathcal{J}$ is associated with a discrimination parameter $\boldsymbol{\alpha}_j \in \mathbb{R}^d$ and a difficulty parameter $b_j \in \mathbb{R}$, with response probability
\begin{equation} \label{eq:model}
p_j(\bftheta) = \operatorname{logit}^{-1}(\boldsymbol{\alpha}_j^\top\bftheta - b_j),
\end{equation}
where $\operatorname{logit}^{-1}(x) = \exp(x)/(1+\exp(x))$. The item parameters $\{(\boldsymbol{\alpha}_j, b_j): j \in \mathcal{J}\}$ are assumed to be known from prior calibration.

In an MCAT, items are selected adaptively and responses are collected sequentially. Specifically, for a test of length $T$, at each step $t=1, \ldots, T$, let $j_t \in \mathcal{J}$ denote the index of the selected item and $Y_t \in \{0,1\}$ denote the corresponding response. The item $j_t$ is chosen according to an item-selection rule that depends on the history observed up to step $t-1$, defined as $H_{t-1} = \{j_1, \ldots, j_{t-1}, Y_1, \ldots, Y_{t-1}\}$. The response $Y_t$ is then observed with conditional distribution $Y_{t} \mid H_{t-1} \sim \operatorname{Bernoulli}(p_{j_t}(\bftheta))$. After observing $Y_{t}$, an ability estimator $\widehat{\bftheta}_t$ is obtained based on the accumulated history $H_t$.

\subsection{Partitioning Abilities}

Following \citeA{mulder2009multidimensional}, we partition the ability vector as $\bftheta = (\bftheta_I^\top, \bftheta_N^\top)^\top$, where $\bftheta_I \in \mathbb{R}^{d_I}$ ($d_I \ge 1$) and $\bftheta_N \in \mathbb{R}^{d_N}$ are column vectors representing the intentional and nuisance dimensions, respectively, with $d = d_I + d_N$. The intentional abilities are of primary interest, while the nuisance abilities are secondary traits that may influence responses but are not the target of measurement. The goal of MCAT in this setting is to estimate $\bftheta_I$ accurately while controlling for the uncertainty in $\bftheta_N$.

Although this partition is introduced to address MCAT with nuisance dimensions, our theoretical results are more general: they cover the standard MCAT case (where $d_N=0$ and all dimensions are intentional) as well as the UCAT case (where $d=1$).

\section{Methods and Evaluation Metrics} \label{sec:methods}
In this section, we introduce the ability estimators, item-selection rules, and evaluation metrics used in the paper.

\subsection{Ability Estimators}
We first define the maximum likelihood (ML) estimator. At step $t$, it is given by
\begin{equation*}
\hat{\bftheta}_t^{\mathrm{ML}}=\arg\max_{\bftheta \in \bfTheta} \sum_{s=1}^t\log p_{j_s}(\bftheta)^{Y_s}(1-p_{j_s}(\bftheta))^{1-Y_s},
\end{equation*}
where $\bfTheta \subseteq \mathbb{R}^d$ is a compact parameter space (e.g., $\bfTheta = [-3,3]^d$). Other commonly used estimators include Bayesian procedures such as posterior mode estimators \cite{segall1996multidimensional, segall2009principles}.

In this study, we focus on the ML estimator because of its widespread use in operational MCAT systems. In practice, the covariance of the ML estimator is often approximated by the inverse of the cumulative Fisher information matrix.
The cumulative Fisher information matrix after $t-1$ steps is defined as 
\begin{equation} \label{e:cumulative_fisher}
\Lambda_{t-1}( \bftheta )=\sum_{s=1}^{t-1} \mathcal{I}_{j_s}(\bftheta), 
\qquad \mathcal{I}_j(\bftheta)=p_j(\bftheta)(1-p_j(\bftheta))\bfalpha_j\bfalpha_j^\top,
\end{equation}
where $\mathcal{I}_j(\bftheta)$ is the Fisher information matrix of item $j$ at $\bftheta$.

\subsection{Item-Selection Rules} \label{sub:itemselection}
In this subsection, we describe item-selection rules based on the Fisher information matrix. These rules choose the next item by minimizing a scalar criterion of the inverse Fisher information matrix. We first introduce rules that treat all latent traits equally, followed by criteria designed to distinguish between intentional and nuisance abilities.
\paragraph{Item-Selection Rules Treating All Abilities Equally.}
When all dimensions of $\bftheta$ are of equal interest, we choose the next item that minimizes the determinant (i.e., D-optimality) or trace (i.e., A-optimality) of the inverse Fisher information matrix.

\paragraph{D-optimal item-selection rule} 
For the item pool $\mathcal{J}$, the D-optimal item-selection rule selects the item $j_t$ as follows.
\begin{equation} \label{e:d-optimal}
 j_t \in \underset{j \in \mathcal{J} \setminus \{j_1, \ldots, j_{t-1}\}}{\arg \min} \operatorname{det} \Big[ \Big\{ \Lambda_{t-1}( \hat{\bftheta}_{t-1}^{\mathrm{ML}}) + \mathcal{I}_j(\hat{\bftheta}_{t-1}^{\mathrm{ML}}) \Big\}^{-1} \Big],
\end{equation}
where $\operatorname{det}(\cdot)$ denotes the determinant, and $\Lambda_{t-1}(\hat{\bftheta}_{t-1}^{\mathrm{ML}})$ and $\mathcal{I}_j(\hat{\boldsymbol{\theta}}_{t-1}^{\mathrm{ML}})$ are defined in Equation~\eqref{e:cumulative_fisher}. This D-optimal item-selection rule has also been considered in \citeA{mulder2009multidimensional} and \citeA{wang2011item}. A Bayesian variant that uses the posterior covariance matrix under a multivariate normal prior was proposed in \citeA{segall1996multidimensional}.

\paragraph{A-optimal item-selection rule} 
For the item pool $\mathcal{J}$, the A-optimal item-selection rule selects the item $j_t$ as follows.
\begin{equation} \label{e:a-optimal}
 j_t \in \underset{j \in \mathcal{J} \setminus \{j_1, \ldots, j_{t-1}\}}{\arg \min} \operatorname{tr} \Big[ \Big\{ \Lambda_{t-1}( \hat{\bftheta}_{t-1}^{\mathrm{ML}}) + \mathcal{I}_j(\hat{\bftheta}_{t-1}^{\mathrm{ML}}) \Big\}^{-1} \Big],
\end{equation}
where $\operatorname{tr}(\cdot)$ denotes the trace, and $\Lambda_{t-1}(\hat{\bftheta}_{t-1}^{\mathrm{ML}})$ and $\mathcal{I}_j(\hat{\boldsymbol{\theta}}_{t-1}^{\mathrm{ML}})$ are defined in Equation~\eqref{e:cumulative_fisher}. This A-optimal item-selection rule was also introduced in \citeA{van1999multidimensional} and \citeA{mulder2009multidimensional}.

\paragraph{Item-Selection Rules Distinguishing between Intentional and Nuisance Abilities.}
To incorporate intentional and nuisance abilities as in \citeA{mulder2009multidimensional}, we order the abilities so that nuisance dimensions are placed after intentional dimensions. Specifically, for a $d \times d$ positive definite symmetric matrix $\Sigma$, we write
\begin{equation} \label{e:decomp}
\Sigma =
\begin{bmatrix}
 \Sigma_{II} & \Sigma_{IN}
 \\ \Sigma_{NI} & \Sigma_{NN}
\end{bmatrix},
\end{equation}
where $\Sigma_{II}$ is the $d_I \times d_I$ block corresponding to the intentional dimensions, and $\Sigma_{NN}$ is the $d_N \times d_N$ block corresponding to the nuisance dimensions. Later in this subsection, we replace $\Sigma$ by an approximation to the covariance matrix of $\hat{\bftheta}^{\text{ML}}_{t}$.
We represent different item-selection rules using a unified criterion function $\Phi_{q,\delta}: \mathcal{S}_d^{+} \to \mathbb{R}$, where $\mathcal{S}_d^{+}$ denotes the set of $d \times d$ positive definite symmetric matrices. 
The mapping $\Phi_{q,\delta}$ is indexed by a criterion parameter $q \ge 0$ and a pre-specified weight $\delta \in (0,1]$. It is defined as
\begin{equation} \label{e:selection}
\Phi_{q,\delta}(\Sigma) =
\begin{cases}
  (1 - \delta) \log \operatorname{det}( \Sigma_{II} ) + \delta \log \operatorname{det}( \Sigma ) & q = 0,
  \\ (1 - \delta) \{ \operatorname{tr}( (\Sigma_{II})^{q} )\}^{1/q} + \delta \{ \operatorname{tr}(\Sigma^q) \}^{1/q}  & q > 0.
\end{cases}
\end{equation}
We then consider the following item-selection rule:
\begin{equation} \label{e:phi-q-delta-optimal}
  j_t \in \underset{j \in \mathcal{J} \setminus \{j_1, \ldots, j_{t-1}\}}{\arg \min} \Phi_{q, \delta} \Big[ \Big\{ \Lambda_{t-1}( \hat{\bftheta}_{t-1}^{\mathrm{ML}}) + \mathcal{I}_j(\hat{\bftheta}_{t-1}^{\mathrm{ML}}) \Big\}^{-1} \Big],
\end{equation}
where $\Lambda_{t-1}( \hat{\bftheta}_{t-1}^{\mathrm{ML}} )$ and $\mathcal{I}_j(\hat{\boldsymbol{\theta}}_{t-1}^{\mathrm{ML}})$ are defined in Equation~\eqref{e:cumulative_fisher}. 

We comment on the above item-selection rule.

First, \(\hat{\Sigma} =\big\{\Lambda_{t-1}(\hat{\bftheta}_{t-1}^{\mathrm{ML}})+\mathcal{I}_j(\hat{\bftheta}_{t-1}^{\mathrm{ML}})\big\}^{-1}\) serves as a plug-in approximation to the asymptotic covariance matrix of the ML estimator. In classical i.i.d. settings, the inverse observed Fisher information is routinely used for this purpose. In our adaptive setting, however, this approximation does not automatically hold because the data are collected sequentially and depend on past observations through the selection rule. Our theoretical results establish that \(\hat{\Sigma}\) approximates the asymptotic covariance of the ML estimator, similar to the classical case.

Second, the criterion \(\Phi_{q,\delta}\) depends on the partitioned matrix 
\begin{equation*}
\hat{\Sigma} =
\begin{bmatrix}
 \hat{\Sigma}_{II} & \hat{\Sigma}_{IN}
 \\ \hat{\Sigma}_{NI} & \hat{\Sigma}_{NN}
\end{bmatrix}.
\end{equation*}
Specifically, $\Phi_{q,\delta}(\hat{\Sigma}) = (1-\delta)\phi_q(\hat{\Sigma}_{II})+\delta\phi_q(\hat{\Sigma})$, where $\phi_q(\Sigma) = \{\operatorname{tr}(\Sigma^q)\}^{1/q}$ for $q>0$ and $\phi_0(\Sigma) = \log \operatorname{det}(\Sigma)$. The class $\phi_q$ is well known in the optimal design literature \cite{pukelsheim2006optimal}. It includes common criteria as special cases: $q=1$ corresponds to the A-optimality selection rule in Equation~\eqref{e:a-optimal}, whereas $q=0$ corresponds to D-optimality in Equation~\eqref{e:d-optimal}. The criterion $\Phi_{q,\delta}$ generalizes these choices by balancing a target-block criterion $\phi_q(\hat{\Sigma}_{II})$ with an overall criterion $\phi_q(\hat{\Sigma})$ for the full ability vector. The first term emphasizes precision for the intentional abilities, whereas the second term prevents the nuisance dimensions from being ignored. The weight $\delta \in (0,1]$ controls this balance. When $\delta = 1$, intentional and nuisance abilities are treated equally. As $\delta$ gets closer to $0$, the criterion puts more emphasis on the intentional abilities while retaining a term that controls uncertainty in the full ability vector.

Several special cases are worth noting. First, the case $q=1$ with fixed $\delta>0$ gives the weighted A-optimal rule analyzed in our main optimality result.
\paragraph{Weighted A-optimal item-selection rule} For the item pool $\mathcal{J}$ and a fixed weight $\delta \in (0,1]$, the weighted A-optimal item-selection rule chooses the next item $j_t$ as follows.
\begin{equation} \label{e:wa-optimal}
  j_t \in \underset{j \in \mathcal{J} \setminus \{j_1, \ldots, j_{t-1}\}}{\arg \min} \Phi_{1, \delta} \Big[ \Big\{ \Lambda_{t-1}( \hat{\bftheta}_{t-1}^{\mathrm{ML}}) + \mathcal{I}_j(\hat{\bftheta}_{t-1}^{\mathrm{ML}}) \Big\}^{-1} \Big].
\end{equation}
Another special case occurs at the endpoint $\delta \to 0$. When $q=1$ or $q=0$, respectively, the rule reduces to the $\mathrm{A}_s$-optimality or $\mathrm{D}_s$-optimality criterion.
\paragraph{$\mathrm{A}_s$- and $\mathrm{D}_s$-optimal item-selection rules} 
In the presence of intentional and nuisance abilities, \citeA{mulder2009multidimensional} introduced $\operatorname{A_s}$- and $\operatorname{D_s}$-optimality, which prioritize estimation accuracy for the intentional abilities. For the item pool $\mathcal{J}$, the $\operatorname{A_s}$-optimal item-selection rule selects the item $j_t$ as follows.
\begin{equation} \label{e:as-optimal}
 j_t \in \underset{j \in \mathcal{J} \setminus \{j_1, \ldots, j_{t-1}\}}{\arg \min} \operatorname{tr} \Big[ \Big( \{ \Lambda_{t-1}( \hat{\bftheta}_{t-1}^{\mathrm{ML}} ) + \mathcal{I}_j(\hat{\bftheta}_{t-1}^{\mathrm{ML}})\}^{-1} \Big)_{II} \Big],
\end{equation}
where $\Big( \{ \Lambda_{t-1}( \hat{\bftheta}_{t-1}^{\mathrm{ML}} ) + \mathcal{I}_j(\hat{\bftheta}_{t-1}^{\mathrm{ML}}) \}^{-1} \Big)_{II}$ is the first $d_I \times d_I$ block corresponding to the intentional abilities, with $\Lambda_{t-1}(\hat{\bftheta}_{t-1}^{\mathrm{ML}})$ and $\mathcal{I}_j(\hat{\bftheta}_{t-1}^{\mathrm{ML}})$ defined in Equation~\eqref{e:cumulative_fisher}. Similarly, the $\operatorname{D_s}$-optimal item-selection rule selects the item $j_t$ as follows.
\begin{equation} \label{e:ds-optimal}
 j_t \in \underset{j \in \mathcal{J} \setminus \{j_1, \ldots, j_{t-1}\}}{\arg \min} \operatorname{det} \Big[ \Big( \{ \Lambda_{t-1}( \hat{\bftheta}_{t-1}^{\mathrm{ML}} ) + \mathcal{I}_j(\hat{\bftheta}_{t-1}^{\mathrm{ML}})\}^{-1} \Big)_{II} \Big].
\end{equation}

\begin{remark}
Algebraically, the weighted A-optimal item-selection rule in Equation~\eqref{e:wa-optimal} coincides with the $\operatorname{A_K}$-optimal item-selection rule of \citeA{sagnol2015computing}. The standard use of $\operatorname{A_K}$-optimality focuses on lower-dimensional linear combinations of the abilities, whereas our formulation prioritizes intentional abilities while accounting for nuisance abilities. We point out that the weighted A-optimal item selection is relatively new in the psychometric literature: it extends the $\operatorname{A_s}$-optimal item-selection rule of \citeA{mulder2009multidimensional} to more general settings through the choice of $\delta$.
\end{remark}

\subsection{Evaluation Metrics} \label{sub:evaluation}
We write the true ability vector as $\bftheta^{\ast} = ((\bftheta^{\ast}_I)^{\top}, (\bftheta^{\ast}_N)^{\top})^{\top}$, where $\bftheta^{\ast}_I$ and $\bftheta^{\ast}_N$ denote the true intentional and nuisance ability components, respectively. For a test of length $T$, we write the estimator as $\hat{\bftheta}_T = (\hat{\bftheta}_{I,T}^{\top}, \hat{\bftheta}_{N,T}^{\top})^{\top}$. To evaluate the performance of different item-selection rules and estimators, we consider the mean squared error (MSE) of the intentional and nuisance dimensions. The MSE for the intentional dimensions is defined as
\begin{equation*}
\operatorname{MSE}_I(\hat{\bftheta}_T)=\mathbb{E}[\|\hat{\bftheta}_{I,T}-\bftheta_I^\ast\|^2],
\end{equation*}
and the MSE for the nuisance dimensions is defined as 
\begin{equation*}
\operatorname{MSE}_N(\hat{\bftheta}_T) = \mathbb{E}[\|\hat{\bftheta}_{N,T}-\bftheta_N^\ast\|^2],
\end{equation*}
where the expectation is taken with respect to the joint distribution of the responses $\{Y_t\}_{t=1}^T$, and $\|\cdot\|$ denotes the Euclidean norm. We further define the weighted MSE (WMSE) as 
\begin{equation*}
\operatorname{WMSE}_{\delta}(\hat{\bftheta}_T)=\operatorname{MSE}_I(\hat{\bftheta}_T)+\delta \operatorname{MSE}_N(\hat{\bftheta}_T).
\end{equation*}
Here, $\delta \in (0,1]$ controls the relative importance assigned to nuisance dimensions. When $\delta=1$, $\operatorname{WMSE}_{\delta}(\hat{\bftheta}_T)$ reduces to the overall MSE. When $\delta$ is close to $0$, $\operatorname{WMSE}_{\delta}(\hat{\bftheta}_T)$ is close to the MSE for the intentional dimensions. A smaller WMSE indicates a more accurate estimator under this weighted criterion. Our focus in this work is rigorous theoretical comparison of item-selection rules and estimators beyond specific simulation settings. To this end, we consider the asymptotic behavior of the WMSE as the test length $T\to\infty$.

We say that an estimator paired with a selection rule $(\hat{\bftheta}_T, \{j_1,\cdots, j_T\})$ is asymptotically optimal if for any other estimator and selection rule $(\tilde{\bftheta}_T, \{\tilde{j}_1,\cdots, \tilde{j}_T\})$ and any fixed $\delta \in (0,1]$, 
\[\limsup_{T\to\infty}\frac{\operatorname{WMSE}_{\delta}(\widehat{\boldsymbol{\theta}}_T)}{\operatorname{WMSE}_{\delta}(\widetilde{\boldsymbol{\theta}}_T)}\leq 1.\]
In other words, no other estimator and item-selection rule can achieve a strictly smaller asymptotic WMSE as $T\to\infty$. This is a strong mathematical notion, because it compares a selection-estimation procedure against a broad class of possible estimator--selection-rule pairs, rather than only against a finite set of pre-specified methods in a simulation study. For a rigorous mathematical statement, this broad comparison is made within a reasonable class of estimator--selection-rule pairs, which is specified precisely in Section \ref{sec:theory}. Establishing such asymptotic optimality provides theoretical justification complementing prior empirical studies, which typically demonstrate the superiority of certain methods only under specific simulation settings.

\section{Weighted-Target Optimality and Asymptotic Normality}
\label{sec:theory}
In this section, we present the main theoretical results under an asymptotic framework in which the test length grows. Specifically, we consider a sequence of tests indexed by $r\to\infty$, with test length $T_r$ nondecreasing in $r$ and satisfying $T_r\to\infty$. The item pool may also depend on $r$; we denote it by $\mathcal{J}_r$ and allow $|\mathcal{J}_r|$ to grow with $T_r$. This growing-pool formulation is used for asymptotic analysis and reflects the operational reality that test length is typically small relative to the available item pool.

We consider two item-pool regimes. Regime R1 is an item-type reuse regime, in which the pool is represented through finitely many calibrated item-parameter values, each corresponding to an item type. Here, reuse refers to selecting an item type more than once while administering distinct operational items. Regime R2 is a unique-item-parameter regime, in which each operational item is treated as having its own calibrated item-parameter values rather than belonging to a repeated type.
\begin{assumption}[R1: Item-type reuse regime] \label{ass:r1}
Suppose there are $M \in \mathbb{Z}_+$ item types, where $\mathbb{Z}_+$ is the set of positive integers.  
Under R1, an item pool with $M$ calibrated types and $J_r$ operational items is represented as a multiset (i.e., a set that may contain repeating elements)
\begin{equation} \label{e:r1}
\mathcal{J}_{r} 
= \{\,\underbrace{1,\ldots,1}_{J_r/M\text{ times}},\
     \underbrace{2,\ldots,2}_{J_r/M\text{ times}},\ \ldots,\
     \underbrace{M,\ldots,M}_{J_r/M\text{ times}}\,\}, 
\end{equation}
where $J_r$ is a positive integer divisible by $M$ such that $J_r/M\geq T_r$ and $\lim_{r\to\infty}J_r=\infty$.
Letting $r \to \infty$ gives the limiting item pool $\mathcal{J}_{\infty} = \{1, 1, \cdots, 2, 2, \cdots, M, M, \cdots\}$. 
\end{assumption}
We elaborate on R1 as follows.
In practice, a calibrated item pool may contain groups of items with very similar discrimination and difficulty parameters. R1 represents each such group by an item type: items of the same type share calibrated parameters in the asymptotic analysis, but they remain operationally distinct items. The condition $|\mathcal{J}_r|/M \ge T_r$ means that each type contains enough distinct items to support a full test. Thus, even if the adaptive rule selects the same type at several steps, the test can administer different operational items and will not run out of items of that type. This regime can therefore be viewed as an idealized version of using a large calibrated pool whose items have been grouped into parameter-similar clusters. From a theoretical perspective, R1 corresponds to the ``action reuse'' framework adopted in \citeA{li2025globally} and is related to fixed-type assumptions used in theoretical analyses of CAT and sequential design problems \cite{bartroff2008modern,wang2015sequential}. Such analyses often assume that the empirical selection frequencies of item types converge to fixed limits (e.g., \citeNP{bartroff2008modern}). However, for commonly used adaptive item-selection rules, these limiting empirical selection frequencies are typically assumed rather than derived from the rule itself. Our analysis fills this gap by deriving the limiting empirical selection frequencies generated by the weighted A-optimal rule. To the best of our knowledge, this is the first asymptotic derivation of such limiting frequencies for an information-based MCAT item-selection rule.

\begin{assumption}[R2: Unique-item-parameter regime]
\label{ass:r2}
Suppose the pool contains $J_r \in \mathbb{N}$ calibrated items, where $J_r$ depends on the auxiliary index $r$, satisfies $\lim_{r \to \infty} J_r = \infty$, and is nondecreasing in $r$. For the feasibility of the test, we further assume $J_r \geq T_r$. Then, the item pool for the test of length $T_r$ is given by
\begin{equation} \label{e:r2}
\mathcal{J}_{r} = \{1, \cdots, J_r\},
\end{equation}
Letting $r \to \infty$ gives the limiting item pool $\mathcal{J}_{\infty} = \{1,2,3, \cdots \}$.
\end{assumption}
The feasibility condition $J_r \geq T_r$ ensures that the pool contains enough items for a test of length $T_r$. Unlike R1, this regime does not impose a repeated-type structure on the item parameters; each operational item is treated as its own calibrated item.

Next, we state the regularity conditions used in our theoretical analysis.

\begin{assumption} \label{ass:compact} 
The parameter space $\boldsymbol{\Theta}$ is a non-empty compact and convex subset of $\mathbb{R}^d$. The true ability $\bftheta^{\ast}$ is an interior point of $\boldsymbol{\Theta}$.
\end{assumption}

\begin{assumption} \label{ass:dimension}
Under Assumption \ref{ass:r1} (R1), 
$$
\operatorname{dim}( \operatorname{span} \{ \bfalpha_j : j \in \{1,\cdots,M\} \}) = d,
$$
where $\bfalpha_j$ is defined in \eqref{eq:model}. 
\end{assumption}

\begin{assumption} \label{ass:information1} (Eigen growth) There exist $c>0$ and $r_0 > 0$ such that $\underline\lambda_{T_r} \geq c$ almost surely for all $r \ge r_0$. Here, $\underline\lambda_{T_r} = \lambda_{\min}(\frac{1}{T_r} \Lambda_{T_r}( \hat{\bftheta}_{T_{r}}^{\mathrm{ML}} ))$, where $\lambda_{\min}(\Sigma)$ denotes the smallest eigenvalue of the matrix $\Sigma$, and $\Lambda_{T_r}(  \hat{\bftheta}_{T_r}^{\mathrm{ML}}  )$ is defined in \eqref{e:cumulative_fisher}. 
\end{assumption}

\begin{assumption} \label{ass:information2} (Stabilization) $\lim_{r \to \infty} \frac{1}{T_r} \Lambda_{T_r}( \hat{\bftheta}_{T_r}^{\mathrm{ML}} ) = \mathcal I_{\infty}$ almost surely for some deterministic positive definite matrix $\mathcal I_\infty$.
\end{assumption}

\begin{assumption} \label{ass:information3} (Uniform boundedness) For the item pool $\mathcal{J}_{\infty}$ defined under either Assumption \ref{ass:r1} or \ref{ass:r2}, $\sup_{j \in \mathcal{J}_{\infty}} \| \bfalpha_j \|_2 < \infty$ and $\sup_{j \in \mathcal{J}_{\infty}} |b_j| < \infty$. 
\end{assumption}

The following theorems use these assumptions in different combinations. We begin with the asymptotic optimality result for the weighted A-optimal selection rule paired with the ML estimator.

\begin{theorem}[Weighted-target optimality]
\label{thm:optimality}
Suppose Assumptions \ref{ass:r1} (R1), \ref{ass:compact}, and \ref{ass:dimension} hold. Then, the following statements hold.
\begin{enumerate}
\item Define an empirical frequency vector $\bar{\bfpi}_{T_r} = \Big( \frac{1}{T_r} \sum_{i=1}^{T_r} \mathbf{1}_{(j_i = 1)}, \ldots, \frac{1}{T_r} \sum_{i=1}^{T_r} \mathbf{1}_{(j_i = M)} \Big)^{\top}$, where $j_i$ is chosen using the item-selection criterion \eqref{e:wa-optimal}. For $q = 1$ and $\delta \in (0,1]$, assume that $\Phi_{1,\delta}\Big[ \Big\{ \sum_{a=1}^{M} \pi_a \mathcal{I}_a(\bftheta^{\ast}) \Big\}^{-1}  \Big]$ has a unique minimizer $\bfpi^{\ast} = (\pi_1^{\ast}, \ldots, \pi_M^{\ast}) \in \mathcal{S}^{M-1}$, where $\mathcal{S}^{M-1}=\{\bfpi \in \mathbb{R}^M: \pi_a\geq 0, \sum_{a=1}^M \pi_a=1\}$ is the standard simplex in $\mathbb{R}^M$. Then, 
\begin{equation} \label{eq:freqlimit}
\lim_{r \to \infty} \bar{\bfpi}_{T_r} = \bfpi^{\ast} \quad \text{(almost surely).}
\end{equation}
\item Let $\hat{\bftheta}_{T_r}$ be any unbiased estimator (i.e., $\mathbb{E}_{\bftheta} [\hat{\bftheta}_{T_r}] = \bftheta$ for all $\bftheta \in \bfTheta$) under an arbitrary item-selection rule. Then, 
\begin{equation} \label{eq:lowerbound}
\liminf_{r\to\infty} {T_r} \cdot \operatorname{WMSE}_\delta(\hat{\bftheta}_{T_r}) \geq \min_{\bfpi=(\pi_1,\cdots,\pi_M)\in \mathcal{S}^{M-1}} \Phi_{1,\delta} \Big[  \Big\{ \sum_{a=1}^M \pi_a \mathcal{I}_a(\bftheta^{\ast}) \Big\}^{-1} \Big],
\end{equation}
where $\mathcal{I}_a(\bftheta^{\ast})$ is the Fisher information matrix defined in \eqref{e:cumulative_fisher}. 
\item The weighted A-optimal item selection rule defined in \eqref{e:wa-optimal} paired with the ML estimator $\hat{\bftheta}_{T_r}^{\mathrm{ML}}$ achieves the lower bound in \eqref{eq:lowerbound}. That is,
\begin{equation} \label{eq:upperbound}
\lim_{r\to\infty}\{ T_r \cdot \operatorname{WMSE}_\delta(\hat{\bftheta}_{T_r}^{\mathrm{ML}})\} = \min_{\bfpi=(\pi_1,\cdots,\pi_M)\in \mathcal{S}^{M-1}} \Phi_{1,\delta} \Big[  \Big\{ \sum_{a=1}^M \pi_a \mathcal{I}_a(\bftheta^{\ast}) \Big\}^{-1} \Big].
\end{equation}
\end{enumerate}
\end{theorem}
\begin{proof}
See Appendix \ref{appendix:proof}.
\end{proof}
Hence, the weighted A-optimal item selection rule paired with the ML estimator is asymptotically optimal in the sense that for any unbiased estimator paired with an arbitrary selection rule $(\tilde{\bftheta}_{T_r}, \{\tilde{j}_1,\cdots, \tilde{j}_{T_r}\})$, we have
\begin{equation*}
\limsup_{r\to\infty}\frac{\operatorname{WMSE}_{\delta}(\hat{\bftheta}_{T_r}^{\mathrm{ML}})}{\operatorname{WMSE}_{\delta}(\tilde{\bftheta}_{T_r})}\leq 1.
\end{equation*}

\begin{remark}[Unbiasedness assumption]
\label{rem:unbiased}
The unbiasedness assumption in Theorem \ref{thm:optimality} is a technical condition that simplifies the proof. It extends the classical Cramer-Rao lower bound, originally established for independent data, to the adaptive setting. We note that the ML estimator is usually biased in finite samples. Thus, Theorem \ref{thm:optimality} does not directly state that the ML estimator has the minimal risk within a comparison class that includes the ML estimator itself. This is analogous to the classical asymptotic efficiency theory for ML estimation with i.i.d. data, where the ML estimator is shown to match the asymptotic variance of the Cramer-Rao lower bound for unbiased estimators, even though the ML estimator itself is not unbiased. It is possible to show that the ML estimator has the asymptotically smallest risk among a reasonable class of estimators and selection rules, but the theorem statement becomes much more technical with little additional practical insight. We therefore present the unbiased version for clarity. Interested readers may refer to \citeA{li2025globally} for a theorem statement addressing this point in a similar active-estimation setting.
\end{remark}

The next theorem gives sufficient conditions for consistency and asymptotic normality of the ML estimator under a general item-selection rule. This result provides the theoretical basis for standard information-based confidence intervals in MCAT practice.

\begin{theorem}[Asymptotic normality of ML estimator]
\label{thm:normality}
Suppose either Assumption \ref{ass:r1} (R1) or Assumption \ref{ass:r2} (R2) holds. For the item sequence generated by the selection rule, suppose %
Assumption \ref{ass:compact} and Assumptions \ref{ass:information1}--\ref{ass:information3} hold. 
Then, we have
\begin{enumerate}
\item (Consistency) $\lim_{r\to\infty} \hat{\bftheta}^{\mathrm{ML}}_{T_r}=\bftheta^{\ast}$ almost surely. 
\item (Asymptotic normality) As $r \to \infty$,
\begin{equation}
\sqrt{T_r}(\hat{\bftheta}^{\mathrm{ML}}_{T_r}- \bftheta^{\ast})\overset{d}{\to} \mathcal{N}(0, \mathcal{I}_\infty^{-1}).
\end{equation}
\item (Asymptotic $\operatorname{WMSE}_{\delta}$) For each $\delta \in (0,1]$, under the item-selection criterion \eqref{e:wa-optimal},
$$
\lim_{r\to\infty}\{ T_r\cdot \operatorname{WMSE}_\delta(\hat{\bftheta}_{T_r}^{\mathrm{ML}})\} = \Phi_{1,\delta} \Big[ \mathcal{I}_{\infty}^{-1} \Big].
$$
\end{enumerate}
\end{theorem}
\begin{proof}
See Appendix \ref{appendix:proof}.
\end{proof}

\begin{remark}[Verification under the weighted A-optimal rule]
\label{rem:r1-normality-verification}
Under the item-type reuse regime in Assumption \ref{ass:r1} (R1), when items are selected by the weighted A-optimal rule in \eqref{e:wa-optimal}, Assumptions \ref{ass:compact} and \ref{ass:dimension} imply the requirements of Theorem \ref{thm:normality}. Thus, the assumptions used for the R1 optimality result also give the conditions needed for asymptotic normality of the ML estimator. See Lemma \ref{lemma:lemma2} in Appendix \ref{appendix:proof} for more details.
\end{remark}

\begin{remark}[Variable-length tests]
\label{rem:variable}
A fixed test length $T_r$ is assumed in Theorems~\ref{thm:optimality} and \ref{thm:normality}. In practice, CAT may use variable test lengths to reduce the test length and help reduce test-taker fatigue. That is, the test terminates once the estimation accuracy is sufficient, making the total length a random variable. The asymptotic normality results can be extended to such variable-length tests. Specifically, for a user-defined $\operatorname{MSE}$ threshold $\tau$ and $\delta \in [0,1]$, we define the random test length
\begin{equation*}
T_{\delta, \tau} = \inf \{ t \geq 1: \widehat{\operatorname{WMSE}}_{\delta}(\hat{\bftheta}_t) \leq \tau \},
\end{equation*}
where the estimated $\widehat{\operatorname{WMSE}}$ at step $t$ is given as 
\begin{equation*}
\widehat{\operatorname{WMSE}}_{\delta}(\hat{\bftheta}_t) = \frac{1}{t} \Big\{ \operatorname{tr} \Big( (\hat{\mathcal{I}}_t^{-1})_{II} \Big) + \delta \operatorname{tr}\Big( (\hat{\mathcal{I}}_t^{-1})_{NN} \Big) \Big\}.
\end{equation*}
Here, $(\hat{\mathcal I}_t^{-1})_{II}$ and $(\hat{\mathcal I}_t^{-1})_{NN}$ denote the block submatrices induced by the decomposition in \eqref{e:decomp}. Such variable-length tests can be analyzed by combining the results in the current study with the proof of Theorem 10 in \citeA{li2025globally}.
\end{remark}

\begin{remark}[Valid statistical inference]
The asymptotic normality statement in Theorem \ref{thm:normality} provides a rigorous justification for constructing confidence intervals and conducting hypothesis tests based on the asymptotic distribution of the ML estimator. In practice, the asymptotic covariance matrix $\mathcal{I}_{\infty}^{-1}$ is replaced by its finite-sample counterpart $\hat{\mathcal{I}}_T^{-1}$, where $\hat{\mathcal{I}}_T$ denotes the scaled cumulative Fisher information at test length $T$. The same asymptotic normal approximation remains valid with this plug-in estimator under the stated conditions.
This asymptotic normality result is not specific to the A-optimal rule. It also holds for other item-selection procedures, provided that the ML estimator is employed.
\end{remark}

\section{Simulation Study}
\label{sec:simulation}
In this section, we use simulation studies to examine the theoretical results in Section \ref{sec:theory} under Assumptions \ref{ass:r1} (R1) and \ref{ass:r2} (R2). The results in Section \ref{sec:theory} are asymptotic and rely on a large test length $T$. To assess how the large-$T$ theory behaves in practically relevant finite-test settings, we consider test lengths $T=30$ and $T=100$ with a two-dimensional ability vector $\boldsymbol{\theta}^{\ast} \in \mathbb{R}^2$. We investigate finite-sample behavior related to the first and third statements of Theorem \ref{thm:optimality}, and we also examine selection stabilization under Assumption \ref{ass:r2} (R2). We first describe the simulation design used throughout this section.

\subsection{Simulation Design} \label{sub:design}
We follow the nine cases in \citeA{mulder2009multidimensional}, where the true ability parameter is $\boldsymbol{\theta}^{\ast} = (\theta_1^{\ast}, \theta_2^{\ast})$ with $\theta_i^{\ast} \in  \{-1, 0, 1\}$ for $i = 1,2$. We treat $\theta_1^{\ast}$ as the intentional ability and $\theta_2^{\ast}$ as the nuisance ability. Under the M2PL model with $d = 2$, each item $j$ has two discrimination parameters, $\alpha_{1,j}$ and $\alpha_{2,j}$, generated independently from a folded normal distribution obtained by taking the absolute value of an $N(1,0.3)$ random variable, and a single difficulty parameter $b_j$ generated from $N(0,3)$. 
Under Assumption \ref{ass:r1} (R1), we fix the test length to be either $T = 30$ or $T = 100$, with the item pool as a multiset 
\begin{equation*} 
\mathcal{J} = \{\,\underbrace{1,\ldots,1}_{100 \text{ times}},\
\underbrace{2,\ldots,2}_{100 \text{ times}},\ \ldots,\
\underbrace{30,\ldots,30}_{100 \text{ times}}\,\}.
\end{equation*}
Under Assumption \ref{ass:r2} (R2), we fix $T = 30$ or $T = 100$ and take the item pool to be 
\begin{equation*} 
\mathcal{J} = \{1,2,\dots,200\}.  
\end{equation*} 
For notational simplicity, we omit the subscript $r$ in the item set specified in Assumptions \ref{ass:r1} and \ref{ass:r2}. All reported simulation results are based on $500$ adaptive test administrations.

\subsection{Selection Stabilization} \label{sub:stabilization}
In this subsection, we investigate stabilization of the item-selection process under Assumptions \ref{ass:r1} (R1) and \ref{ass:r2} (R2), and relate it to the first statement of Theorem \ref{thm:optimality}. Under Assumption \ref{ass:r1} (R1), Theorem \ref{thm:optimality} states that the empirical frequency vector of selected item types converges to an optimal limiting frequency. This means that, as the test length $T$ becomes large, the selection frequency for each item type stabilizes. We examine this result under finite test lengths. Because Theorem~\ref{thm:optimality} is established under R1 and is not directly applicable to R2, we use an additional metric to study stabilization under R2.

Under Assumption \ref{ass:r1} (R1), for each item type $j \in \mathcal{J}$ (i.e., $j \in \{1,\ldots,30\}$) and test length $T$, we compute the empirical selection frequency
\begin{equation*}
\frac{1}{T} \sum_{s=1}^{T} \mathbf{1}_{(j = j_s)},
\end{equation*}
where $j_s$ denotes the index of the selected item at step $s$. By tracking these empirical frequencies as $T$ increases, we assess whether the selection proportions for each item type begin to stabilize. To check this, we set a longer test length $T = 100$ for a single adaptive test administration and compare the empirical frequencies evaluated at $T = 30$ and $T = 100$. Since our interest is whether stabilization starts near $T = 30$, we examine how close these two frequencies are. Due to page limits, we only report the case under $\bftheta^{\ast} = (-1,0)$ and $\delta = 1$ in this subsection and provide the results for other cases in Appendix \ref{appendix:simulation}. The empirical selection frequencies of item types under Assumption \ref{ass:r1} (R1) are shown in Figure \ref{fig:reuse1}.

From Figure \ref{fig:reuse1}, only $5$ out of $30$ item types (item types $5$, $12$, $26$, $28$, and $29$) are selected. Comparing the empirical selection frequencies at $T = 30$ and $T = 100$, the frequencies for item types $5$, $12$, $28$, and $29$ are close to their longer-horizon behavior values at $T = 100$, while the frequency for item type $26$ continues to decrease as $T$ increases. Thus, the selection frequencies are not fully stabilized at $T = 30$, but most of the frequencies are already close to their long-run pattern.

Under Assumption \ref{ass:r2} (R2), each operational item has its own calibrated item-parameter values, so empirical frequencies of individual item indices are not directly comparable to the repeated-type frequencies in R1. Instead, we consider stabilization in terms of discrimination parameters $\alpha_{1,j}$ and $\alpha_{2,j}$ and the difficulty parameter $b_j$ for each item $j$. The idea is that, even without repeated item types, the selected items may concentrate in regions of the item-parameter space. To check this, we fix several cutoff values $c$, and for each test length $T$, we compute the empirical distributions 
\begin{equation} \label{e:emp_dist1}
\frac{1}{T} \sum_{t=1}^{T} \mathbf{1}_{(\alpha_{i, j_t} \leq c)}, \quad i = 1,2
\end{equation} 
for the discrimination parameters, and 
\begin{equation} \label{e:emp_dist2}
\frac{1}{T} \sum_{t=1}^{T} \mathbf{1}_{(b_{j_t} \leq c)},
\end{equation} 
for the difficulty parameter. These quantities describe how the empirical distribution of selected items over the parameter space evolves with $T$. By tracking them for several cutoff values $c$ as $T$ increases, we can evaluate whether the selection stabilizes in terms of discrimination and difficulty. We fix $T = 100$ for a single adaptive test administration and compare the empirical distributions at $T = 30$ and $T = 100$. Figures \ref{fig:prop1} and \ref{fig:prop2} display the results for two discrimination parameters using cutoffs $(0.5, 0.7, 1, 1.5, 2, 2.2)$ (from bottom to top), and Figure \ref{fig:prop3} shows the results for the difficulty parameter using cutoffs $(-2, -1.5, -1, 0, 1, 1.5, 2)$.

Figures \ref{fig:prop1}--\ref{fig:prop3} show that at $T = 100$ the empirical distributions are stable for all cutoffs. Moreover, for most cutoff values, the empirical proportions at $T = 30$ are already close to those at $T = 100$ for all three parameters. In our simulations, this suggests that under Assumption \ref{ass:r2} (R2), the selection effectively concentrates on regions of the item pool with similar discrimination and difficulty values, and that selection behavior at $T = 30$ is already close to the longer-test pattern.

\subsection{Comparison of $\mathrm{WMSE}_{\delta}$} \label{sub:compare}
In this comparison, we evaluate $\widehat{\operatorname{WMSE}}_{\delta}$ across different selection criteria. By the third statement of Theorem \ref{thm:optimality}, under Assumption \ref{ass:r1} (R1), the weighted A-optimal rule is asymptotically optimal for the corresponding weighted MSE criterion. In the simulations, we investigate how weighted A-optimality behaves relative to other selection criteria under both Assumption \ref{ass:r1} (R1) and Assumption \ref{ass:r2} (R2) at the practical finite test length $T = 30$. For each assumption and fixed test length $T$, we compute the estimated mean squared error for the intentional and nuisance abilities as follows.
\begin{equation} \label{e:hat mse}
\widehat{\operatorname{MSE}}_I(\hat{\theta}_{1,T}^{\mathrm{ML}}) = \frac{1}{500} \sum_{k = 1}^{500} (\hat{\theta}_{1,T}^{\mathrm{ML}, k} - \theta_1^{\ast})^2,
\quad
\widehat{\operatorname{MSE}}_N(\hat{\theta}_{2,T}^{\mathrm{ML}}) = \frac{1}{500} \sum_{k = 1}^{500} (\hat{\theta}_{2,T}^{\mathrm{ML},k} - \theta_2^{\ast})^2,
\end{equation}
where $\hat{\theta}_{1,T}^{\mathrm{ML},k}$ and $\hat{\theta}_{2,T}^{\mathrm{ML},k}$ represent the ML estimators from the $k$-th adaptive test 
administration using the test data of length $T$. Then, for $\delta \in (0,1]$ and test length $T$, the estimated weighted MSE is computed as
\begin{equation} \label{e:hat wmse}
\widehat{\operatorname{WMSE}}_{\delta}(\hat{\bftheta}_{T}^{\mathrm{ML}}) = \widehat{\operatorname{MSE}}_I(\hat{\theta}_{1,T}^{\mathrm{ML}}) + \delta \cdot \widehat{\operatorname{MSE}}_N(\hat{\theta}_{2,T}^{\mathrm{ML}}).
\end{equation}
We report the results for A-optimality (A), D-optimality (D), random selection (R), weighted A-optimality with $\delta = 0.1$ ($\mathrm{WA}_{0.1}$), weighted A-optimality with $\delta= 0.001$ ($\mathrm{WA}_{0.001}$), and $\mathrm{D}_s$-optimality ($\mathrm{D}_s$). Tables \ref{t:r1 mse} and \ref{t:r1 aggre} present the results under Assumption \ref{ass:r1}, and Tables \ref{t:r2 mse} and \ref{t:r2 aggre} present the results under Assumption \ref{ass:r2}.

We observe a very similar pattern across the selection criteria under both Assumption \ref{ass:r1} (R1) and Assumption \ref{ass:r2} (R2). First, all information-based selection criteria perform better than random selection. For a fixed value of $\delta$, Theorem \ref{thm:optimality} shows that, under Assumption \ref{ass:r1} (R1), the weighted A-optimal rule with that $\delta$ (i.e., $\mathrm{WA}_{\delta}$) is asymptotically optimal for $\operatorname{WMSE}_{\delta}$. This theoretical optimality is reflected in the finite-sample simulations: $\widehat{\mathrm{WMSE}}_1$, $\widehat{\mathrm{WMSE}}_{0.1}$, and $\widehat{\mathrm{WMSE}}_{0.001}$ are minimized by A-optimality, $\mathrm{WA}_{0.1}$, and $\mathrm{WA}_{0.001}$, respectively. This suggests that, in the reported settings, the asymptotic optimality result is informative for the finite test length $T = 30$. Also, decreasing $\delta$ toward $0$ is effective when the main goal is accurate estimation of the intentional ability. Specifically, under both Assumption \ref{ass:r1} and Assumption \ref{ass:r2}, $\widehat{\mathrm{MSE}}_{I}(\hat{\boldsymbol{\theta}}_{1,30})$ for $\mathrm{WA}_{0.001}$ is similar to that for $\mathrm{D}_s$-optimality.

\section{Discussion and Future Work} \label{sec:discuss} 
This paper provides a theoretical and empirical investigation of Fisher--information--based item-selection rules for MCAT in the presence of intentional and nuisance abilities. Our main theoretical contribution is an asymptotic optimality result for the weighted A-optimal item-selection rule under the item-type reuse regime (R1). When paired with the ML estimator, this rule attains the minimal asymptotic weighted mean squared error over a large class of adaptive item-selection procedures. A second contribution is an asymptotic normality result for the ML estimator under general adaptive item-selection rules that satisfy mild regularity conditions, covering both the item-type reuse regime (R1) and the unique-item-parameter regime (R2). This result justifies the use of standard information-based confidence intervals and hypothesis tests in MCAT, including settings where operational items are adaptively selected without repetition.

The simulation study complements the asymptotic theory by examining finite-sample behavior under a two-dimensional M2PL model. Under R1, the empirical selection frequencies of item types show much of the predicted long-run pattern at the finite test length $T=30$. Under R2, where each operational item has its own calibrated item-parameter values, we instead study stabilization in terms of the distribution of discrimination and difficulty parameters among selected items. The empirical distributions of these parameters also stabilize by $T=30$, suggesting that the adaptive design concentrates on regions of the item pool that are most informative for the examinee's ability. In terms of estimation accuracy, weighted A-optimality yields smaller $\widehat{\mathrm{WMSE}}_\delta$ than the competing item-selection rules. When $\delta$ is small, the weighted A-optimal rule improves intentional-ability accuracy with modest losses in nuisance-ability accuracy, providing a flexible tool for prioritizing intentional dimensions in practice.

Despite these strengths, our analysis has some limitations. The asymptotic optimality result is established under the item-type reuse regime (R1). In many MCAT applications, the operational item pool is closer to the unique-item-parameter regime (R2), where each calibrated item has its own item-parameter values rather than belonging to a repeated type. Although our simulations indicate that the R1 optimality result remains informative, it remains an open question whether an exact optimality theory can be developed under R2.

Several future directions remain. One direction is to relax the modeling assumptions, for example by allowing calibration error, model misspecification, or more general multidimensional response models, and to investigate the robustness of the proposed item-selection rule under these settings. Finally, our asymptotic normality and $\widehat{\mathrm{WMSE}}_\delta$ results are derived for fixed ability dimension $d$. Extending the theory to higher-dimensional or increasing-$d$ settings would make the framework more relevant for modern assessment applications involving large-scale data, hierarchical skill structures, and many latent traits.

\vspace{\fill}\pagebreak
\bibliographystyle{apacite}
\bibliography{reference}
\clearpage

\begin{table}[ht!] 
\centering
\caption{MSE for ability parameters under Assumption \ref{ass:r1} (R1)}
\label{t:r1 mse}
\begin{threeparttable}
\small
\begin{adjustbox}{max width=\textwidth}
\begin{tabular}{cc *{12}{S[table-format=1.3]}}
\toprule
\multicolumn{2}{c}{\makecell{\textbf{Parameters}\\$\bftheta$}} &
\multicolumn{6}{c}{${\operatorname{MSE}}_I(\hat{\bftheta}_{1,30})$} &
\multicolumn{6}{c}{${\operatorname{MSE}}_N(\hat{\bftheta}_{2,30})$}\\
\cmidrule(lr){1-2}\cmidrule(lr){3-8}\cmidrule(lr){9-14}
{$\theta_1$} & {$\theta_2$} & {A} & {D} & {R} & {$\mathrm{WA}_{0.1}$} & {$\mathrm{WA}_{0.001}$} & {$\mathrm{D}_s$} & {A} & {D} & {R} & {$\mathrm{WA}_{0.1}$} & {$\mathrm{WA}_{0.001}$} & {$\mathrm{D}_s$}\\
\midrule
1  &  1   & 0.4989 & 0.5783 & 1.3595 & 0.4649 & 0.4415 & 0.4399 & 0.4895 & 0.5748 & 1.2789 & 0.6565 & 0.8182 & 0.8142\\
1  &  0   & 0.2257 & 0.1970 & 0.8630 & 0.2021 & 0.2103 & 0.2104 & 0.2132 & 0.2397 & 0.8835 & 0.2631 & 0.3046 & 0.3079\\
1  & -1   & 0.1405 & 0.1375 & 0.7975 & 0.1281 & 0.1369 & 0.1370 & 0.1414 & 0.1753 & 0.8740 & 0.1668 & 0.2101 & 0.2101\\
0  &  1   & 0.2392 & 0.2503 & 0.8618 & 0.2155 & 0.2147 & 0.2160 & 0.3069 & 0.3953 & 0.9131 & 0.3908 & 0.4729 & 0.4757\\
0  &  0   & 0.1871 & 0.1780 & 0.8989 & 0.1440 & 0.1303 & 0.1307 & 0.1810 & 0.2198 & 0.9625 & 0.2041 & 0.2332 & 0.2342\\
0  & -1   & 0.1425 & 0.1677 & 0.9800 & 0.1288 & 0.1288 & 0.1285 & 0.1067 & 0.1267 & 1.0562 & 0.1848 & 0.3207 & 0.3190\\
-1 &  1   & 0.1990 & 0.1871 & 0.8252 & 0.1537 & 0.1435 & 0.1479 & 0.2890 & 0.2991 & 0.8558 & 0.3600 & 0.4480 & 0.5066\\
-1 &  0   & 0.2107 & 0.2397 & 0.9121 & 0.1723 & 0.1555 & 0.1562 & 0.2093 & 0.2122 & 0.9604 & 0.3415 & 0.4846 & 0.4995\\
-1 & -1   & 0.2228 & 0.2935 & 1.0515 & 0.1747 & 0.1651 & 0.1663 & 0.1454 & 0.1438 & 1.1662 & 0.2627 & 0.5201 & 0.5299\\
\midrule
\multicolumn{2}{l}{\textbf{Average}} &
0.2296 & 0.2477 & 0.9499 & 0.1982 & 0.1918 & 0.1925 &
0.2314 & 0.2652 & 0.9945 & 0.3145 & 0.4236 & 0.4330 \\
\bottomrule
\end{tabular}
\end{adjustbox}
\end{threeparttable}
\end{table}

\begin{table}[ht!] 
\centering
\caption{Weighted MSE for ability parameters under Assumption \ref{ass:r1} (R1)}
\label{t:r1 aggre}
\begin{threeparttable}
\small
\begin{adjustbox}{max width=\textwidth}
\begin{tabular}{cc *{18}{S[table-format=1.3]}}
\toprule
\multicolumn{2}{c}{\makecell{\textbf{Parameters}\\$\bftheta$}} &
\multicolumn{6}{c}
{${\operatorname{WMSE}}_{1}$} &
\multicolumn{6}{c}
{${\operatorname{WMSE}}_{0.1}$} &
\multicolumn{6}{c}
{${\operatorname{WMSE}}_{0.001}$}\\
\cmidrule(lr){1-2}\cmidrule(lr){3-8}\cmidrule(lr){9-14}\cmidrule(lr){15-20}
{$\theta_1$} & {$\theta_2$} &
{A} & {D} & {R} & {$\mathrm{WA}_{0.1}$} & {$\mathrm{WA}_{0.001}$} & {$\mathrm{D}_s$} &
{A} & {D} & {R} & {$\mathrm{WA}_{0.1}$} & {$\mathrm{WA}_{0.001}$} & {$\mathrm{D}_s$} &
{A} & {D} & {R} & {$\mathrm{WA}_{0.1}$} & {$\mathrm{WA}_{0.001}$} & {$\mathrm{D}_s$}\\
\midrule
1  &  1   & 0.9884 & 1.1531 & 2.6384 & 1.1214 & 1.2597 & 1.2541 & 0.5478 & 0.6358 & 1.4874 & 0.5305 & 0.5233 & 0.5213 & 0.4994 & 0.5789 & 1.3608 & 0.4656 & 0.4423 & 0.4407 \\
1  &  0   & 0.4389 & 0.4367 & 1.7465 & 0.4652 & 0.5149 & 0.5183 & 0.2470 & 0.2210 & 0.9513 & 0.2284 & 0.2408 & 0.2412 & 0.2259 & 0.1972 & 0.8639 & 0.2024 & 0.2106 & 0.2107 \\
1  & -1   & 0.2819 & 0.3128 & 1.6715 & 0.2949 & 0.3470 & 0.3471 & 0.1546 & 0.1550 & 0.8849 & 0.1448 & 0.1579 & 0.1579 & 0.1406 & 0.1377 & 0.7984 & 0.1282 & 0.1371 & 0.1372 \\
0  &  1   & 0.5461 & 0.6456 & 1.7749 & 0.6063 & 0.6876 & 0.6917 & 0.2699 & 0.2898 & 0.9531 & 0.2546 & 0.2620 & 0.2636 & 0.2395 & 0.2507 & 0.8627 & 0.2159 & 0.2152 & 0.2165 \\
0  &  0   & 0.3681 & 0.3978 & 1.8614 & 0.3481 & 0.3635 & 0.3649 & 0.2052 & 0.2000 & 0.9951 & 0.1644 & 0.1536 & 0.1541 & 0.1873 & 0.1782 & 0.8999 & 0.1442 & 0.1305 & 0.1309 \\
0  & -1   & 0.2492 & 0.2944 & 2.0362 & 0.3136 & 0.4495 & 0.4475 & 0.1532 & 0.1804 & 1.0856 & 0.1473 & 0.1609 & 0.1604 & 0.1426 & 0.1678 & 0.9811 & 0.1290 & 0.1291 & 0.1288 \\
-1 &  1   & 0.4880 & 0.4862 & 1.6810 & 0.5137 & 0.5915 & 0.6545 & 0.2279 & 0.2170 & 0.9108 & 0.1897 & 0.1883 & 0.1985 & 0.1993 & 0.1874 & 0.8261 & 0.1541 & 0.1440 & 0.1484 \\
-1 &  0   & 0.4200 & 0.4519 & 1.8725 & 0.5138 & 0.6401 & 0.6557 & 0.2316 & 0.2609 & 1.0081 & 0.2065 & 0.2039 & 0.2061 & 0.2109 & 0.2399 & 0.9131 & 0.1726 & 0.1560 & 0.1567 \\
-1 & -1   & 0.3682 & 0.4373 & 2.2177 & 0.4374 & 0.6852 & 0.6962 & 0.2373 & 0.3079 & 1.1681 & 0.2010 & 0.2171 & 0.2189 & 0.2230 & 0.2936 & 1.0527 & 0.1750 & 0.1656 & 0.1668 \\
\midrule
\multicolumn{2}{l}{\textbf{Average}} & 0.4610 & 0.5129 & 1.9445 & 0.5127 & 0.6154 & 0.6256 & 0.2527 & 0.2742 & 1.0494 & 0.2297 & 0.2342 & 0.2358 & 0.2298 & 0.2479 & 0.9509 & 0.1985 & 0.1923 & 0.1930 \\
\bottomrule
\end{tabular}
\end{adjustbox}
\end{threeparttable}
\end{table}

\begin{table}[ht!] 
\centering
\caption{MSE for ability parameters under Assumption \ref{ass:r2} (R2)}
\label{t:r2 mse}
\begin{threeparttable}
\small
\begin{adjustbox}{max width=\textwidth}
\begin{tabular}{cc *{12}{S[table-format=1.3]}}
\toprule
\multicolumn{2}{c}{\makecell{\textbf{Parameters}\\$\bftheta$}} &
\multicolumn{6}{c}{${\operatorname{MSE}}_I(\hat{\bftheta}_{1,30})$} &
\multicolumn{6}{c}{${\operatorname{MSE}}_N(\hat{\bftheta}_{2,30})$}\\
\cmidrule(lr){1-2}\cmidrule(lr){3-8}\cmidrule(lr){9-14}
{$\theta_1$} & {$\theta_2$} & {A} & {D} & {R} & {$\mathrm{WA}_{0.1}$} & {$\mathrm{WA}_{0.001}$} & {$\mathrm{D}_s$} & {A} & {D} & {R} & {$\mathrm{WA}_{0.1}$} & {$\mathrm{WA}_{0.001}$} & {$\mathrm{D}_s$}\\
\midrule
1  &  1   & 0.1742 & 0.2023 & 0.7604 & 0.1548 & 0.1507 & 0.1506 & 0.2235 & 0.2429 & 0.7158 & 0.2769 & 0.3538 & 0.3544\\
1  &  0   & 0.1747 & 0.2017 & 0.5851 & 0.1455 & 0.1507 & 0.1484 & 0.1410 & 0.1531 & 0.4580 & 0.1731 & 0.2139 & 0.2193\\
1  & -1   & 0.1375 & 0.1592 & 0.5627 & 0.1285 & 0.1231 & 0.1251 & 0.1180 & 0.1287 & 0.4932 & 0.1524 & 0.1840 & 0.1937\\
0  &  1   & 0.1544 & 0.1794 & 0.6073 & 0.1355 & 0.1277 & 0.1279 & 0.1433 & 0.1600 & 0.5023 & 0.1921 & 0.2317 & 0.2335\\
0  &  0   & 0.1404 & 0.1686 & 0.5101 & 0.1263 & 0.1323 & 0.1330 & 0.1448 & 0.1750 & 0.4349 & 0.1897 & 0.2355 & 0.2506\\
0  & -1   & 0.1414 & 0.1564 & 0.5968 & 0.1288 & 0.1288 & 0.1283 & 0.1346 & 0.1514 & 0.5374 & 0.1848 & 0.2034 & 0.2175\\
-1 &  1   & 0.1785 & 0.1931 & 0.6113 & 0.1595 & 0.1573 & 0.1559 & 0.1653 & 0.1785 & 0.5053 & 0.2163 & 0.2468 & 0.2514\\
-1 &  0   & 0.1660 & 0.1737 & 0.5843 & 0.1564 & 0.1463 & 0.1520 & 0.1408 & 0.1408 & 0.4383 & 0.1923 & 0.2227 & 0.2334\\
-1 & -1   & 0.1933 & 0.2130 & 0.7425 & 0.1740 & 0.1716 & 0.1712 & 0.1717 & 0.1804 & 0.6878 & 0.2105 & 0.2712 & 0.2912\\
\midrule
\multicolumn{2}{l}{\textbf{Average}} &
0.1623 & 0.1830 & 0.6178 & 0.1455 & 0.1432 & 0.1436 &
0.1537 & 0.1679 & 0.5303 & 0.1987 & 0.2403 & 0.2494 \\
\bottomrule
\end{tabular}
\end{adjustbox}
\end{threeparttable}
\end{table}

\begin{table}[ht!] 
\centering
\caption{Weighted MSE for ability parameters under Assumption \ref{ass:r2} (R2)}
\label{t:r2 aggre}
\begin{threeparttable}
\small
\begin{adjustbox}{max width=\textwidth}
\begin{tabular}{cc *{18}{S[table-format=1.3]}}
\toprule
\multicolumn{2}{c}{\makecell{\textbf{Parameters}\\$\bftheta$}} &
\multicolumn{6}{c}{${\operatorname{WMSE}}_{1}$} &
\multicolumn{6}{c}{${\operatorname{WMSE}}_{0.1}$} &
\multicolumn{6}{c}{${\operatorname{WMSE}}_{0.001}$} \\
\cmidrule(lr){1-2}\cmidrule(lr){3-8}\cmidrule(lr){9-14}\cmidrule(lr){15-20}
{$\theta_1$} & {$\theta_2$} &
{A} & {D} & {R} & {$\mathrm{WA}_{0.1}$} & {$\mathrm{WA}_{0.001}$} & {$\mathrm{D}_s$} &
{A} & {D} & {R} & {$\mathrm{WA}_{0.1}$} & {$\mathrm{WA}_{0.001}$} & {$\mathrm{D}_s$} &
{A} & {D} & {R} & {$\mathrm{WA}_{0.1}$} & {$\mathrm{WA}_{0.001}$} & {$\mathrm{D}_s$} \\
\midrule
1  &  1   & 0.3977 & 0.4452 & 1.4762 & 0.4317 & 0.5045 & 0.5050 & 0.1966 & 0.2266 & 0.8320 & 0.1825 & 0.1861 & 0.1860 & 0.1744 & 0.2025 & 0.7611 & 0.1551 & 0.1511 & 0.1510 \\
1  &  0   & 0.3157 & 0.3548 & 1.0431 & 0.3186 & 0.3646 & 0.3677 & 0.1888 & 0.2170 & 0.6309 & 0.1628 & 0.1721 & 0.1703 & 0.1748 & 0.2019 & 0.5856 & 0.1457 & 0.1509 & 0.1486 \\
1  & -1   & 0.2555 & 0.2879 & 1.0559 & 0.2809 & 0.3071 & 0.3188 & 0.1493 & 0.1721 & 0.6120 & 0.1437 & 0.1415 & 0.1445 & 0.1376 & 0.1593 & 0.5632 & 0.1286 & 0.1233 & 0.1253 \\
0  &  1   & 0.2977 & 0.3394 & 1.1096 & 0.3276 & 0.3594 & 0.3614 & 0.1687 & 0.1954 & 0.6575 & 0.1637 & 0.1509 & 0.1513 & 0.1545 & 0.1796 & 0.6078 & 0.1357 & 0.1279 & 0.1281 \\
0  &  0   & 0.2852 & 0.3436 & 0.9450 & 0.3157 & 0.3678 & 0.3836 & 0.1549 & 0.1861 & 0.5536 & 0.1453 & 0.1559 & 0.1581 & 0.1405 & 0.1688 & 0.5106 & 0.1265 & 0.1325 & 0.1333 \\
0  & -1   & 0.2760 & 0.3078 & 1.1342 & 0.3132 & 0.3322 & 0.3458 & 0.1549 & 0.1715 & 0.6505 & 0.1529 & 0.1492 & 0.1501 & 0.1415 & 0.1566 & 0.5973 & 0.1346 & 0.1290 & 0.1285 \\
-1 &  1   & 0.3438 & 0.3716 & 1.1166 & 0.3758 & 0.4041 & 0.4073 & 0.1950 & 0.2110 & 0.6618 & 0.1814 & 0.1820 & 0.1810 & 0.1787 & 0.1933 & 0.6118 & 0.1597 & 0.1576 & 0.1561 \\
-1 &  0   & 0.3068 & 0.3145 & 1.0226 & 0.3487 & 0.3690 & 0.3854 & 0.1801 & 0.1878 & 0.6281 & 0.1756 & 0.1685 & 0.1753 & 0.1661 & 0.1738 & 0.5848 & 0.1566 & 0.1465 & 0.1522 \\
-1 & -1   & 0.3650 & 0.3934 & 1.4303 & 0.3845 & 0.4428 & 0.4624 & 0.2105 & 0.2310 & 0.8113 & 0.1950 & 0.1987 & 0.2003 & 0.1935 & 0.2132 & 0.7432 & 0.1742 & 0.1719 & 0.1715 \\
\midrule
\multicolumn{2}{l}{\textbf{Average}} &
0.3162 & 0.3512 & 1.1437 & 0.3441 & 0.3835 & 0.3930 &
0.1771 & 0.1998 & 0.6709 & 0.1670 & 0.1672 & 0.1686 &
0.1625 & 0.1832 & 0.6183 & 0.1457 & 0.1434 & 0.1439 \\
\bottomrule
\end{tabular}
\end{adjustbox}
\end{threeparttable}
\end{table}

\begin{figure}[ht!]
    \centering
    \includegraphics[width=0.8\textwidth]{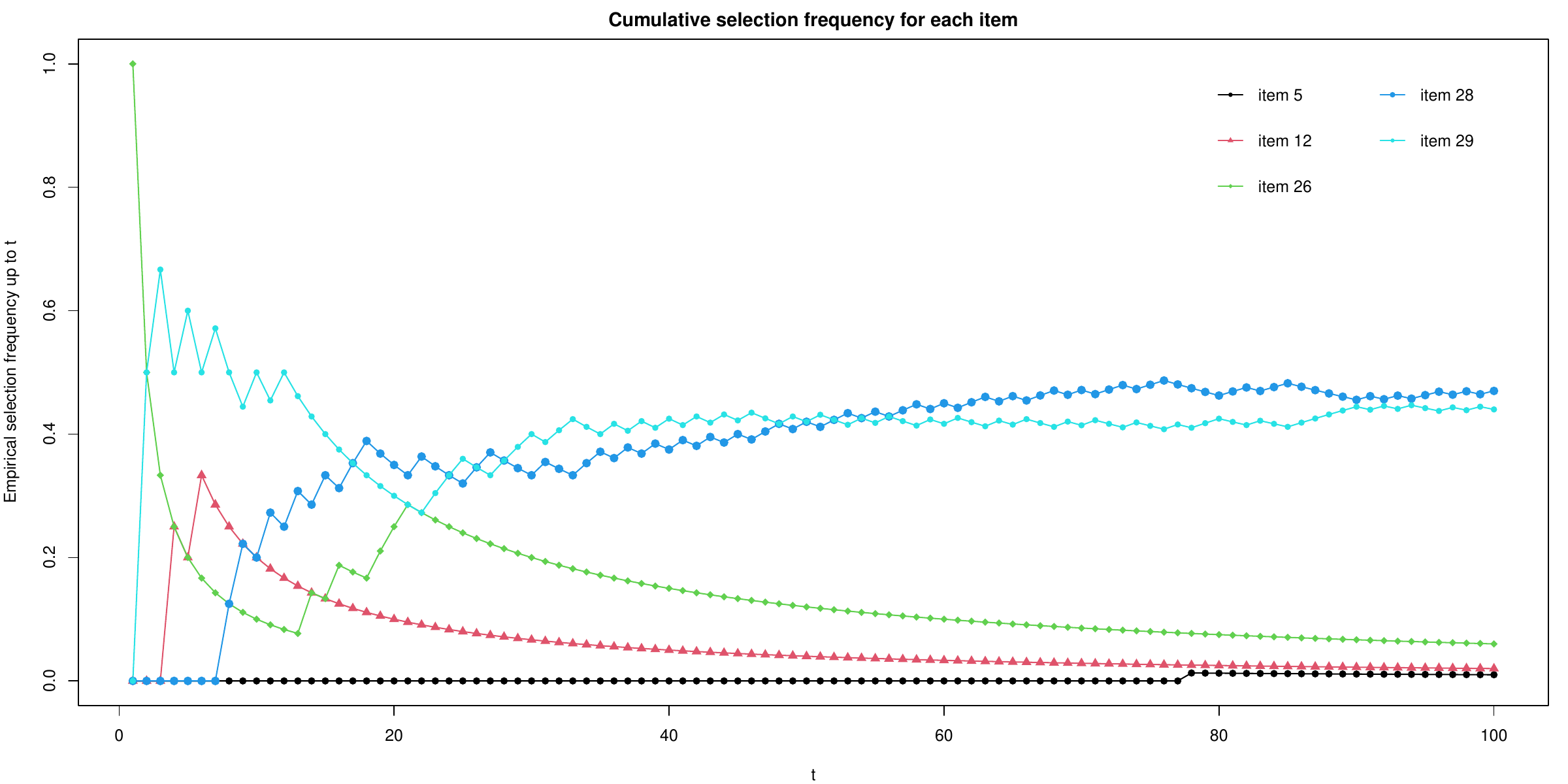}
    \caption{Empirical ratios of the selected items for $\boldsymbol{\theta}^{\ast} = (-1,0)$.}
    \label{fig:reuse1}
\end{figure}

\begin{figure}[ht!]
    \centering
    \includegraphics[width=0.8\textwidth]{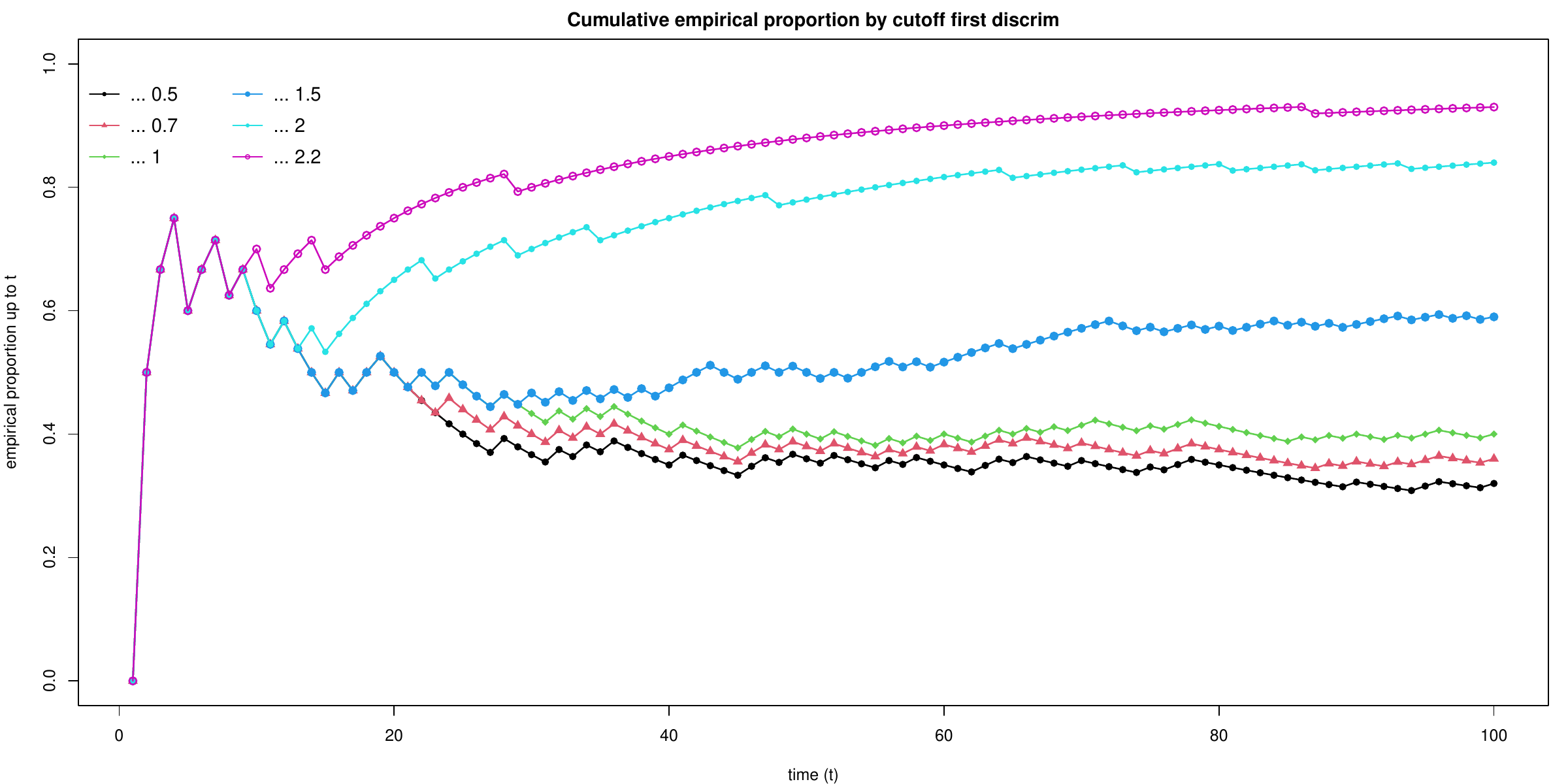}
    \caption{Empirical ratios of the first discrimination parameter $\alpha_1$, determined by cutoffs $(0.5,0.7,1,1.5,2,2.2)$ (from bottom to top) for $\boldsymbol{\theta}^{\ast} = (-1,0)$.}
    \label{fig:prop1}
\end{figure}

\begin{figure}[ht!]
    \centering
    \includegraphics[width=0.8\textwidth]{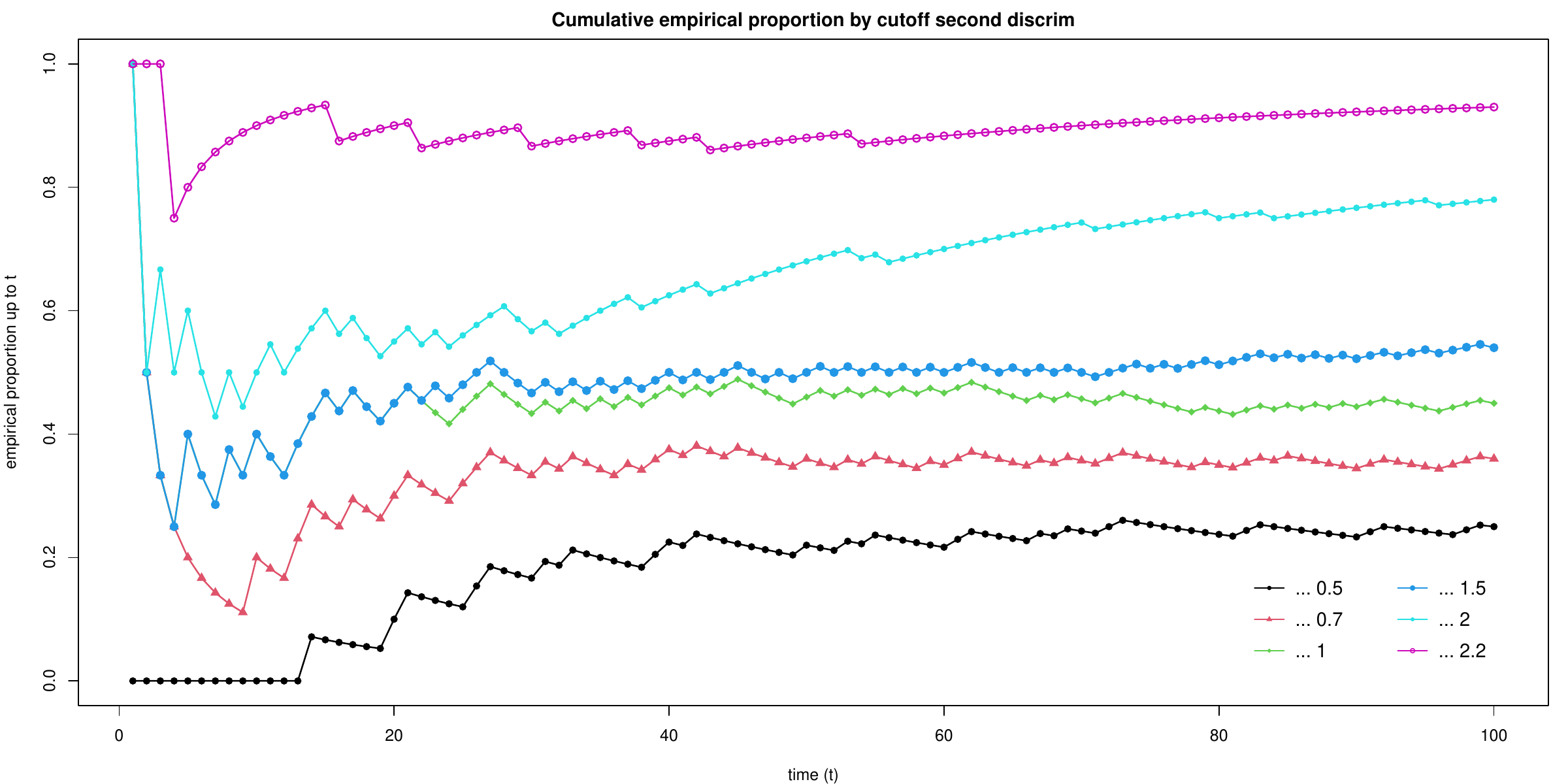}
    \caption{Empirical ratios of the second discrimination parameter $\alpha_2$, determined by cutoffs $(0.5,0.7,1,1.5,2,2.2)$ (from bottom to top) for $\boldsymbol{\theta}^{\ast} = (-1,0)$.}
    \label{fig:prop2}
\end{figure}

\begin{figure}[ht!]
    \centering
    \includegraphics[width=0.8\textwidth]{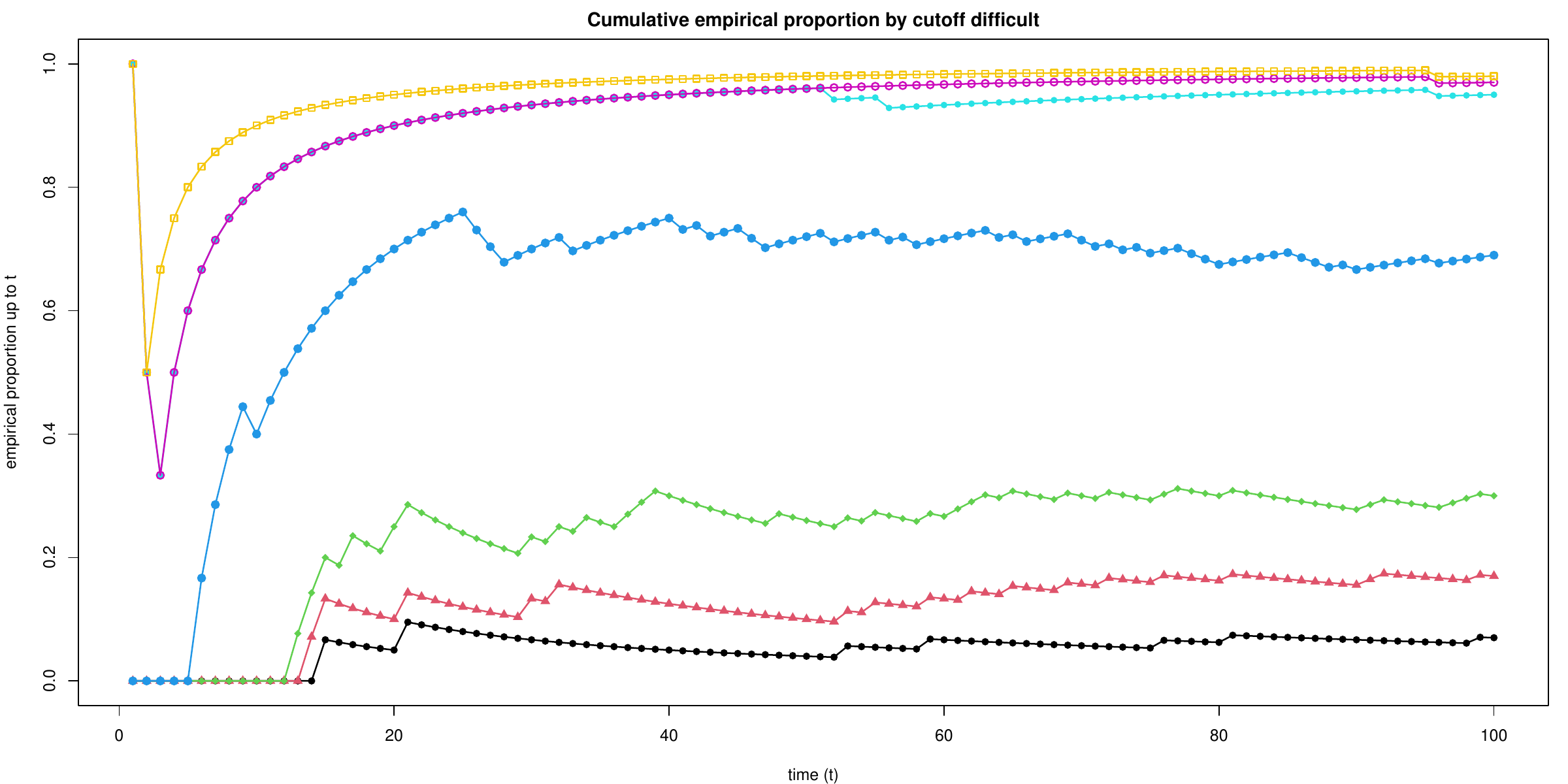}
    \caption{Empirical ratios of the difficulty parameter $b$, determined by cutoffs $(-2,-1.5,-1,0,1,1.5,2)$ (from bottom to top) for $\boldsymbol{\theta}^{\ast} = (-1,0)$.}
    \label{fig:prop3}
\end{figure}

\clearpage
\appendix
\setcounter{figure}{0}
\renewcommand{\thefigure}{A\arabic{figure}}
\renewcommand{\theequation}{\thesection\arabic{equation}}
\renewcommand{\theHfigure}{appendix.\thesection.\arabic{figure}}
\renewcommand{\theHequation}{appendix.\thesection.\arabic{equation}}
\setcounter{equation}{0}
\renewcommand{\thesection}{\Alph{section}}
\setcounter{section}{0}
\refstepcounter{section}
\section*{Appendix \thesection} \label{appendix:simulation}
In this appendix, we present additional simulation results that complement Subsection \ref{sub:stabilization}. We report the case with $\delta = 1$, and similar behavior was observed for $\delta = 0.1$ and $\delta = 0.001$. The overall patterns across the eight remaining cases, namely all settings except $\bftheta^{\ast} = (-1,0)$, are consistent with those reported in Section \ref{sec:simulation}. Figures \ref{fig:1_1_reuse} -- \ref{fig:m1_m1_reuse} display the empirical selection ratios of item types under Assumption \ref{ass:r1}. As in Section \ref{sec:simulation}, the item-selection behavior stabilizes around $T =30$. Also, only a small subset of item types is selected repeatedly across different choices of $\bftheta^{\ast}$, indicating that relatively few item types are informative for estimating examinee abilities.

Figures \ref{fig:1_1_single} -- \ref{fig:m1_m1_single} report the empirical proportions of the two discrimination parameters $(\alpha_{1,j}, \alpha_{2,j})$ and the difficulty parameter  $b_j$ under Assumption \ref{ass:r2} (R2). We use the same cutoff values for discrimination and difficulty parameters as in Section \ref{sec:simulation}. For all choices of $\bftheta^{\ast}$, the selection patterns stabilize around a test length of $30$, which is the behavior observed in the $\bftheta^{\ast} = (-1,0)$ case. These results further support the robustness of the empirical selection stabilization under both item-pool regimes.

\begin{figure}[ht!]
    \centering
    \includegraphics[width=0.8\textwidth]{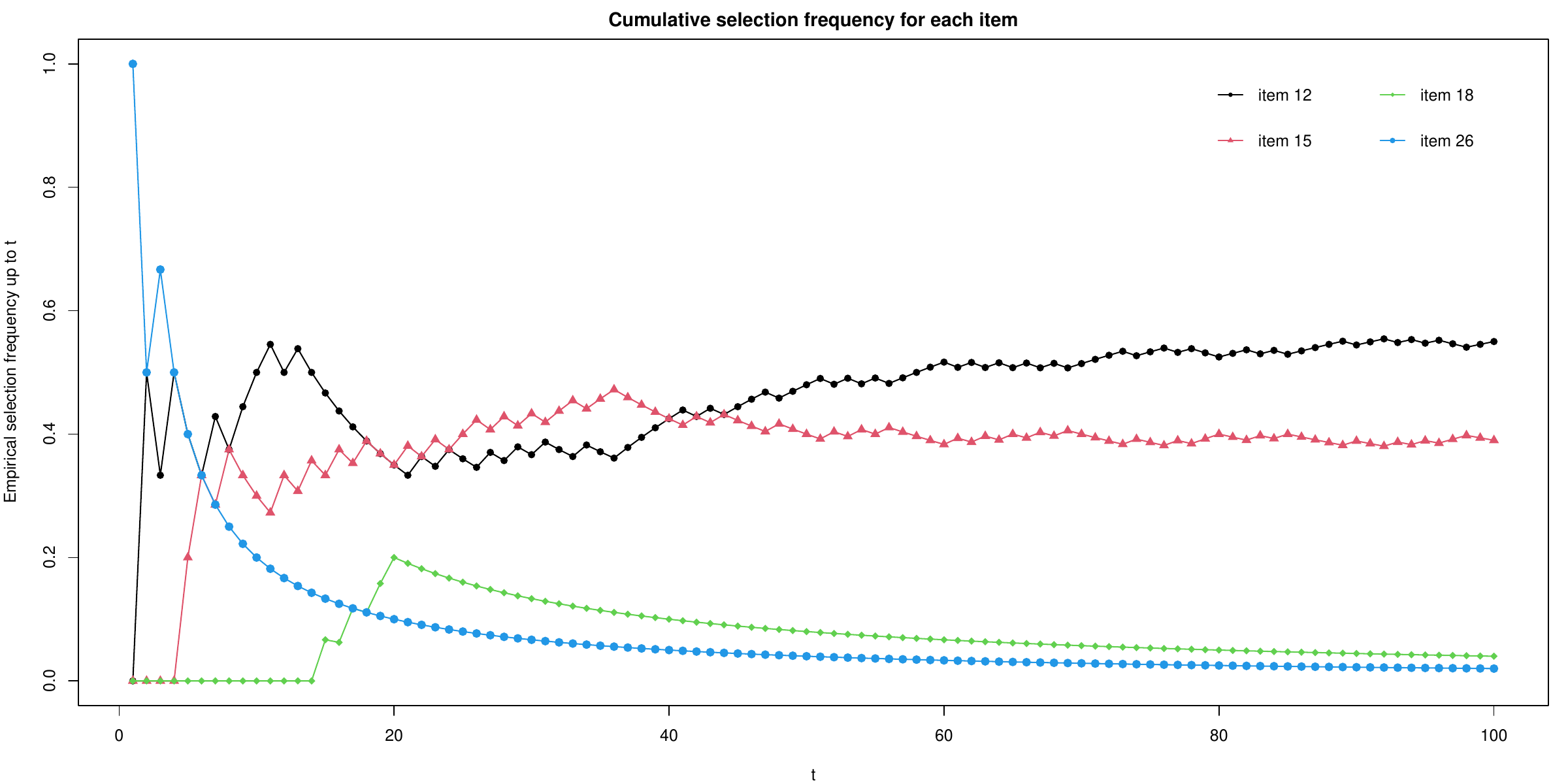}
    \caption{Empirical ratios of the selected item types for $\boldsymbol{\theta}^{\ast} = (1,1)$.}
    \label{fig:1_1_reuse}
\end{figure}

\begin{figure}[ht!]
    \centering
    \includegraphics[width=0.8\textwidth]{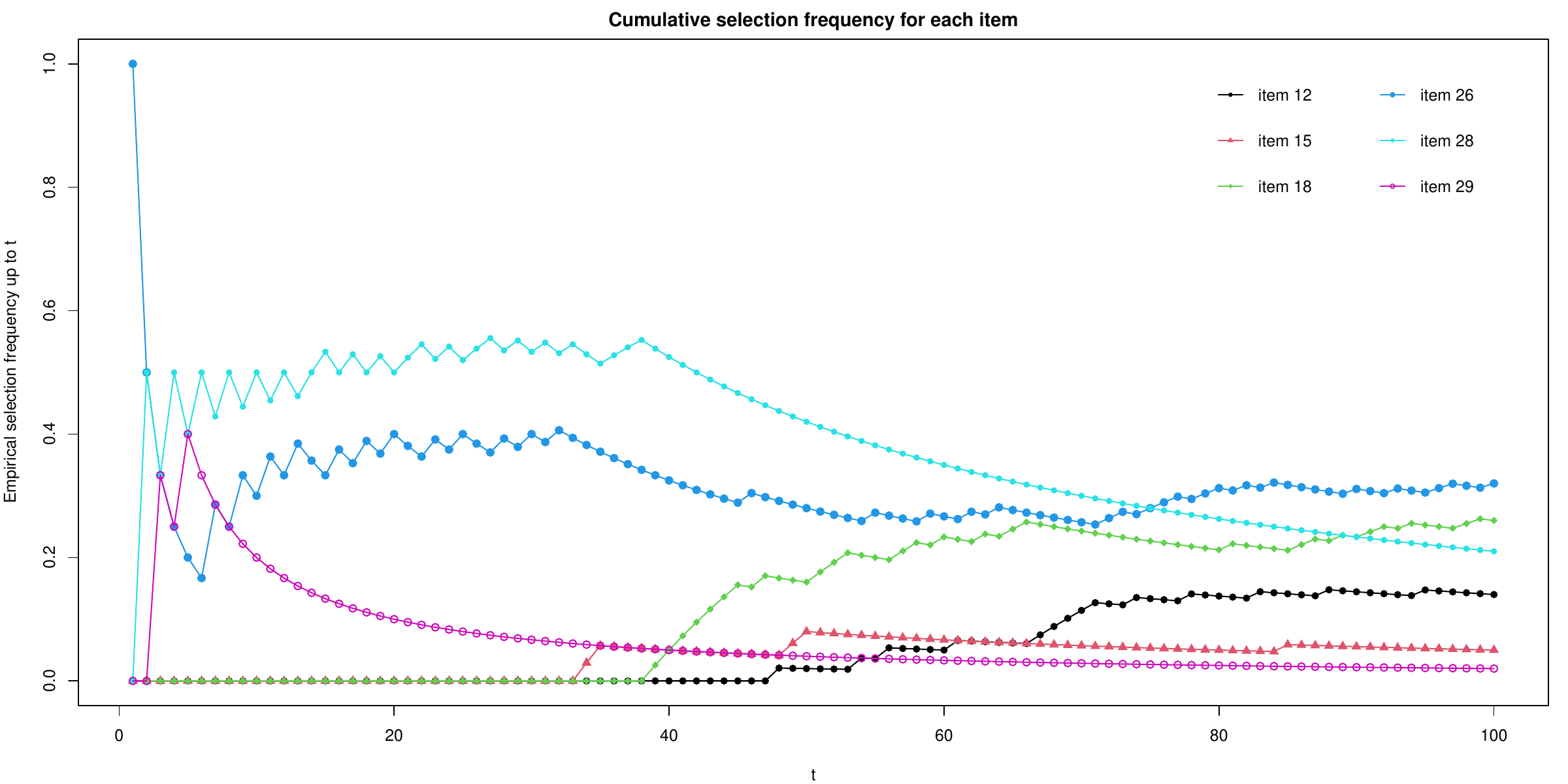}
    \caption{Empirical ratios of the selected item types for $\boldsymbol{\theta}^{\ast} = (1,0)$.}
    \label{fig:1_0_reuse}
\end{figure}

\begin{figure}[ht!]
    \centering
    \includegraphics[width=0.8\textwidth]{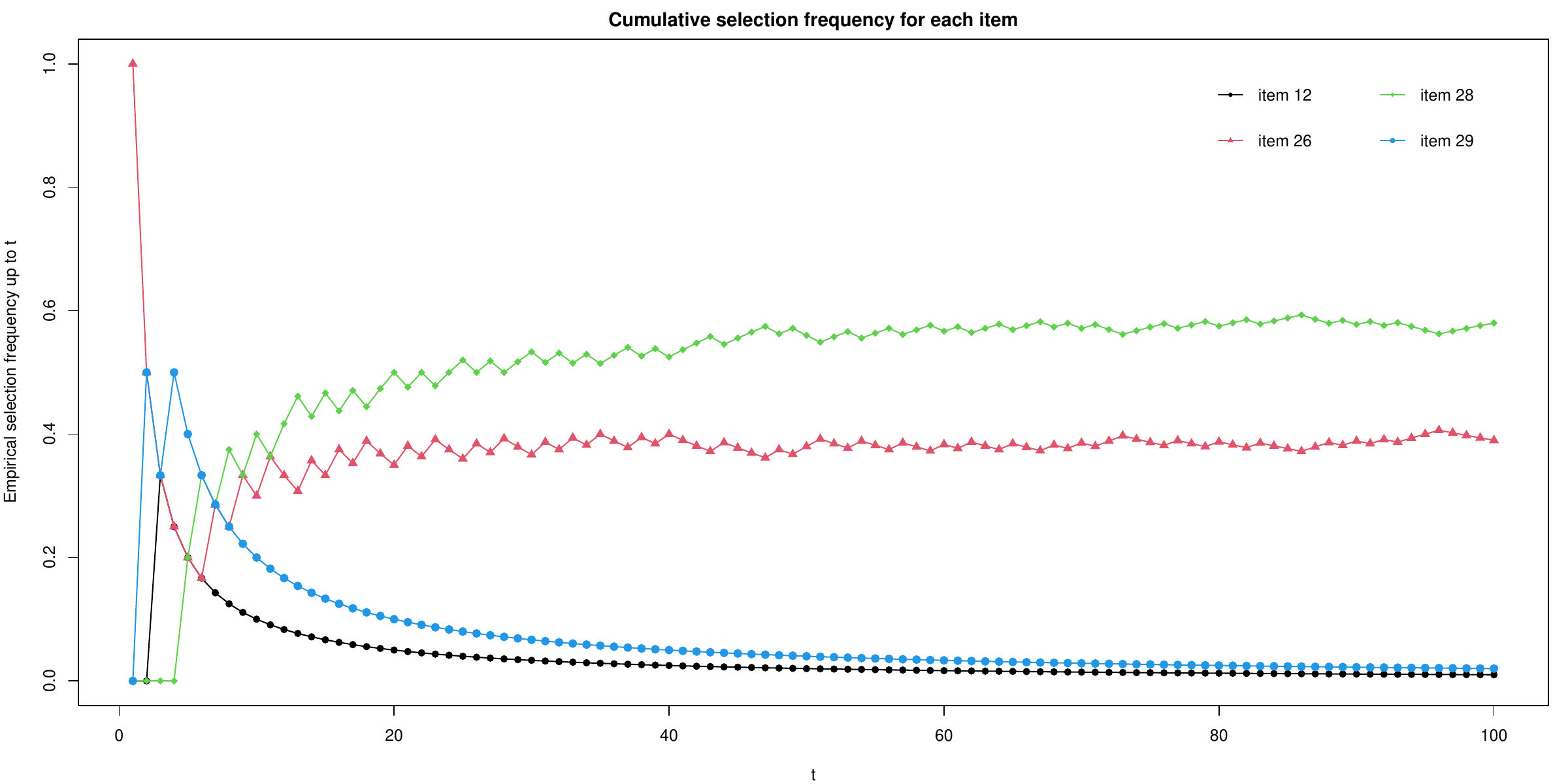}
    \caption{Empirical ratios of the selected item types for $\boldsymbol{\theta}^{\ast} = (1,-1)$.}
    \label{fig:1_m1_reuse}
\end{figure}

\begin{figure}[ht!]
    \centering
    \includegraphics[width=0.8\textwidth]{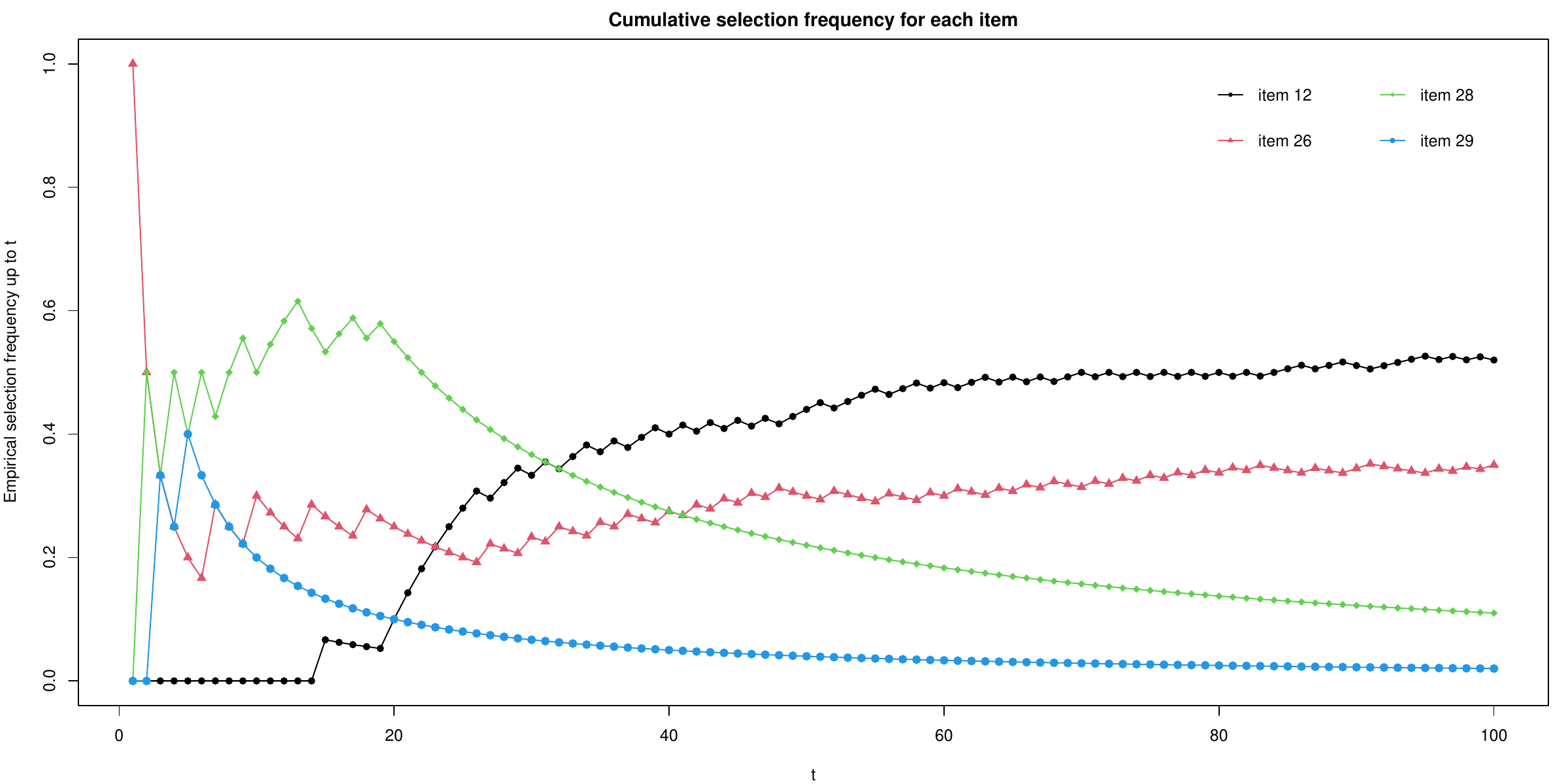}
    \caption{Empirical ratios of the selected item types for $\boldsymbol{\theta}^{\ast} = (0,1)$.}
    \label{fig:0_1_resuse}
\end{figure}

\begin{figure}[ht!]
    \centering
    \includegraphics[width=0.8\textwidth]{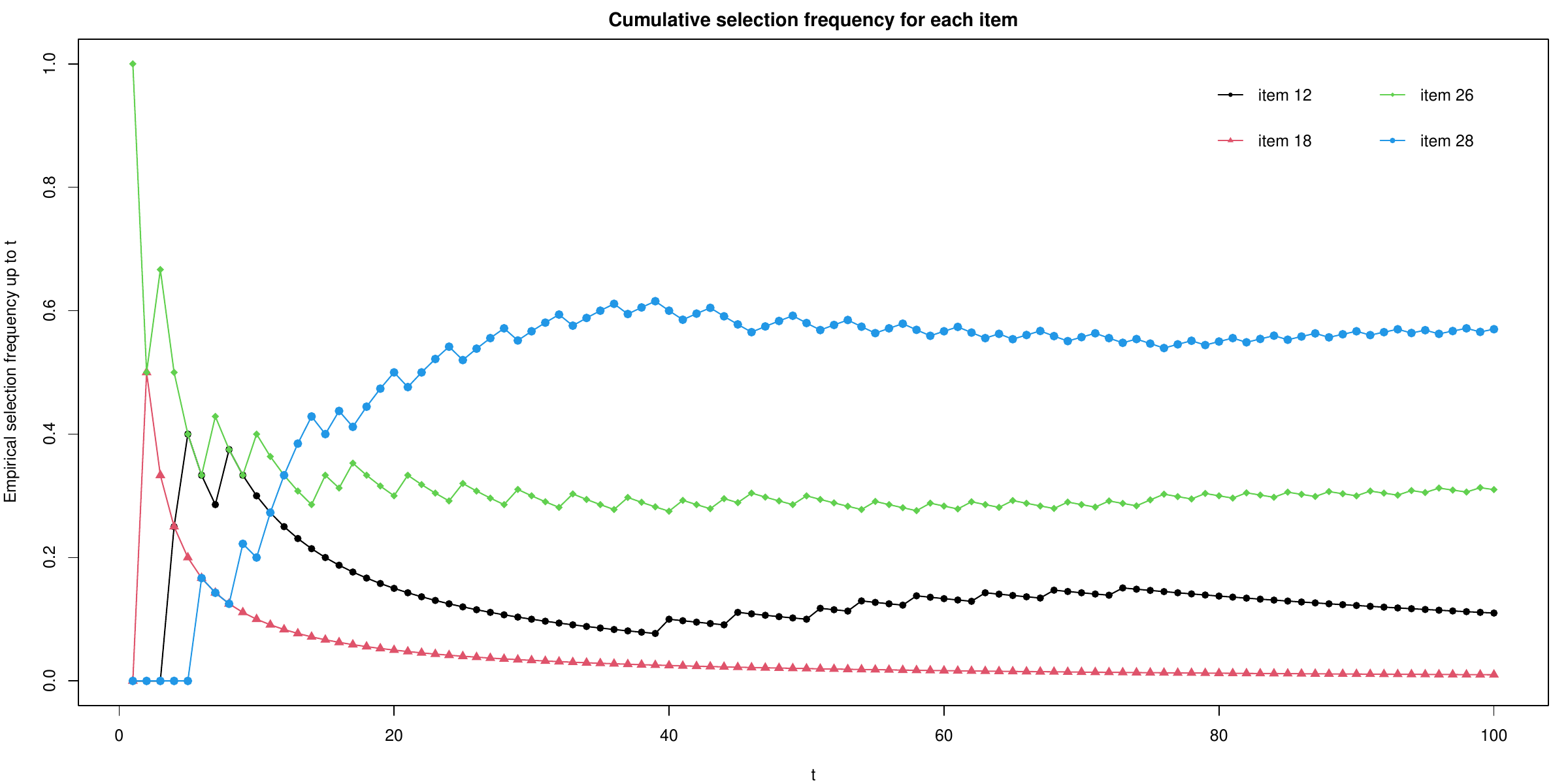}
    \caption{Empirical ratios of the selected item types for $\boldsymbol{\theta}^{\ast} = (0,0)$.}
    \label{fig:0_0_reuse}
\end{figure}

\begin{figure}[ht!]
    \centering
    \includegraphics[width=0.8\textwidth]{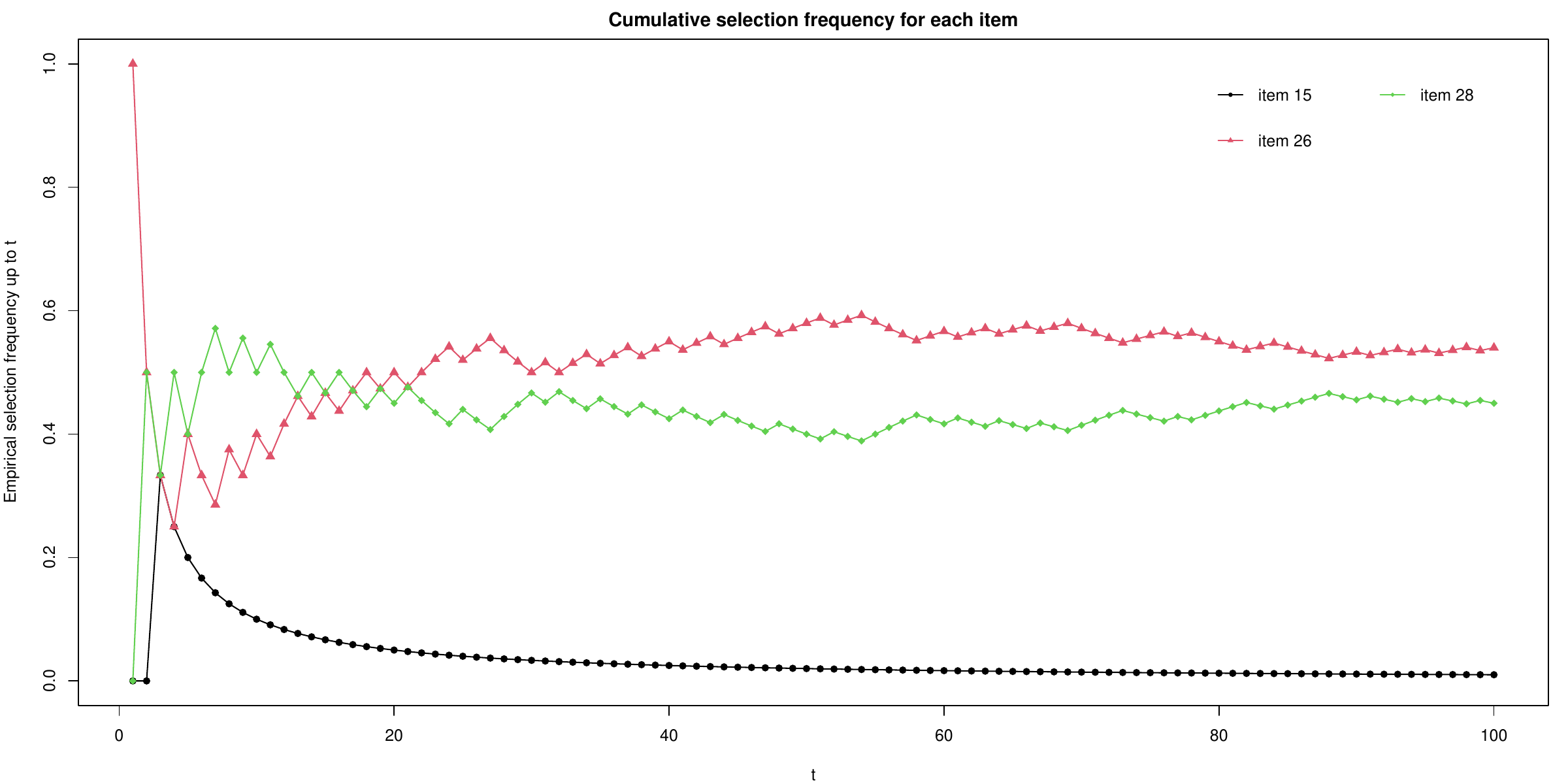}
    \caption{Empirical ratios of the selected item types for $\boldsymbol{\theta}^{\ast} = (0,-1)$.}
    \label{fig:0_m1_reuse}
\end{figure}

\begin{figure}[ht!]
    \centering
    \includegraphics[width=0.8\textwidth]{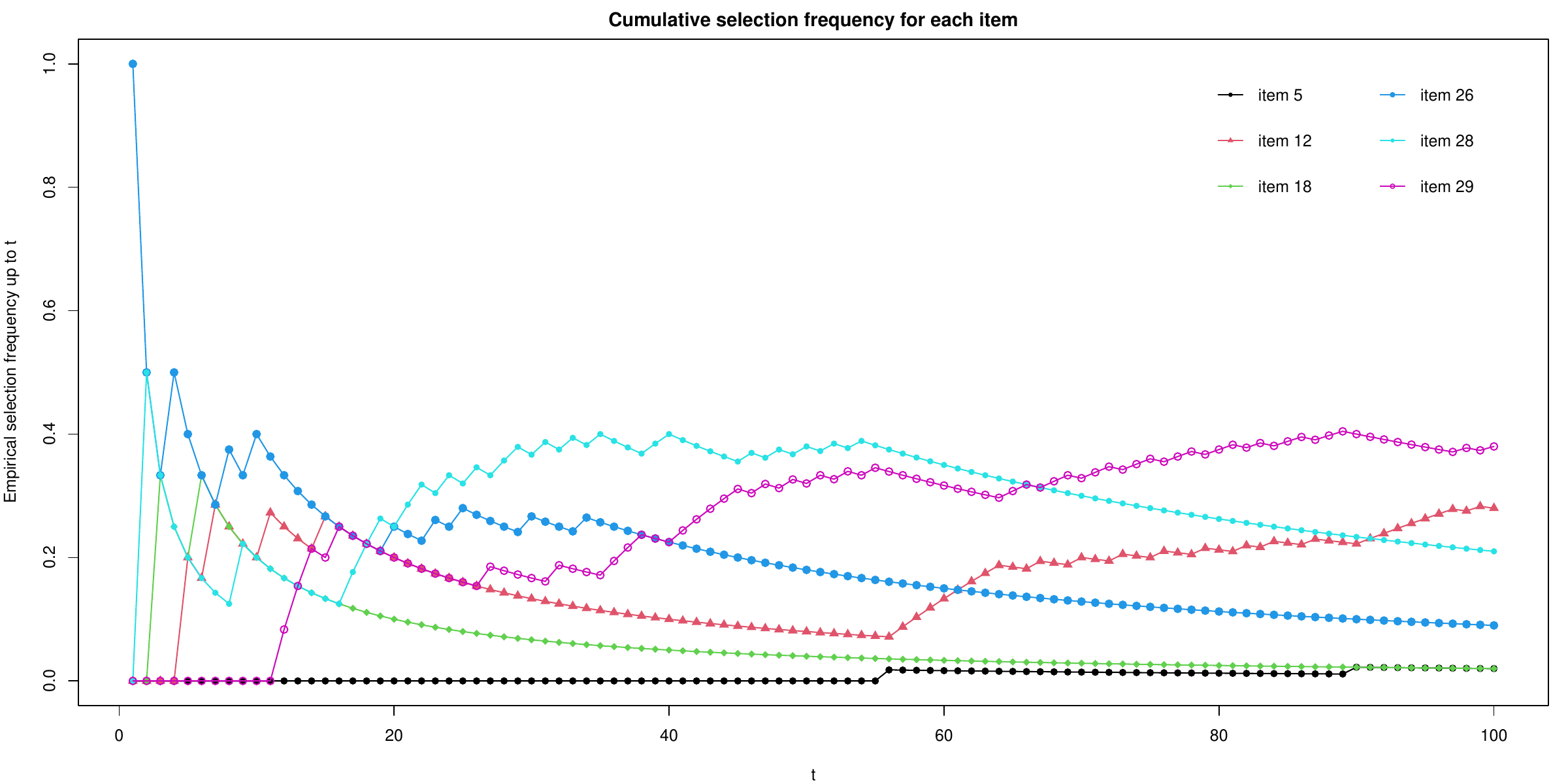}
    \caption{Empirical ratios of the selected item types for $\boldsymbol{\theta}^{\ast} = (-1,1)$.}
    \label{fig:m1_1_reuse}
\end{figure}

\begin{figure}[ht!]
    \centering
    \includegraphics[width=0.8\textwidth]{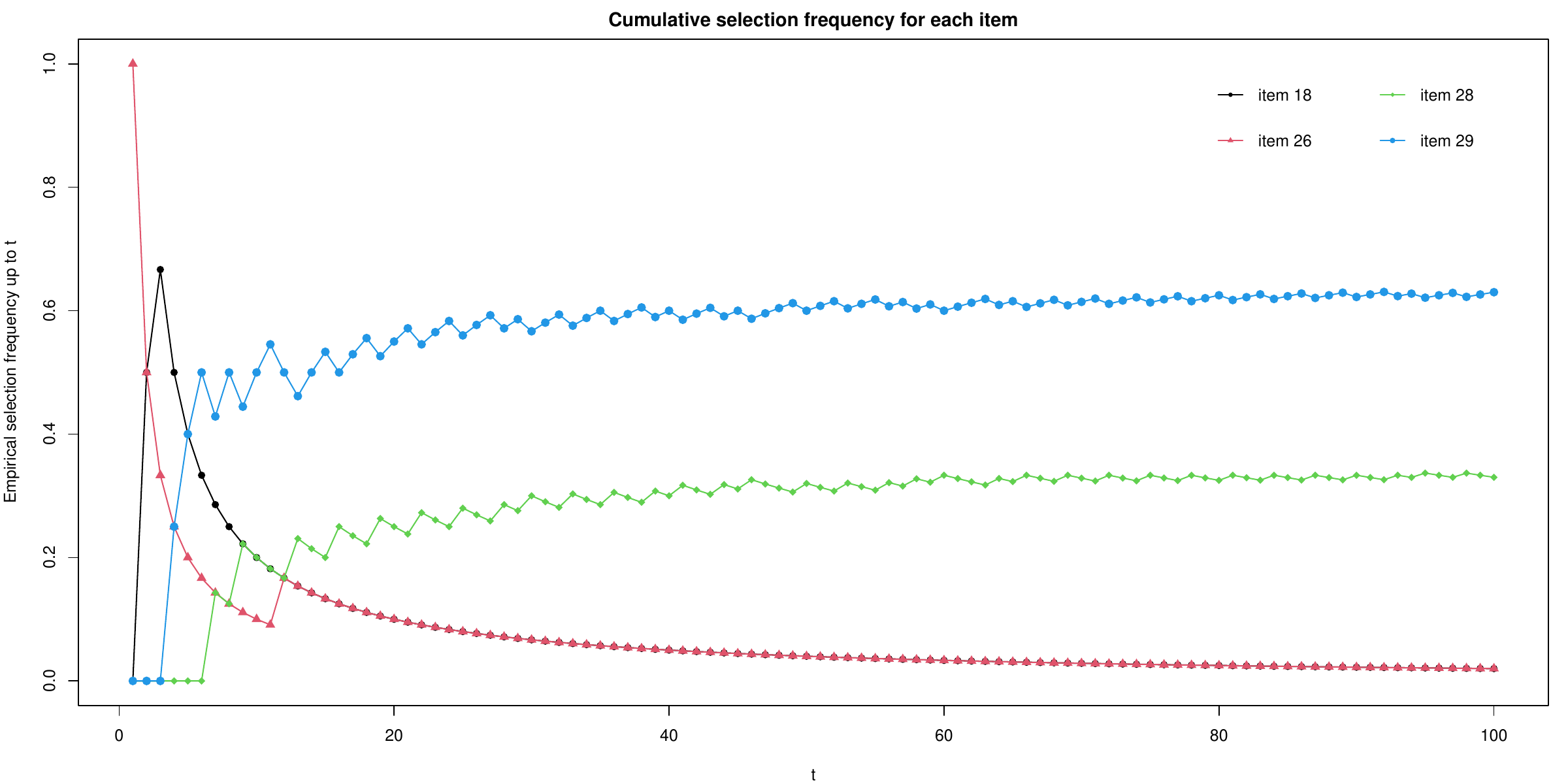}
    \caption{Empirical ratios of the selected item types for $\boldsymbol{\theta}^{\ast} = (-1,-1)$.}
    \label{fig:m1_m1_reuse}
\end{figure}

\begin{figure}[ht!]
    \centering
    \includegraphics[width=0.8\textwidth]{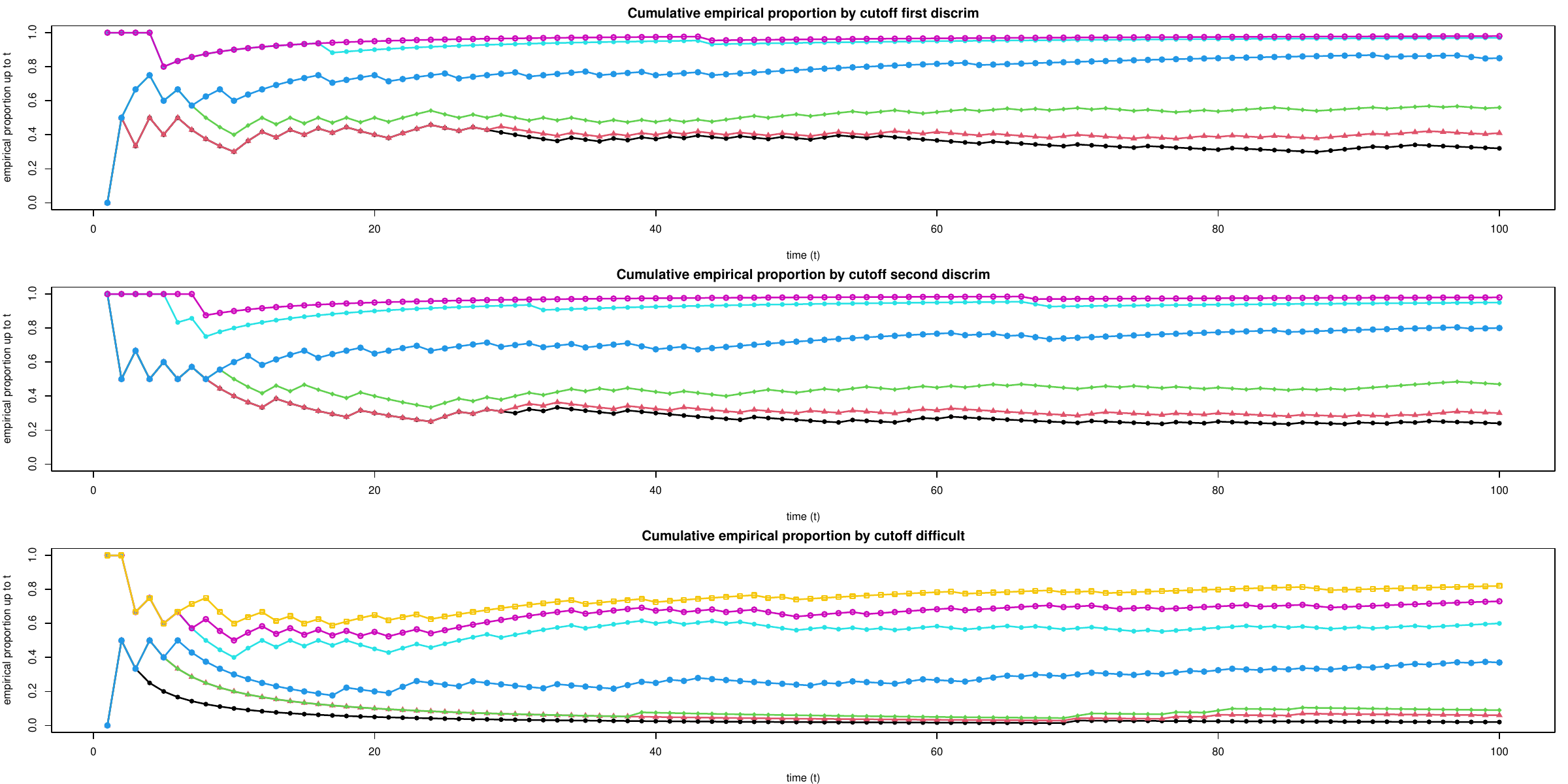}
    \caption{Empirical ratios of parameters $(\alpha_1, \alpha_2, b)$ under $\boldsymbol{\theta}^{\ast} = (1,1)$.}
    \label{fig:1_1_single}
\end{figure}

\begin{figure}[ht!]
    \centering
    \includegraphics[width=0.8\textwidth]{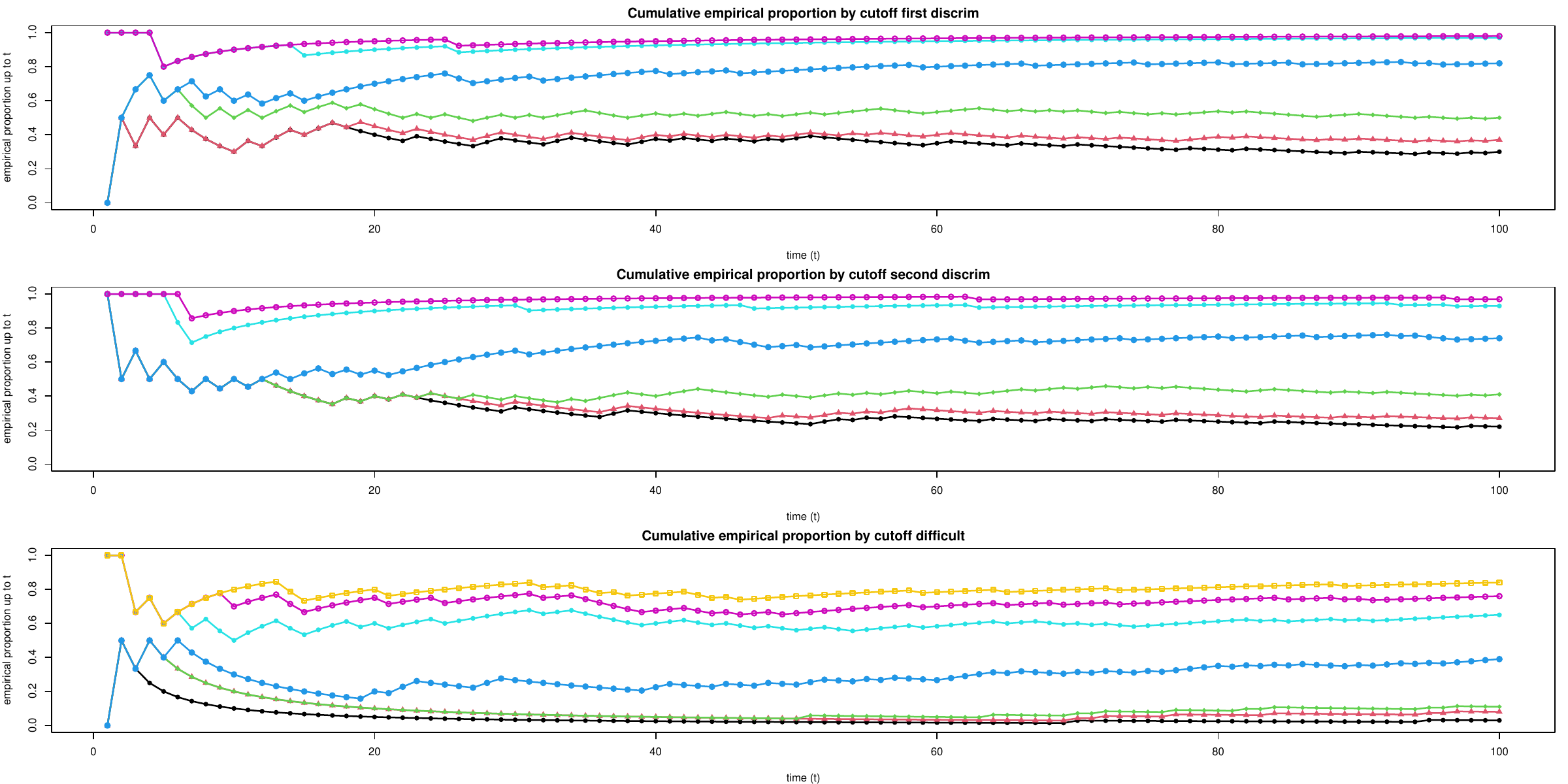}
    \caption{Empirical ratios of parameters $(\alpha_1, \alpha_2, b)$ under $\boldsymbol{\theta}^{\ast} = (1,0)$.}
    \label{fig:1_0_single}
\end{figure}

\begin{figure}[ht!]
    \centering
    \includegraphics[width=0.8\textwidth]{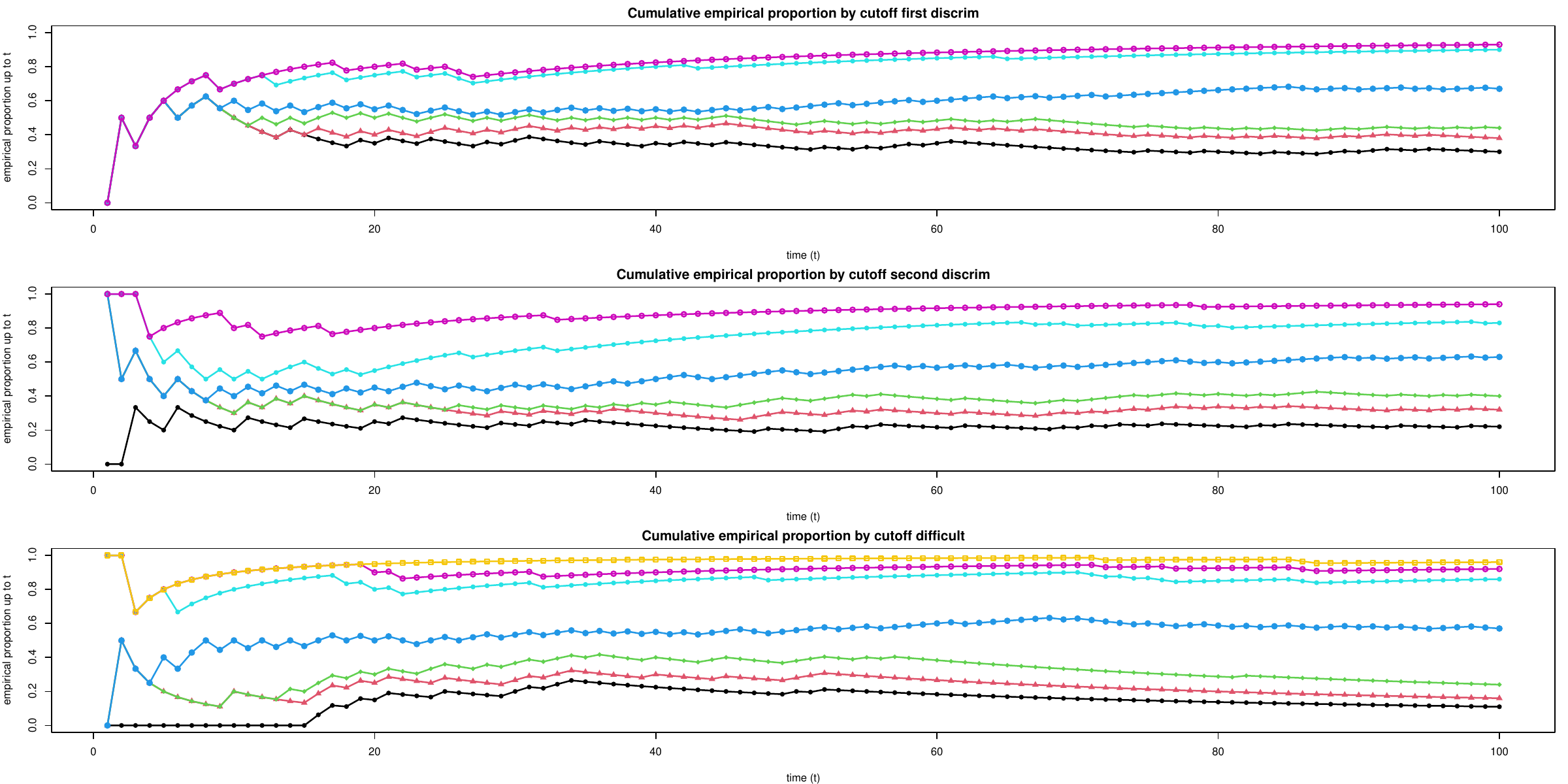}
    \caption{Empirical ratios of parameters $(\alpha_1, \alpha_2, b)$ under $\boldsymbol{\theta}^{\ast} = (1,-1)$.}
    \label{fig:1_m1_single}
\end{figure}

\begin{figure}[ht!]
    \centering
    \includegraphics[width=0.8\textwidth]{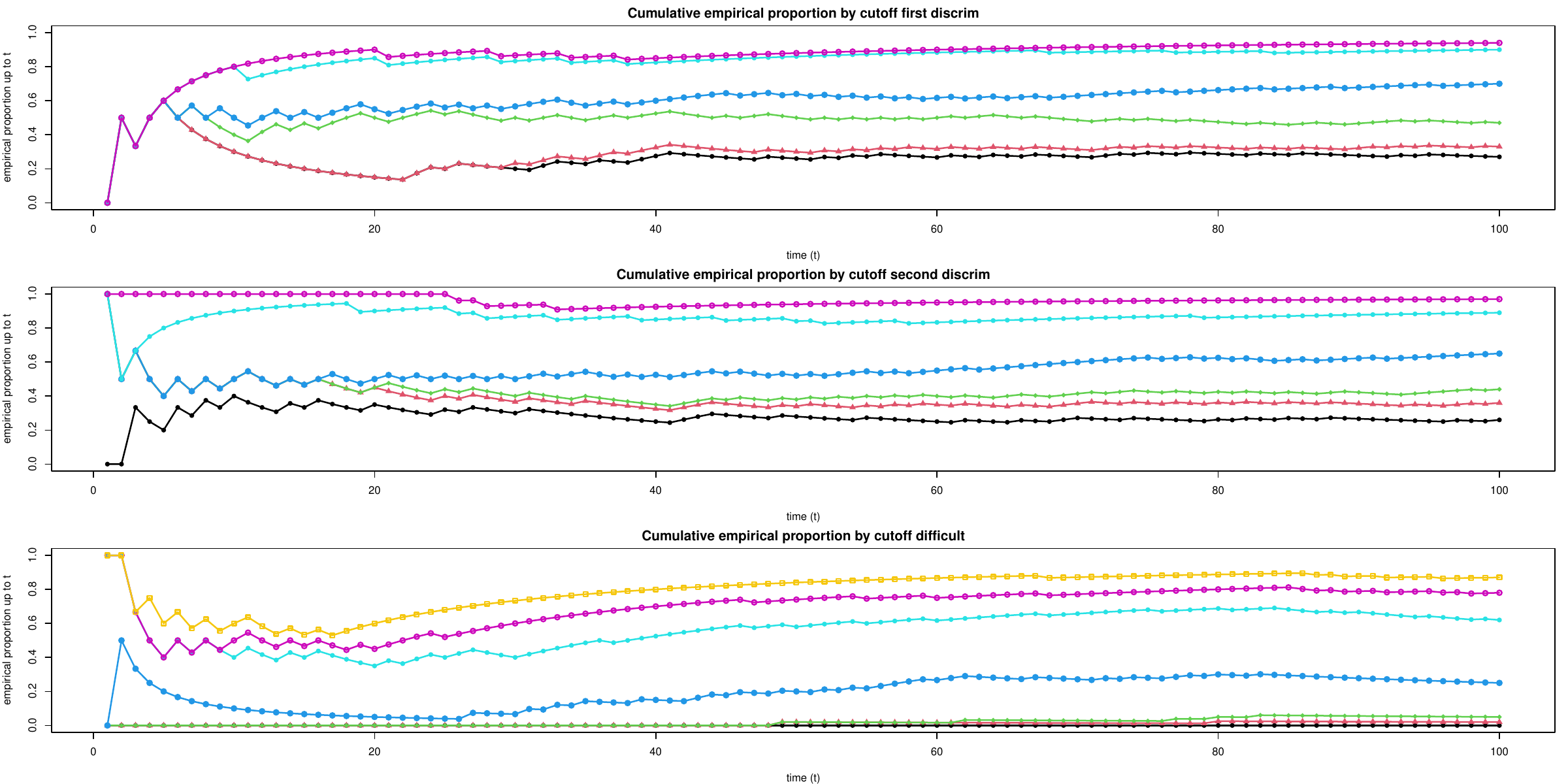}
    \caption{Empirical ratios of parameters $(\alpha_1, \alpha_2, b)$ under $\boldsymbol{\theta}^{\ast} = (0,1)$.}
    \label{fig:0_1_single}
\end{figure}

\begin{figure}[ht!]
    \centering
    \includegraphics[width=0.8\textwidth]{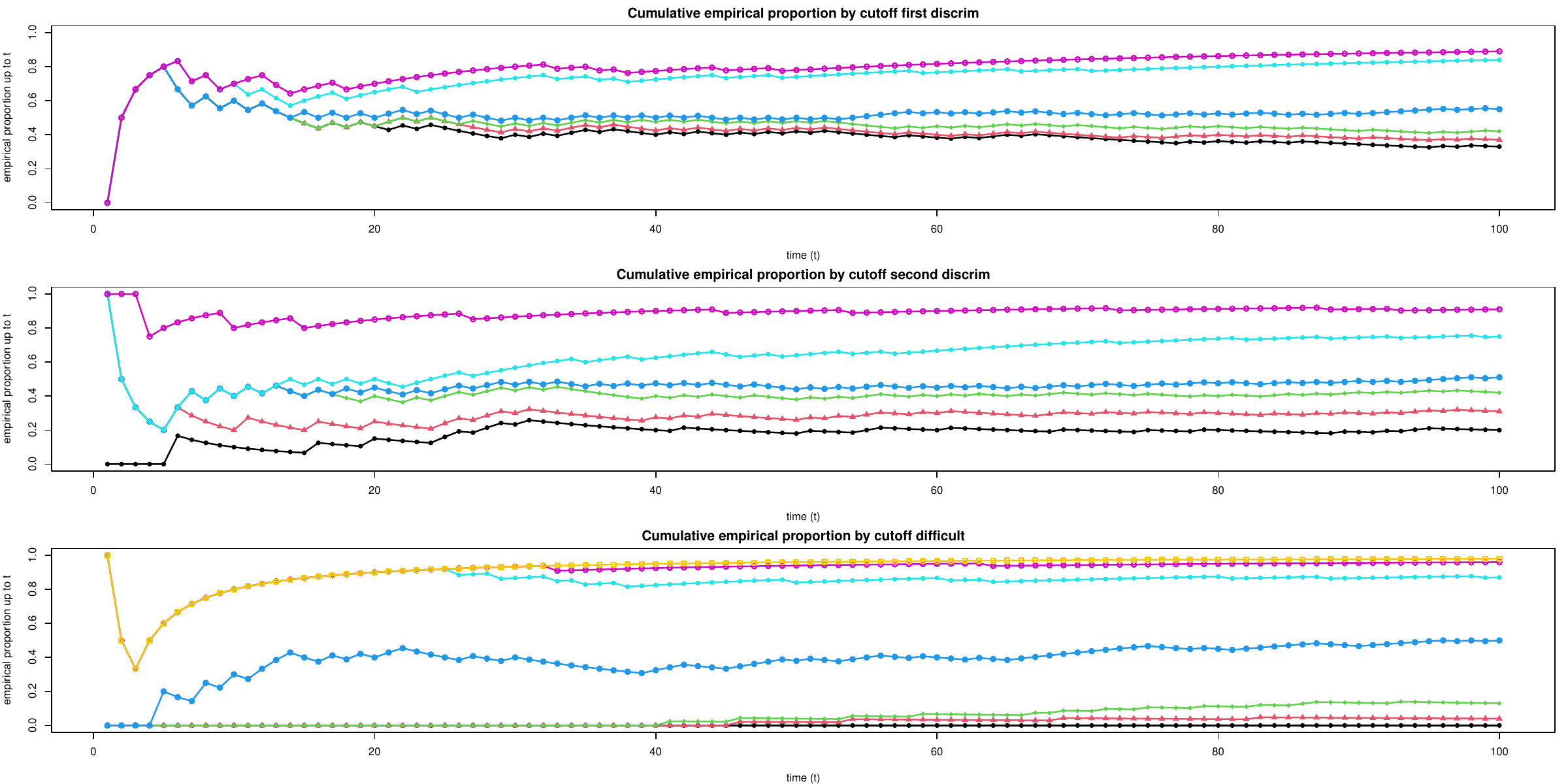}
    \caption{Empirical ratios of parameters $(\alpha_1, \alpha_2, b)$ under $\boldsymbol{\theta}^{\ast} = (0,0)$.}
    \label{fig:0_0_single}
\end{figure}

\begin{figure}[ht!]
    \centering
    \includegraphics[width=0.8\textwidth]{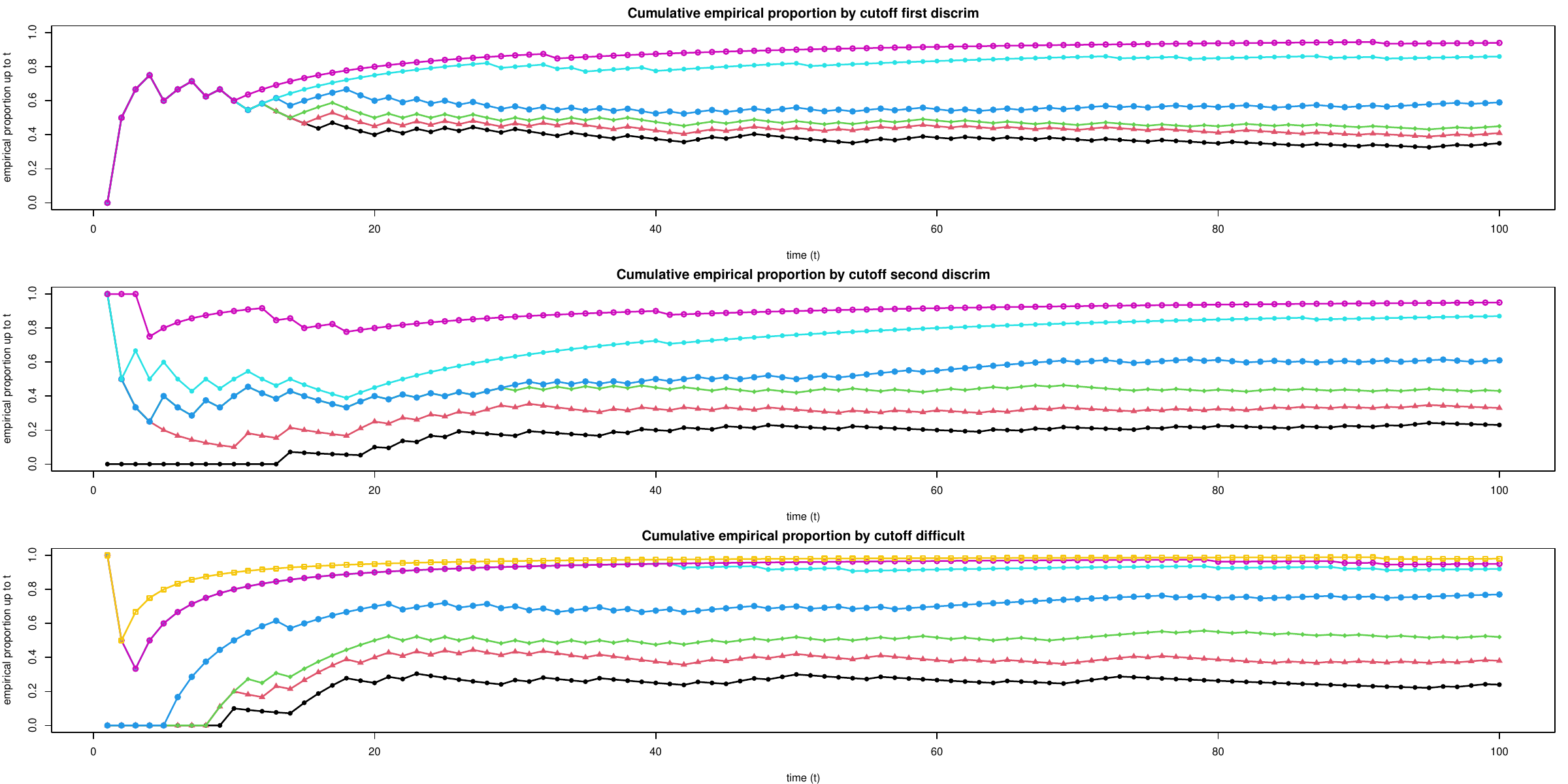}
    \caption{Empirical ratios of parameters $(\alpha_1, \alpha_2, b)$ under $\boldsymbol{\theta}^{\ast} = (0,-1)$.}
    \label{fig:0_m1_single}
\end{figure}

\begin{figure}[ht!]
    \centering
    \includegraphics[width=0.8\textwidth]{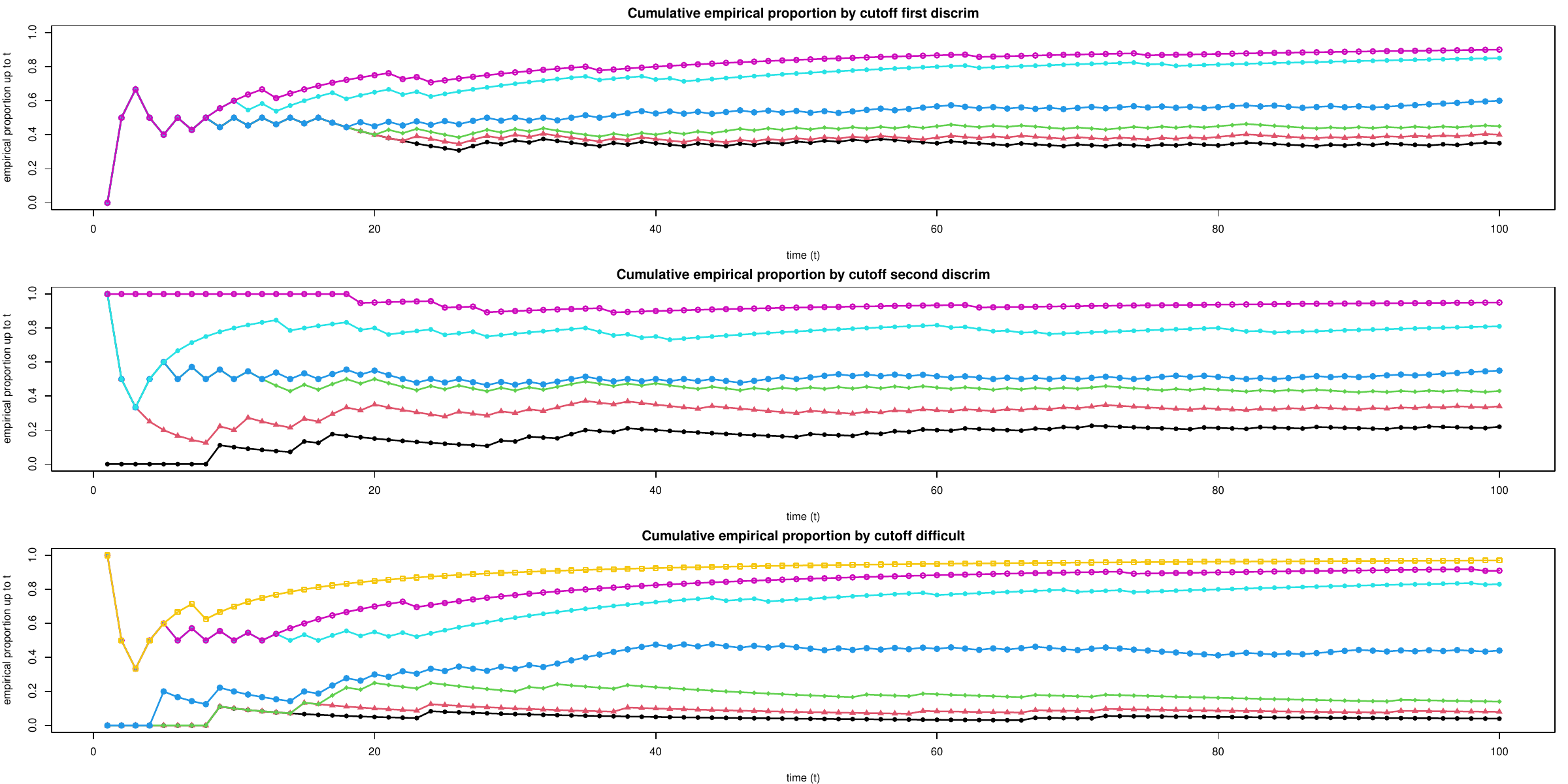}
    \caption{Empirical ratios of parameters $(\alpha_1, \alpha_2, b)$ under $\boldsymbol{\theta}^{\ast} = (-1,1)$.}
    \label{fig:m1_1_single}
\end{figure}

\begin{figure}[ht!]
    \centering
    \includegraphics[width=0.8\textwidth]{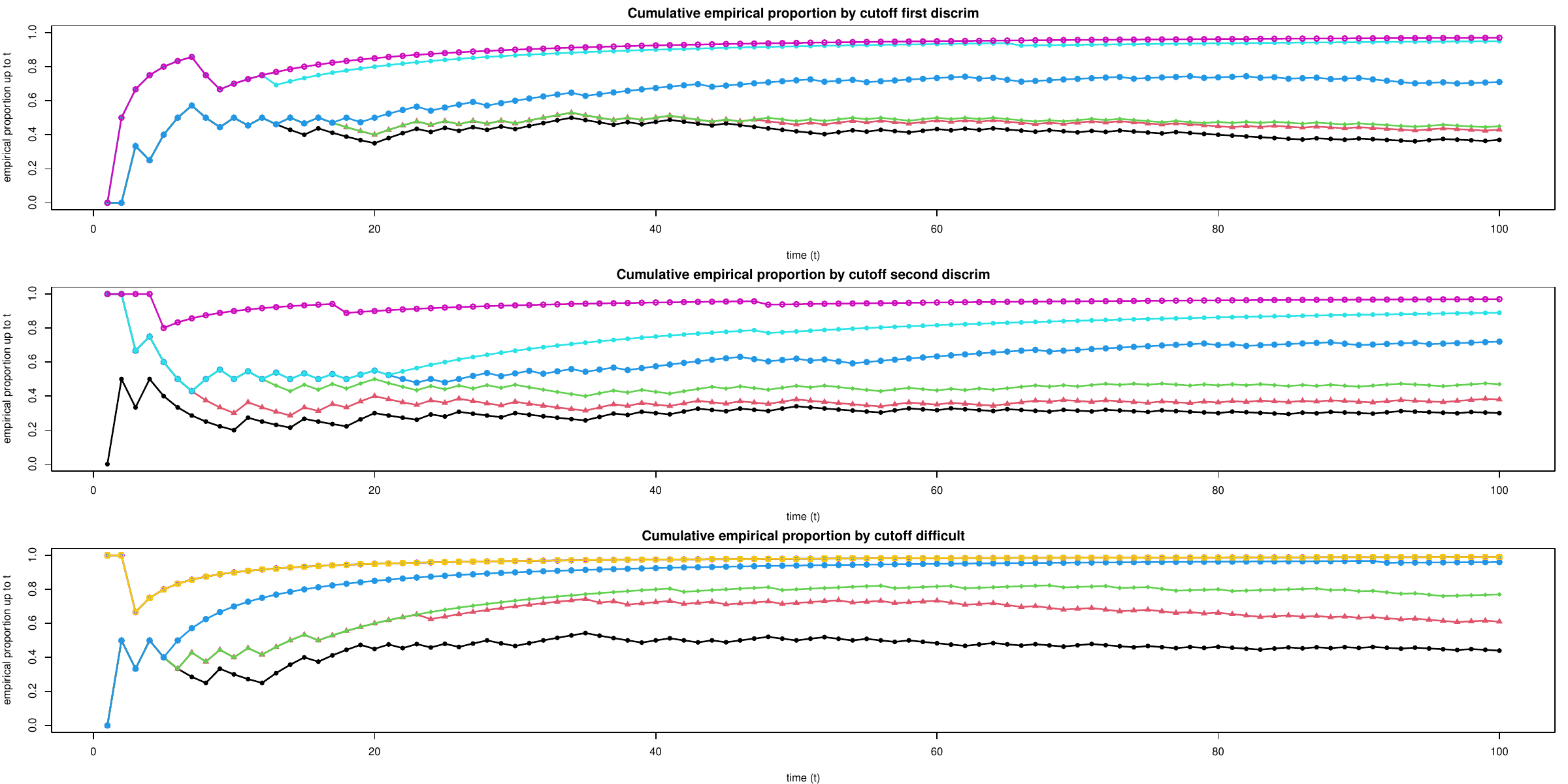}
    \caption{Empirical ratios of parameters $(\alpha_1, \alpha_2, b)$ under $\boldsymbol{\theta}^{\ast} = (-1,-1)$.}
    \label{fig:m1_m1_single}
\end{figure}
\clearpage
\refstepcounter{section}
\section*{Appendix \thesection} \label{appendix:proof}
In Appendix B, we provide proof for the theoretical results.
Under model \eqref{eq:model}, we define two random variables $Y_t$ response to the item $j_t$, adaptively selected item at step $t$, and $Y^j$ response to the item $j$. We note that $Y^j$ has the density $f_{\bftheta, j}$ under model \eqref{eq:model} is given as
\begin{equation} \label{e:density}
f_{\bftheta, j}(y^j) = \Big( \frac{\exp( \bfalpha_j^T \bftheta - b_j)}{1 + \exp( \bfalpha_j^T \bftheta - b_j)} \Big)^{y^j} \, \Big( \frac{1}{1 + \exp( \bfalpha_j^T \bftheta - b_j)} \Big)^{1 - y^j}, \quad y^j \in \{0,1\}.
\end{equation} 
Note that by defining filtration at time $t-1$ as $\mathcal{F}_{t-1} = \sigma( j_1, \cdots, j_{t-1}, Y_1, \cdots, Y_{t-1})$, $(Y_t \mid \mathcal{F}_{t-1}, j_t = j) \sim f_{\bftheta, j}$, and $Y^j \sim f_{\bftheta, j}$.

We now prove Theorem \ref{thm:optimality} and \ref{thm:normality}. The proof of Theorem \ref{thm:optimality} utilizes the proof in \citeA{li2025globally}, where verifying the $7$ regularity conditions provided in \citeA{li2025globally} is critical. Thus, we state a modified version of the $7$ regularity conditions under model \eqref{eq:model} and Assumption \ref{ass:r1}. We define a separate version of the conditions under Assumption \ref{ass:r2}, which requires additional theoretical justifications to prove Theorem \ref{thm:normality}.
 
\begin{condition}[Assumption 1 in \citeA{li2025globally}] \label{c:c1}
The parameter space $\bfTheta$ is a non-empty compact and convex subset of $\mathbb{R}^d$. The true parameter $\bftheta^{\ast}$ is an interior point of $\bfTheta$.
\end{condition}

\begin{condition}[Assumption 2 in \citeA{li2025globally}] \label{c:c2}
The support of the probability density $f_{\bftheta, j}$, $\mathrm{supp}(f_{\bftheta,j})$, depends only on $j$ and does not depend on $\bftheta$, where the support of a function is defined as
$$
\mathrm{supp}( f_{\bftheta, j} ) = \mathrm{cl} \{ y^j : f_{\bftheta, j}(y^j) > 0\},
$$
where $\mathrm{cl}(S)$ is the closure of a set $S$. Moreover, for all $j \in \{1,\cdots,M\}$ and $Y^j \in \mathrm{supp}(f_{\bftheta,j})$, the first derivative $\nabla_{\bftheta} \log f_{\bftheta, j}(Y^j) = \Big( \frac{\partial \log f_{\bftheta, j}(Y^j)}{\partial \theta_i} \Big)_{1 \leq i \leq d}$, and the Hessian matrix $\nabla_{\bftheta}^2 \log f_{\bftheta, j}(Y^j) = \Big( \frac{\partial^2 \log f_{\bftheta, j}(Y^j)}{ \partial \theta_i \partial \theta_l} \Big)_{1 \leq i,l \leq d}$ exist. Assume that there exist functions $\Psi_1^j$ and $\Psi_2^j$ satisfying $\sup_{\bftheta \in \bfTheta} \mathbb{E}_{Y^j \sim f_{\bftheta,j}} \{\Psi_1^j(Y^j) \}^2 < \infty$, $\sup_{\bftheta \in \bfTheta} \mathbb{E}_{Y^j \sim f_{\bftheta,j}} \{\Psi_2^j(Y^j) \}^2 < \infty$,
\begin{equation*}
\| \nabla_{\bftheta} \log f_{\bftheta_1, j}(Y^j) - \nabla_{\bftheta} \log f_{\bftheta_2,j}(Y^j) \|_2 \leq \Psi_1^j(Y^j) \| \bftheta_1 - \bftheta_2 \|_2,
\end{equation*}
and
\begin{equation*}
\| \nabla_{\bftheta}^2 \log f_{\bftheta_1, j}(Y^j) - \nabla_{\bftheta}^2 \log f_{\bftheta_2,j}(Y^j) \|_{op} \leq \Psi_2^j(Y^j) \| \bftheta_1 - \bftheta_2 \|_2,
\end{equation*}
for all $\bftheta_1, \bftheta_2 \in \bfTheta$, and $j \in \{1,\cdots,M\}$. Furthermore, for all $j \in \{1,\cdots,M\}$,
\begin{equation*}
\sup_{\bftheta \in \bfTheta} \mathbb{E}_{Y \sim f_{\bftheta^{\ast}, j}} \| \nabla_{\bftheta} \log f_{\bftheta,j}(Y) \|_2^2 < \infty,
\end{equation*}
and
\begin{equation*}
\sup_{\bftheta \in \bfTheta} \mathbb{E}_{Y \sim f_{\bftheta^{\ast}, j}} \| \nabla_{\bftheta}^2 \log f_{\bftheta,j}(Y) \|_{op} < \infty.
\end{equation*}
\end{condition}

\begin{condition}[Assumption 3 in \citeA{li2025globally}] \label{c:c3}
The Fisher information matrices satisfy the following conditions:
\begin{equation*}
\mathcal{I}_j(\bftheta) 
= \mathbb{E}_{Y \sim f_{\bftheta, j}} [ \nabla_{\bftheta} \log f_{\bftheta, j}(Y) \{ \nabla_{\bftheta} \log f_{\bftheta, j}(Y)\}^T ] 
= - \mathbb{E}_{Y \sim f_{\bftheta, j}} [\nabla_{\bftheta}^2 \log f_{\bftheta, j}(Y)],
\end{equation*}
and those Fisher information matrices are continuously differentiable with respect to $\bftheta$ for all $j \in \{1,\cdots,M\}$. Also, $\sum_{j \in \{1,\cdots,M\}} \mathcal{I}_j(\bftheta)$ is positive definite for every $\bftheta \in \bfTheta$.
\end{condition}

\begin{condition}[Assumption 4 in \citeA{li2025globally}] \label{c:c4}
Let $M(\bftheta; \bfpi) = \sum_{j \in \{1,\cdots,M\}} \pi_j \mathbb{E}_{Y \sim f_{\bftheta^{\ast}, j}} [\log f_{\bftheta,j}(Y)]$ for $\bfpi = (\pi_j)_{j \in \{1,\cdots,M\}}$. Assume the following uniform law of large numbers holds for all sequences $\mathbf{j}_{T_r} = (j_1,\cdots,j_{T_r})$ such that $j_i$ is measurable with respect to $\mathcal{F}_{i-1}$. For all $1 \leq i \leq T_r$,
\begin{equation*}
P \Big( \lim_{r \to \infty} \sup_{\bftheta \in \bfTheta} | \ell_{T_r}(\bftheta ; \mathbf{j}_{T_r}) - M(\bftheta ; \bar{\bfpi}_{T_r})| = 0  \Big) = 1,
\end{equation*}
where $\bar{\bfpi}_{T_r} = (\bar{\pi}_{T_r}(j ; j_1, \cdots, j_{T_r}))_{j \in \{1,\cdots,M\}}$, $\ell_{T_r}(\bftheta;\mathbf{j}_{T_r}) = \frac{1}{T_r} \sum_{s=1}^{T_r} \log f_{\bftheta,j_s}(Y_s)$, and $\bar{\pi}_{T_r}(j; j_1, \cdots, j_{T_r}) = \frac{1}{T_r} \sum_{s=1}^{T_r} \mathbf{1}_{(j_s = j)}$ for $j \in \{1,\cdots,M\}$.
\end{condition}

\begin{condition}[Assumption 5 in \citeA{li2025globally}] \label{c:c5}
The function $\Phi_{1,\delta}(\cdot):\mathcal{S}_d^{+} \to \mathbb{R}$ is convex, and it satisfies: for all positive definite matrix $\Sigma$, $\nabla \Phi_{1,\delta}(\Sigma)$ and $\nabla^2 \Phi_{1,\delta}(\Sigma)$ is continuous in $\Sigma$. Also, for all positive definite matrices satisfying $\Sigma_1 \preceq \Sigma_2$, we have $\Phi_{1,\delta}(\Sigma_1) \le \Phi_{1,\delta}(\Sigma_2)$. Finally, for $\mathcal{S}_d^+$, the set of positive definite matrices, $\sup_{\Sigma \in \mathcal{S}_d^+} \kappa( \nabla \Phi_{1,\delta}(\Sigma)) < \infty$, and $\lim_{\lambda_{max}(\Sigma) \to \infty} \Phi_{1,\delta}(\Sigma) = \infty$.
\end{condition}

\begin{condition}[Assumption 6A in \citeA{li2025globally}] \label{c:c6a}
There exist $\mathbb{R}^{d \times 1}$ vectors $\{\bfalpha_j\}_{j \in \{1,\cdots,M\}}$ and probability density functions $\{h_{\bfalpha_j^T \bftheta, j}(\cdot)\}_{j \in \{1,\cdots,M\}}$ satisfying the following requirements.
\begin{enumerate}
\item{$f_{\bftheta, j}(\cdot) = h_{\bfalpha_j^T \bftheta, j}(\cdot)$ for all $j \in \{1, \cdots, M\}$.}
\item{Let $\bfxi_j = \bfalpha_j^T \bftheta$ be a reparametrization of $\bftheta$. Assume that the Fisher information for each item $j$ with respect to $\bfxi_j$ is positive. That is, 
\begin{equation*}
\mathcal{I}_{\bfxi_j, j}(\bfxi_j) 
= \mathbb{E}_{Y \sim h_{\bfxi_j, j}}[ \{ \nabla_{\bfxi_j} \log h_{\bfxi_j, j}(Y) \}^2]
= -\mathbb{E}_{Y \sim h_{\bfxi_j, j}} [\nabla_{\bfxi_j}^2 \log h_{\bfxi_j, j}(Y)]
\end{equation*}
is positive for all $\bftheta \in \bfTheta$.}
\end{enumerate}
\end{condition}

\begin{condition}[Assumption 7A in \citeA{li2025globally}] \label{c:c7a}
There exists a constant $C_{\bfxi} > 0$ such that for all $\bftheta \in \bfTheta$, 
\begin{equation*}
\mathrm{D}_{\mathrm{KL}}(h_{\bfxi_j^{\ast}, j} || h_{\bfxi_j, j}) \geq C_{\bfxi} ( \bfxi_j^{\ast} - \bfxi_j )^2
\end{equation*}
where $\bfxi_j^{\ast} = \bfalpha_j^T \bftheta^{\ast}$, and 
\begin{equation*}
\mathrm{D}_{\mathrm{KL}}( h_{\bfxi_j^{\ast}, j} || h_{\bfxi_j, j}) = \mathbb{E}_{Y \sim h_{\bfxi_j^{\ast}, j}} \Big[ \log \Big( \frac{h_{\bfxi_j^{\ast},j}(Y)}{h_{\bfxi_j, j}(Y)} \Big) \Big].
\end{equation*}
\end{condition}

\begin{condition}[Assumption 6B in \citeA{li2025globally}] \label{c:c6b}
For $Q \subset \{1,\cdots,M\}$, define a vector space $V_Q(\bftheta) =  \sum_{j \in Q} \mathcal{R}( \mathcal{I}_j(\bftheta))$, where $\mathcal{R}(\mathbf{A})$ represents the column space of a matrix $\mathbf{A}$. Assume that the dimension $\mathrm{dim}(V_Q(\bftheta))$ does not depend on $\bftheta$, and there exist constants $0 < \underline{c} \leq \bar{c} < \infty$ which do not depend on $Q$ and $\bftheta$, such that for all $Q \subset \{1,\cdots,M\}$ and $\bftheta \in \bfTheta$
\begin{equation*}
\underline{c} \cdot \mathbf{P}_{V_{Q}(\bftheta)} \preceq \sum_{j \in Q} \mathcal{I}_j(\bftheta) \preceq \bar{c} \cdot \mathbf{P}_{V_Q(\bftheta)},
\end{equation*}
where $\mathbf{P}_{V_Q(\bftheta)}$ denotes the orthogonal projection matrix onto vector space $V_Q(\bftheta)$.
\end{condition}

\begin{condition}[Assumption 7B in \citeA{li2025globally}] \label{c:c7b}
Let $\mathcal{S}^M = \{ \bfpi = (\pi_j)_{j \in \{1,\cdots,M\}} : \sum_{j \in \{1,\cdots,M\}} \pi_j = 1, \, \text{and} \, \pi_j \geq 0 \, \text{for all} \, j \in \{1,\cdots,M\} \}$ denote the simplex in $\mathbb{R}^M$. Assume that there exists a positive constant $C_{\mathrm{KL}}$ such that for all $\bfpi \in \mathcal{S}^M$ and $\bftheta \in \bfTheta$, 
\begin{equation*}
\sum_{j \in \{1,\cdots,M\}} \pi_j \, \mathrm{D}_{\mathrm{KL}}( f_{\bftheta^{\ast},j} || f_{\bftheta, j} ) 
\geq {C}_{\mathrm{KL}} \sum_{j \in \{1, \cdots, M \} } \pi_j (\bftheta - \bftheta^{\ast})^T \mathcal{I}_{j}(\bftheta^{\ast}) (\bftheta - \bftheta^{\ast}).
\end{equation*}
\end{condition}

We now aim to show all Conditions \hyperref[c:c1]{B.1.} -- \hyperref[c:c7b]{B.9.} hold under model \eqref{eq:model} under Assumption \ref{ass:r1}. To start with, we verify Condition \hyperref[c:c5]{B.5.} is satisfied for the item-selection rule based on $\Phi_{1,\delta}$ for $\delta \in (0,1]$ under either Assumption \ref{ass:r1} or \ref{ass:r2}. 

\begin{lemma} \label{lemma:lemma1}
Assume either Assumption \ref{ass:r1} or \ref{ass:r2} holds. Let $\mathcal{S}_d^+$ be the set of positive definite symmetric matrices of dimension $d$, and $\Phi_{1,\delta}(\cdot)$ be the criterion defined in \eqref{e:selection}. Then, the following are true for $\delta \in (0,1]$.
\begin{enumerate}
\item{$\Phi_{1,\delta}(\Sigma)$ is convex over $\Sigma \in \mathcal{S}_d^+$.}
\item{For all $\Sigma \in \mathcal{S}_d^+$, $\nabla \Phi_{1,\delta}(\Sigma)$ and $\nabla^2 \Phi_{1,\delta}(\Sigma)$ are continuous in $\Sigma$.}
\item{For all matrices $\Sigma_1, \Sigma_2 \in \mathcal{S}_d^+$ satisfying $\Sigma_1 \preceq \Sigma_2$, $\Phi_{1,\delta}(\Sigma_1) \leq \Phi_{1,\delta}(\Sigma_2)$, where $\Sigma_1 \preceq \Sigma_2$ implies $\Sigma_2 - \Sigma_1$ is positive semi-definite.}
\item{$\sup_{\Sigma \in \mathcal{S}_d^+} \kappa( \nabla \Phi_{1,\delta}(\Sigma)) < \infty$, where $\kappa(\Sigma)$ is conditional number of $\Sigma \in \mathcal{S}_d^+$.}
\item{$\lim_{\lambda_{max}(\Sigma) \to \infty} \Phi_{1,\delta}(\Sigma) = \infty$, where $\lambda_{max}(\Sigma)$ is the largest eigenvalue of $\Sigma \in \mathcal{S}_d^+$.}
\end{enumerate}
That is, Condition \hyperref[c:c5]{B.5.} is satisfied for item selection rule $\Phi_{1,\delta}$.
\end{lemma}
\begin{proof}
Throughout the proof, for $q=1$, we use the identity 
\begin{equation*}
\Phi_{1,\delta}(\Sigma) = \operatorname{tr( W_{\delta} \Sigma W_{\delta})},
\end{equation*}
where $W_{\delta} = diag(w_1,\cdots,w_d)$ with $w_i = \mathbf{1}_{i \in I} + \sqrt{\delta} \mathbf{1}_{i \in N}$. The proof to Statements 1--5 are as follows.
\\ 1. For any $t \in (0,1)$, and $\Sigma_1, \Sigma_2 \in \mathcal{S}_d^+$, we have
\begin{align*}
\Phi_{1,\delta}(t \Sigma_1 + (1-t) \Sigma_2) &= \operatorname{tr}\Big( W_{\delta} ( t \Sigma_1 + (1-t) \Sigma_2) W_{\delta} \Big)
\\ & = t \cdot \Phi_{1,\delta}( \Sigma_1 ) + (1 - t) \cdot \Phi_{1,\delta} (\Sigma_2).
\end{align*}
Thus, $\Phi_{1,\delta}(\Sigma)$ is a convex function over $\mathcal{S}_d^+$. 
\\ 2. The Gateaux derivative of $\Phi_{1,\delta}(\Sigma)$ at perturbation $H$, $\nabla_H \Phi_{1,\delta}(\Sigma)$ is defined as 
\begin{equation*}
\nabla_H \Phi_{1,\delta}(\Sigma) = \lim_{\bfepsilon \to 0} \frac{\Phi_{1,\delta}(\Sigma + \bfepsilon H) - \Phi_{1,\delta}(\Sigma)}{\bfepsilon}.
\end{equation*}
Then, together with the chain rule, we obtain $\nabla_{H} \Phi_{1,\delta}(\Sigma) = \operatorname{tr}( H W_{\delta}^2)$. Using the Riesz representation theorem over the Hilbert space of symmetric positive definite matrices, $\nabla \Phi_{1,\delta}(\Sigma) = W_{\delta}^2$. This shows both $\nabla \Phi_{1,\delta}(\Sigma)$ and $\nabla^2 \Phi_{1, \delta}(\Sigma)$ are continuous in $\Sigma$.
\\ 3. Consider two matrices $\Sigma_1, \Sigma_2 \in \mathcal{S}_d^+$. From the definition of the positive semi-definite matrices, for any $x \in \mathbb{R}^d$, $x^T(\Sigma_2 - \Sigma_1)x \geq 0$. Choosing $x = W_{\delta} \tilde{x}$ for any $\tilde{x} \in \mathbb{R}^d$, we have $\tilde{x}^T W_{\delta} (\Sigma_2 - \Sigma_1) W_{\delta} \tilde{x} \geq 0$, implying $W_{\delta} \Sigma_1 W_{\delta} \preceq W_{\delta} \Sigma_2 W_{\delta}$. From the Courant-Fischer-Weyl minimax principle (see Corollary III.1.2 \citeA{bhatia2013matrix}), we have $\lambda_{i}( W_{\delta} \Sigma_1 W_{\delta}) \leq \lambda_i( W_{\delta} \Sigma_2 W_{\delta})$ for $i = 1, \ldots, d$, where $\lambda_i(\Sigma)$ is the $i$-th largest eigenvalue of $\Sigma$. Thus, we have $\Phi_{1,\delta}(\Sigma_1) \leq \Phi_{1,\delta}(\Sigma_2)$.
\\ 4. From $\nabla \Phi_{1,\delta}(\Sigma) = W_{\delta}^2$, $\kappa(\nabla \Phi_{1,\delta}(\Sigma)) = \frac{1}{\delta} < \infty$ for all $\Sigma \in \mathcal{S}_{d}^+$.
\\ 5. Let $I_d$ be an identity matrix of dimension $d$. Then, $W_{\delta} \succeq \sqrt{\delta} I_d$. Observe that for $\Sigma \in \mathcal{S}_d^+$,
\begin{align*}
\operatorname{tr}( W_{\delta} \Sigma W_{\delta} ) &= \operatorname{tr}( \Sigma W_{\delta}^2) 
\\ &= \operatorname{tr}(\Sigma( \delta I_d + W_{\delta}^2 - \delta I_d) )
\\ &= \delta \operatorname{tr}(\Sigma) + \operatorname{tr}( \Sigma^{1/2} ( W_{\delta}^2 - \delta I_d) \Sigma^{1/2} )
\\ & \ge \delta \operatorname{tr}(\Sigma),
\end{align*}
where the final inequality follows from $\Sigma^{1/2} ( W_{\delta}^2 - \delta I_d) \Sigma^{1/2}$ is positive semi-definite. Now from $\lim_{\lambda_{max}(\Sigma) \to \infty} \operatorname{tr}(\Sigma) = \infty$, we prove the statement.
\end{proof}

We now show that under Assumption \ref{ass:r1}, \ref{ass:compact}, and \ref{ass:dimension}, the remaining conditions hold. 

\begin{lemma} \label{lemma:lemma2}
Suppose Assumptions \ref{ass:r1}, \ref{ass:compact}, and \ref{ass:dimension} hold. Under the item-selection rule defined in \eqref{e:wa-optimal}, Condition \hyperref[c:c1]{B.1.} --  \hyperref[c:c7b]{B.9.} hold.
\end{lemma}
\begin{proof}
Condition \hyperref[c:c1]{B.1.} holds from the Assumption \ref{ass:compact}, and Condition \hyperref[c:c5]{B.5.} holds from Lemma \ref{lemma:lemma1}. The remaining conditions hold following the Corollary 11 and 12 from \citeA{li2025globally}.
\end{proof}

We next clarify how Assumptions \ref{ass:information1}–\ref{ass:information3} follow under R1. Suppose Assumptions \ref{ass:r1}, \ref{ass:compact}, and \ref{ass:dimension} hold. Then Assumption \ref{ass:information1} follows from Proposition 1 of \citeA{li2025globally}. Moreover, if $\Phi_{1,\delta}\!\left[ \left\{\sum_{a=1}^{M}\pi_a\mathcal{I}_a(\bftheta^\ast)\right\}^{-1}\right]$
has a unique minimizer $\bfpi^\ast$, then Assumption \ref{ass:information2} follows from Theorem 2 of \citeA{li2025globally}. Finally, Assumption \ref{ass:information3} holds automatically because, under R1, there are only finitely many distinct item types. We now prove Theorem \ref{thm:optimality}, using Lemma \ref{lemma:lemma2}, and the results proved in \citeA{li2025globally}.

\begin{proof}[Proof of Theorem \ref{thm:optimality}]
1. All regularity conditions \hyperref[c:c1]{B.1.} -- \hyperref[c:c7b]{B.9.} hold from Lemma \ref{lemma:lemma2}. Then the statement follows from Theorem 6 of \citeA{li2025globally}.
\\ 2. By choosing $L_{\delta}(\bftheta^{\ast}, \hat{\bftheta}) = \langle W_{\delta}^2(\bftheta^{\ast} - \hat{\bftheta}), \bftheta^{\ast} - \hat{\bftheta} \rangle$, we have $\operatorname{WMSE}_{\delta}(\hat{\bftheta}) = L_{\delta}(\bftheta^{\ast}, \hat{\bftheta})$ and $\delta I_d \preceq \frac{1}{2} \nabla_{\hat{\bftheta}}^2 L_{\delta}(\bftheta^{\ast}, \hat{\bftheta}) \preceq I_d$. Since all regularity conditions hold, the statement follows by applying the first statement of Theorem 2 in \citeA{li2025globally}.
\\ 3. All the regularity conditions hold from Lemma \ref{lemma:lemma2}. Then, applying the second statement of Theorem 2 in \citeA{li2025globally}, we prove the statement.
\end{proof}

Next, we prove Theorem \ref{thm:normality}.
Unlike Theorem \ref{thm:optimality}, the proof of Theorem \ref{thm:normality}
requires separate arguments under Assumptions \ref{ass:r1} and \ref{ass:r2}.
Under Assumption \ref{ass:r1}, the result follows directly from the theory in
\citeA{li2025globally}, since Lemma \ref{lemma:lemma2} allows us to verify the
required regularity conditions.
Under Assumption \ref{ass:r2}, however, additional work is needed because the
unique items in the item pool $\mathcal{J}_{\infty}$ is no longer a finite set assumed in \citeA{li2025globally}. This requires modifying several lemmas from \citeA{li2025globally} as well as adjusting the corresponding regularity conditions. We begin with stating the modified regularity conditions under Assumption \ref{ass:r2}. We only state the regularity conditions that need to be modified.

\begin{condition}[Modified Assumption 2 in \citeA{li2025globally}] \label{c:c2'}
The support of the probability density $f_{\bftheta, j}$, $\mathrm{supp}(f_{\bftheta,j})$, depends only on $j$ and does not depend on $\bftheta$, where the support of a function is defined as
$$
\mathrm{supp}( f_{\bftheta, j} ) = \mathrm{cl} \{ y^j : f_{\bftheta, j}(y^j) > 0\},
$$
where $\mathrm{cl}(S)$ is the closure of a set $S$. Moreover, for all $j \in \mathcal{J}_{\infty}$ and $Y^j \in \mathrm{supp}(f_{\bftheta,j})$, the first derivative $\nabla_{\bftheta} \log f_{\bftheta, j}(Y^j) = \Big( \frac{\partial \log f_{\bftheta, j}(Y^j)}{\partial \theta_i} \Big)_{1 \leq i \leq d}$, and the Hessian matrix $\nabla_{\bftheta}^2 \log f_{\bftheta, j}(Y^j) = \Big( \frac{\partial^2 \log f_{\bftheta, j}(Y^j)}{ \partial \theta_i \partial \theta_l} \Big)_{1 \leq i,l \leq d}$ exist. Assume that there exist constants $C_{\Psi_1} < \infty$ and $C_{\Psi_2} < \infty$ satisfying 
\begin{equation*}
\| \nabla_{\bftheta} \log f_{\bftheta_1, j}(Y^j) - \nabla_{\bftheta} \log f_{\bftheta_2,j}(Y^j) \|_2 \leq C_{\Psi_1} \| \bftheta_1 - \bftheta_2 \|_2,
\end{equation*}
and
\begin{equation*}
\| \nabla_{\bftheta}^2 \log f_{\bftheta_1, j}(Y^j) - \nabla_{\bftheta}^2 \log f_{\bftheta_2,j}(Y^j) \|_{op} \leq C_{\Psi_2} \| \bftheta_1 - \bftheta_2 \|_2,
\end{equation*}
for all $\bftheta_1, \bftheta_2 \in \bfTheta$, and $j \in \mathcal{J}_{\infty}$. Furthermore, 
\begin{equation*}
\sup_{j \in \mathcal{J}_{\infty}, \bftheta \in \bfTheta} \mathbb{E}_{Y \sim f_{\bftheta^{\ast}, j}} \| \nabla_{\bftheta} \log f_{\bftheta,j}(Y) \|_2^2 < \infty,
\end{equation*}
and
\begin{equation*}
\sup_{j \in \mathcal{J}_{\infty}, \bftheta \in \bfTheta} \mathbb{E}_{Y \sim f_{\bftheta^{\ast}, j}} \| \nabla_{\bftheta}^2 \log f_{\bftheta,j}(Y) \|_{op} < \infty.
\end{equation*}
\end{condition}

\begin{condition}[Modified Assumption 3 in \citeA{li2025globally}] \label{c:c3'}
The Fisher information matrices satisfy the following conditions:
\begin{equation*}
\mathcal{I}_j(\bftheta) 
= \mathbb{E}_{Y \sim f_{\bftheta, j}} [ \nabla_{\bftheta} \log f_{\bftheta, j}(Y) \{ \nabla_{\bftheta} \log f_{\bftheta, j}(Y)\}^T ] 
= - \mathbb{E}_{Y \sim f_{\bftheta, j}} [\nabla_{\bftheta}^2 \log f_{\bftheta, j}(Y)],
\end{equation*}
and those Fisher information matrices are continuously differentiable with respect to $\bftheta$ for all $j \in \mathcal{J}_{\infty}$. Moreover, for the selected items $(j_1,\dots,j_{T_r})$, there exist $r_0$ and $c_0 > 0$ such that for all $r \ge r_0$ and all $\bftheta \in \bfTheta$,
\begin{equation*}
\lambda_{\min}\!\left(\frac{1}{T_r}\sum_{s=1}^{T_r} \mathcal{I}_{j_s}(\bftheta)\right)\ge c_0.
\end{equation*}
\end{condition}

\begin{condition}[Modified Assumption 4 in \citeA{li2025globally}] \label{c:c4'}
Assume the following uniform law of large numbers holds for all sequences $\mathbf{j}_{T_r} = (j_1,\cdots,j_{T_r})$ such that $j_i$ is measurable with respect to $\mathcal{F}_{i-1}$. For all $1 \leq i \leq T_r$,
\begin{equation*}
P \Big\{ \lim_{r \to \infty} \sup_{\bftheta \in \bfTheta} | \ell_{T_r}(\bftheta ; \mathbf{j}_{T_r}) - M(\bftheta ; \bar{\bfpi}_{T_r})| = 0  \Big\} = 1,
\end{equation*}
where $M(\bftheta; \bar{\bfpi}_{T_r}) = \frac{1}{T_r} \sum_{s=1}^{T_r} \mathbb{E}[ \log f_{\bftheta, j_s} (Y_s) \mid \mathcal{F}_{s-1}]$.
\end{condition}

\begin{condition}[Modified Assumption 6A in \citeA{li2025globally}] \label{c:c6a'}
There exist $\mathbb{R}^{d \times 1}$ vectors $\{\bfalpha_j\}_{j \in \mathcal{J}_{\infty}}$ and probability density functions $\{h_{\bfalpha_j^T \bftheta, j}(\cdot)\}_{j \in \mathcal{J}_{\infty}}$ satisfying the following requirements.
\begin{enumerate}
\item{$f_{\bftheta, j}(\cdot) = h_{\bfalpha_j^T \bftheta, j}(\cdot)$ for all $j \in \mathcal{J}_{\infty}$.}
\item{Let $\bfxi_j = \bfalpha_j^T \bftheta$ be a reparametrization of $\bftheta$. Assume that the Fisher Information of each item $j$ is positive with respect to $\bfxi_j$. That is, 
\begin{equation*}
\mathcal{I}_{\bfxi_j, j}(\bfxi_j) 
= \mathbb{E}_{Y \sim h_{\bfxi_j, j}}[ \{ \nabla_{\bfxi_j} \log h_{\bfxi_j, j}(Y) \}^2]
= -\mathbb{E}_{Y \sim h_{\bfxi_j, j}} [\nabla_{\bfxi_j}^2 \log h_{\bfxi_j, j}(Y)]
\end{equation*}
is positive for all $\bftheta \in \bfTheta$.}
\end{enumerate}
\end{condition}

\begin{condition}[Modified Assumption 7A in \citeA{li2025globally}] \label{c:c7a'}
There exists a constant $C_{\bfxi} > 0$ such that for all $j \in \mathcal{J}_{\infty}$ and $\bftheta \in \bfTheta$, 
\begin{equation*}
\mathrm{D}_{\mathrm{KL}}(h_{\bfxi_j^{\ast}, j} || h_{\bfxi_j, j}) \geq C_{\bfxi} ( \bfxi_j^{\ast} - \bfxi_j )^2
\end{equation*}
where $\bfxi_j^{\ast} = \bfalpha_j^T \bftheta^{\ast}$, and 
\begin{equation*}
\mathrm{D}_{\mathrm{KL}}( h_{\bfxi_j^{\ast}, j} || h_{\bfxi_j, j}) = \mathbb{E}_{Y \sim h_{\bfxi_j^{\ast}, j}} \Big[ \log \Big( \frac{h_{\bfxi_j^{\ast},j}(Y)}{h_{\bfxi_j, j}(Y)} \Big) \Big].
\end{equation*}
\end{condition}

\begin{condition}[Modified Assumption 7B in \citeA{li2025globally}] \label{c:c7b'}
There exists $r_0$ such that for all $r \geq r_0$, and for some constant $\tilde{C}_{\mathrm{KL}}$ 
\begin{equation*}
\sum_{j \in A_r} \frac{1}{T_r} \, \mathrm{D}_{\mathrm{KL}}( f_{\bftheta^{\ast},j} || f_{\bftheta, j} ) 
\geq \tilde{C}_{\mathrm{KL}} \sum_{j \in A_r } \frac{1}{T_r} (\bftheta - \bftheta^{\ast})^T \mathcal{I}_{j}(\bftheta^{\ast}) (\bftheta - \bftheta^{\ast}),
\end{equation*}
where $A_r \in \mathcal{G}_{T_r}$ with $\mathcal{G}_{T_r} = \{ D \subset \mathcal{J}_{T_r} : |D| = T_r \}$.
\end{condition}

We begin with the proving lemmas that are used to prove the modified regularity conditions remain to be satisfied under Assumption \ref{ass:r2}, together with other Assumptions.

\begin{lemma} \label{lemma:lemma3}
Suppose Assumptions \ref{ass:r2}, \ref{ass:compact}, \ref{ass:information1}, and \ref{ass:information3} hold. Then the following statements are true.
\begin{enumerate}
\item{There exists a constant $C_{\ell}$ such that 
\begin{equation*}
\sup_{j \in \mathcal{J}_{\infty}, \bftheta \in \bfTheta} | \log f_{\bftheta, j}(y^j) | \leq C_{\ell}.
\end{equation*}}
\item{There exists some constant $C_{\bfalpha} > 0$ from Assumption \ref{ass:information3}, such that
\begin{equation*}
\sup_{j \in \mathcal{J}_{\infty}, \bftheta \in \bfTheta} \| \nabla_{\bftheta} \log f_{\bftheta, j}(y^j) \|_2 \leq C_{\bfalpha}.
\end{equation*}
}
\item{For any $\bftheta_1, \bftheta_2 \in \bfTheta$, there exist some constant $C_{\alpha} > 0$ such that 
\begin{equation*}
| \log f_{\bftheta_1, j}(y^j) - \log f_{\bftheta_2, j}(y^j) | \leq C_{\alpha} \| \bftheta_1 - \bftheta_2 \|_2,
\end{equation*}
where $C_{\alpha}$ is the same constant defined in Assumption \ref{ass:information3}, and does not depend on $j \in \mathcal{J}_{\infty}$.}
\item{For any $\bftheta_1, \bftheta_2 \in \bfTheta$, there exists some constant $C_{\Psi_1} > 0$ such that 
\begin{equation*}
\| \nabla_{\bftheta} \log f_{\bftheta_1, j}(y^j) - \nabla_{\bftheta} \log f_{\bftheta_2, j}(y^j) \|_{2} \leq C_{\Psi_1} \| \bftheta_1 - \bftheta_2 \|_2,
\end{equation*}
where $C_{\Psi_1}$ does not depend on $j \in \mathcal{J}_{\infty}$.}
\item{For any $\bftheta_1, \bftheta_2 \in \bfTheta$, there exists some constant $C_{\Psi_2} > 0$ such that 
\begin{equation*}
\| \nabla_{\bftheta}^2 \log f_{\bftheta_1, j}(y^j) - \nabla_{\bftheta}^2 \log f_{\bftheta_2, j}(y^j) \|_{op} \leq C_{\Psi_2} \| \bftheta_1 - \bftheta_2 \|_2,
\end{equation*}
where $C_{\Psi_2}$ does not depend on $j \in \mathcal{J}_{\infty}$.}
\item{There exists some positive constants $0 < m_1 \le m_2 < \infty$ such that $m_1 \leq | \nabla_{\bfxi_j}^2 \log h_{\bfxi_j, j}(y_j)| \leq m_2$ for all $j \in\mathcal{J}_{\infty}$ and $\bftheta \in \bfTheta$.}
\item{There exists some constant $C_{op}$ such that 
\begin{equation*}
\sup_{j \in \mathcal{J}_{\infty}, \bftheta \in \bfTheta} \| \nabla_{\bftheta}^2 \log f_{\bftheta, j}(y_j) \|_{op} \le C_{op},
\end{equation*}
where $\| \cdot \|_{op}$ is the operator norm.}
\item{For any $r \geq 1$, and any $\bftheta \in \bfTheta$,
\begin{equation*}
\frac{m_1}{m_2} \cdot \frac{1}{T_r} \Lambda_{T_r}(  \hat{\bftheta}_{T_r}^{\mathrm{ML}} ) \preceq 
\frac{1}{T_r} \Lambda_{T_r}(  \bftheta ) 
\preceq
\frac{m_2}{m_1} \cdot \frac{1}{T_r} \Lambda_{T_r}( \hat{\bftheta}_{T_r}^{\mathrm{ML}}  ),
\end{equation*}
for $m_1, m_2$ defined in the previous statement and $\Lambda_{T_r}( \bftheta )$ defined in \eqref{e:cumulative_fisher}.}
\end{enumerate}
\end{lemma}
\begin{proof}
1. From $f_{\bftheta, j}(y^j)$ defined in \eqref{e:density}, $\log f_{\bftheta, j}(y^j)$ is given as
\begin{equation} \label{e:logdensity}
\log f_{\bftheta, j}(y^j) = y^j (\bfalpha_j^T \bftheta - b_j) - \log \Big( 1 + \exp( \bfalpha_j^T \bftheta - b_j )\Big).
\end{equation}
Under Assumptions \ref{ass:compact} and \ref{ass:information3}, we obtain 
\begin{equation*}
\sup_{j \in \mathcal{J}_{\infty}, \bftheta \in \bfTheta} |\bfalpha_j^T \bftheta - b_j| \le C_{\eta},
\end{equation*}
for some $C_{\eta}$. Applying this, we can check
\begin{equation*}
\sup_{j \in \mathcal{J}_{\infty}, \bftheta \in \bfTheta} | \log f_{\bftheta, j}(y^j) | \leq C_{\eta} + \log( 1 + \exp( C_\eta) ).
\end{equation*}
Letting $C_{\ell} = C_{\eta} + \log( 1 + \exp( C_\eta) )$, we prove the statement.

2. Following the log density defined in \eqref{e:logdensity}, taking derivative with respect to $\bftheta$, we obtain
\begin{equation} \label{e:logdenstiy_deriv}
\nabla_{\bftheta} \log f_{\bftheta,j}(y^j) = \Big( y^j - \frac{\exp(\bfalpha_j^T \bftheta - b_j)}{1 + \exp(\bfalpha_j^T \bftheta - b_j)} \Big) \bfalpha_j.
\end{equation}
Note that $\sup_{j \in \mathcal{J}_{\infty}, \bftheta \in \bfTheta} \Big| y^j - \frac{\exp(\bfalpha_j^T \bftheta - b_j)}{1 + \exp(\bfalpha_j^T \bftheta - b_j)} \Big| \le 1$, and $\sup_{j \in \mathcal{J}_{\infty}} \|\bfalpha_j \|_2 \leq C_{\alpha}$ from Assumption \ref{ass:information3}. Thus, 
\begin{equation*}
\sup_{j \in \mathcal{J}_{\infty}, \bftheta \in \bfTheta} \| \nabla_{\bftheta} \log f_{\bftheta, j}(y^j) \|_2 \le C_{\alpha}
\end{equation*}
holds and we prove the statement.

3. Note that $\log f_{\bftheta, j}(y^j)$ is continuous and differentiable function in $\bftheta$. Applying Mean Value Theorem and Cauchy-Schwartz in equality, we have for any $\bftheta_1, \bftheta_2 \in \bfTheta$,
\begin{equation*}
| \log f_{\bftheta_1, j}(y^j) - \log f_{\bftheta_2, j}(y^j) | \le \| \nabla_{\bftheta} \log f_{\tilde{\bftheta}, j}(y^j) \|_2 \| \bftheta_1 - \bftheta_2 \|_2,
\end{equation*}
for some $\tilde{\bftheta}$ between $\bftheta_1$ and $\bftheta_2$. From the previous statement we proved, we obtain $\| \nabla_{\bftheta} \log f_{\tilde{\bftheta}, j}(y^j) \|_2 \le C_{\alpha}$, proving the statement.

4. Let $g_j(\bftheta) = \frac{\exp(\bfalpha_j^T \bftheta - b_j)}{1 + \exp(\bfalpha_j^T \bftheta - b_j)}$. 
Then, $g_j(\bftheta)$ is continuous and differentiable with respect to $\bftheta$. Applying Mean Value Theorem with Assumptions \ref{ass:compact} and \ref{ass:information3}, $g_j(\bftheta)$ is Lipschitz function of $\bftheta$, with some constant $C_1$ not depending on $j \in \mathcal{J}_{\infty}$. For any $\bftheta_1, \bftheta_2 \in \bfTheta$, we can check
\begin{align*}
\| \nabla_{\bftheta} \log f_{\bftheta_1, j}(y^j) - \nabla_{\bftheta} \log f_{\bftheta_2, j}(y^j) \|_2 
&\leq C_1 \|\bftheta_1 - \bftheta_2 \|_2 \| \bfalpha_j \|_2
\\ &\le C_1 \sup_{j \in \mathcal{J}_{\infty}} \| \bfalpha_j \|_2 \|\bftheta_1 - \bftheta_2 \|_2.
\end{align*}
From Assumption \ref{ass:information3}, by letting $C_{\Psi_1} = C_1 C_{\alpha}$, we prove the statement.

5. The first derivative obtained in \eqref{e:logdenstiy_deriv}, is continuous and differentiable with respect to $\bftheta$. Then, the second derivative of $\log f_{\bftheta, j}(y^j)$ is given as
\begin{equation} \label{e:logdensity_deriv2}
\nabla_{\bftheta}^2 \log f_{\bftheta, j}(y^j) = - B_j( \bfalpha_j^T \bftheta) \bfalpha_j \bfalpha_j^T,
\end{equation}
where $B_j(\bfalpha_j^T \bftheta) = \frac{\exp(\bfalpha_j^T \bftheta - b_j)}{\bigl(1+\exp(\bfalpha_j^T \bftheta - b_j) \bigr)^2}$. From $B_j(\bfalpha_j^T \bftheta)$ is continuous and differentiable with respect to $\bftheta$, applying Mean Value Theorem combined with Assumptions \ref{ass:compact} and \ref{ass:information3}, $B_j(\bfalpha_j^T \bftheta)$ is Lipschitz function of $\bftheta$, with some constant $C_1$ not depending on $j \in \mathcal{J}_{\infty}$. Then, for any $\bftheta_1, \bftheta_2 \in \bfTheta$, we obtain
\begin{align*}
\| \nabla_{\bftheta}^2 \log f_{\bftheta_1, j}(y^j) - \nabla_{\bftheta}^2 \log f_{\bftheta_2, j}(y^j) \|_{op} &\leq C_1 \| \bftheta_1 - \bftheta_2 \|_2 \| \bfalpha_j \bfalpha_j^T \|_{op}
\\ &\le C_1 \sup_{j \in \mathcal{J}_{\infty}} \| \bfalpha_j \|_2^2 \| \bftheta_1 - \bftheta_2 \|_2.
\end{align*}
From Assumption \ref{ass:information3}, by defining $C_{\Psi_2} = C_1 C_{\alpha}^2$, we prove the statement.

6. Following $\bfxi_j = \bfalpha_j^T \bftheta$, $\log h_{\bfxi_j, j}(y^j)$ is defined as we can check 
\begin{equation} \label{e:density_reparam}
\log h_{\bfxi_j^{\ast}, j}(y^j) = y^j (\bfxi_j^{\ast} - b_j) - \log( 1 + \exp( \bfxi_j^{\ast} - b_j)).
\end{equation}
Then, the second derivative with respect to $\bfxi_j$ is 
\begin{equation*}
| \nabla_{\bfxi_j}^2 \log h_{\bfxi_j, j}(y_j)| = B_{j}(\bfxi_j).
\end{equation*}
By Assumption \ref{ass:information3} and the compactness of $\bfTheta$, there exist
constants $0 < m_1 \le m_2 < \infty$ such that
\begin{equation*}
m_1 \le B_j(\bfxi_j \bftheta) \le m_2,
\quad
\forall j\in\mathcal J_\infty,\ \forall \bftheta\in\bfTheta.
\end{equation*}
This proves the statement.

7. From \eqref{e:logdensity_deriv2}, we need to verify $\sup_{j \in \mathcal{J}_{\infty}, \bftheta \in \bfTheta} \| B_j(\bfalpha_j^T \bftheta) \bfalpha_j \bfalpha_j^T \|_{op} \leq C_{op}$. Since $\nabla_{\bftheta}^2 \log f_{\bftheta, j}(y_j)$ is rank one matrix, $\| B_j(\bfalpha_j^T \bftheta) \bfalpha_j \bfalpha_j^T \|_{op} = \operatorname{tr}(B_j(\bfalpha_j^T \bftheta) \bfalpha_j \bfalpha_j^T)$. From the $\sup_{j \in \mathcal{J}_{\infty}, \bftheta \in \bfTheta} B_j( \bfxi_j \bftheta)\le m_2$ and Assumption \ref{ass:information3} we obtain
\begin{equation*}
\sup_{j \in \mathcal{J}_{\infty}, \bftheta \in \bfTheta} \operatorname{tr}(B_j(\bfalpha_j^T \bftheta) \bfalpha_j \bfalpha_j^T) \le m_2 C_{\alpha}^2.
\end{equation*}
By defining $C_{op} = m_2 C_{\alpha}^2$, we prove the statement.

8. Recall that
$\Lambda_{T_r}\bigl(\bftheta\bigr)
 = \sum_{s=1}^{T_r}\mathcal I_{j_s}(\bftheta)$
with
$\mathcal I_j(\bftheta) = B_j(\bfalpha_j^T \bftheta)\,
\bfalpha_j\bfalpha_j^T$.
Then
\begin{equation*}
\frac{1}{T_r}\Lambda_{T_r}\bigl(\bftheta \bigr)
= \frac{1}{T_r}\sum_{s=1}^{T_r}
\mathcal I_{j_s}(\hat{\bftheta}^{\mathrm{ML}}_{T_r})
\cdot
\frac{B_{j_s}(\bfalpha_{j_s}^T \bftheta)}{B_{j_s}(\bfalpha_{j_s}^T \hat{\bftheta}^{\mathrm{ML}}_{T_r})}.
\end{equation*}
By the uniform bounds on $B_j$ we have
\begin{equation*}
\frac{m_1}{m_2}
\le
\frac{B_{j_s}(\bfalpha_{j_s}^T \bftheta)}
{B_{j_s}(\bfalpha_{j_s}^T \hat{\bftheta}^{\mathrm{ML}}_{T_r})}
\le
\frac{m_2}{m_1},
\quad s=1,\dots,T_r.
\end{equation*}
Since $\mathcal I_{j_s}(\hat{\bftheta}^{\mathrm{ML}}_{T_r})$ is positive semidefinite,
summing over $s$ yields
\begin{equation*}
\frac{m_1}{m_2}\,\frac{1}{T_r}\Lambda_{T_r}\bigl(\hat{\bftheta}^{\mathrm{ML}}_{T_r}\bigr)
\preceq
\frac{1}{T_r}\Lambda_{T_r}\bigl(\bftheta\bigr)
\preceq
\frac{m_2}{m_1}\,\frac{1}{T_r}\Lambda_{T_r}\bigl( \hat{\bftheta}^{\mathrm{ML}}_{T_r} \bigr).
\end{equation*}
This completes the proof.
\end{proof}

The next lemma extends Lemma 15 from \citeA{li2025globally}.

\begin{lemma} \label{lemma:lemma4}
Let $\{Y_s\}_{s=1}^{\infty}$ be a sequence of random variables and $\{\mathcal{F}_{s}\}_{s=1}^{\infty}$ be an increasing sequence of $\sigma$-fields with $Y_s$ measurable with respect to $\mathcal{F}_s$ for all $s$. Let $\{j_s\}_{s=1}^{\infty}$ denote a sequence of discrete random variables, where each variable takes value from the set $\mathcal{J}_{\infty} = \{1,2,3, \cdots \}$. Let $Y^1,Y^2, \cdots$ be a sequence of random variables such that $\sup_{j \in \mathcal{J}_{\infty}} \mathbb{E}|Y^j|^{2 + \eta} \le C_{\eta} <\infty$ for some $C_{\eta} > 0$ and $\eta > 0$. If the conditional distribution function of $(Y_s \mid \mathcal{F}_{s-1}, j_s = j)$ is the same as the distribution function of $Y^j$ with probability $1$, then
\begin{equation*}
\frac{1}{r} \sum_{i=1}^{r} \Big\{ Y_{i} - \mathbb{E}[Y_i \mid \mathcal{F}_{i-1}] \Big\} \to 0 \quad \text{a.s}.
\end{equation*}
\end{lemma}
\begin{proof}
From Markov's inequality for any $x > 0$, 
\begin{equation*}
P( |Y^j| \ge x) \le \frac{\mathbb{E}|Y^j|^{2+\eta}}{x^{2 + \eta}} \le \sup_{j \in \mathcal{J}_{\infty}} \frac{\mathbb{E}|Y^j|^{2+\eta}}{x^{2 + \eta}} \le \frac{C_{\eta}}{x^{2+\eta}} < \infty.
\end{equation*}
Define $X_s = Y_s \mathbf{1}_{(|Y_s| \leq s)}$. First, observe that
\begin{align*}
P(|Y_s| > s) &= \mathbb{E}[ P(|Y_s| > s \mid \mathcal{F}_{s-1})]
\\ &= \mathbb{E}[ \sum_{j \in \mathcal{J}_{\infty}} P(|Y_s| > s \mid \mathcal{F}_{s-1}, j_s = j) \cdot P(j_s = j \mid \mathcal{F}_{s-1}) ]
\\ &\le \sup_{j \in \mathcal{J}_{\infty}} P( |Y^j| > s )
\\ &\le \frac{C_{\eta}}{s^{2+\eta}}
\\ &< \infty.
\end{align*}
Also, we have
\begin{align*}
\sum_{s=1}^{\infty} \frac{1}{s^2} \mathbb{E}[ (X_s - \mathbb{E}[X_s \mid \mathcal{F}_{s-1}])^2 ] 
&\le 2 \sum_{s=1}^{\infty} \frac{1}{s^2} \int_{0 < x \le s} x P(|Y_s| > x) \, dx
\\ &\le 2 \sum_{s=1}^{\infty} \frac{1}{s^2} \Big( \int_{0}^{1} x \sup_{j \in \mathcal{J}_{\infty}} P( |Y^j| > x) \, dx + \int_{1}^{s} x \frac{C_{\eta}}{x^{2+\eta}} \, dx  \Big)
\\ &\le 2 \sum_{s=1}^{\infty} \frac{1}{s^2} \Big( \frac{1}{2} + \int_{1}^{s} \frac{C_{\eta}}{x^{1+\eta}} \, dx \Big)
\\ &= 2 \Big( \frac{1}{2} + \frac{C_{\eta}}{\eta} \Big) \sum_{s=1}^{\infty} \frac{1}{s^2} 
\\ &< \infty.
\end{align*}
From the Theorem 2.15 of \citeA{hall2014martingale}, this implies $\frac{1}{r} \sum_{s=1}^{r} ( X_s - \mathbb{E}[X_s \mid \mathcal{F}_{s-1}]) \to 0$ almost surely. Combined with 
\begin{equation*}
\sum_{s=1}^{\infty} P( Y_s \neq X_s ) = \sum_{s=1}^{\infty} P( |Y_s| > s ) \le \sum_{s=1}^{\infty} \frac{C_{\eta}}{s^{2+\eta}} < \infty,
\end{equation*}
we obtain 
\begin{equation*}
\frac{1}{r} \sum_{s=1}^{r} (Y_s - \mathbb{E}[X_s \mid \mathcal{F}_{s-1}]) \to 0 \quad \text{a.s}.
\end{equation*}
Further check that 
\begin{align*}
\mathbb{E}[ |Y_s| \mathbf{1}_{(|Y_s| > s)} \mid \mathcal{F}_{s-1} ] 
&= \mathbb{E}[ \mathbb{E}[ |Y_s| \mathbf{1}_{(|Y_s| > s)} \mid \mathcal{F}_{s-1}, j_s ] \mid \mathcal{F}_{s-1}]
\\ &= \mathbb{E}[ \sum_{j \in \mathcal{J}_{\infty}} \mathbf{1}_{(j_s = j)} \cdot \mathbb{E}[ |Y_s| \mathbf{1}_{(|Y_s| > s)} \mid \mathcal{F}_{s-1}, j_s = j] \mid \mathcal{F}_{s-1} ]
\\ &\le \mathbb{E}[ \sup_{j \in \mathcal{J}_{\infty}} \mathbb{E}[ |Y^j| \mathbf{1}_{(|Y^j| > s)}] \mid \mathcal{F}_{s-1} ]
\\ &= \sup_{j \in \mathcal{J}_{\infty}} \mathbb{E}[ |Y^j| \mathbf{1}_{(|Y^j| > s)}]
\\ &\le \frac{1}{s^{\eta}} \sup_{j \in \mathcal{J}_{\infty}} \mathbb{E}[ |Y^j|^{1 + \eta} ].
\end{align*}
This implies 
\begin{align*}
\frac{1}{r} \sum_{s=1}^{r} | \mathbb{E}[ Y_s - X_s \mid \mathcal{F}_{s-1} ] 
&\le \frac{1}{r} \sum_{s=1}^{r} \mathbb{E}[ |Y_s| \mathbf{1}_{(|Y_s| > s)} \mid \mathcal{F}_{s-1}]]
\\ &\le \frac{1}{r} \sum_{s=1}^{r} \frac{\sup_{j \in \mathcal{J}_{\infty}} \mathbb{E}[ |Y^j|^{1+\eta}] }{s^{\eta}}, 
\end{align*}
and by letting $r \to \infty$, we prove the claim.
\end{proof}

We also provide a modified version of Lemma 49 from \citeA{li2025globally} as follows.  

\begin{lemma} 
\label{lemma:lemma5}
Let $\mathcal{H}^j = \{ \log f_{\bftheta, j}( \cdot ): \bftheta \in \bfTheta \}$, $j \in \mathcal{J}_{\infty}$ be collections of measurable functions with a $\mathbb{P}_{\ast}$ integrable uniform envelope function. That is, for all $\bftheta \in \bfTheta$ and some $\eta > 0$,
\begin{equation*}
\sup_{j \in \mathcal{J}_{\infty}} | \log f_{\bftheta, j}(y)| \leq F(y)
\quad \text{and} \quad \sup_{j \in \mathcal{J}_{\infty}} \mathbb{E}_{Y \sim f_{\bftheta^{\ast}, j}}[|F(Y)|^{2+\eta}] < \infty. 
\end{equation*}
Also, functions in $\mathcal{H}^j$ have a uniform Lipschitz condition in $\bftheta$. Specifically, there exists $\Psi(y)$ such that for $\bftheta_1, \bftheta_2 \in \bfTheta$,
\begin{equation*}
| \log f_{\bftheta_1, j}(y) - \log f_{\bftheta_2, j}(y) | \leq \Psi(y) \| \bftheta_1 - \bftheta_2 \| \quad \text{and} \quad \sup_{j \in \mathcal{J}_{\infty}} \mathbb{E}_{Y \sim f_{\bftheta^{\ast}, j}}[|\Psi(Y)|] < \infty.
\end{equation*}
If $\bfTheta$ is compact and mapping $\bftheta \to \log f_{\bftheta, j}(y^j)$ is continuous for every $y^j$ and $j \in \mathcal{J}_{\infty}$, then
\begin{equation*}
\mathbb{P}_{\ast} \Big\{ \lim_{r \to \infty} \sup_{\bftheta \in \bfTheta} | \ell_{T_r}(\bftheta : \mathbf{j}_{T_r}) - M(\bftheta : \bar{\bfpi}_{T_r})| = 0 \Big\} = 1.
\end{equation*}
\end{lemma}
\begin{proof}
Consider a ball $B(\bftheta, \delta) = \{\bftheta^{'} \in \bfTheta: \| \bftheta^{'} - \bftheta \|_2 < \delta \}$. Define 
\begin{equation*}
u_{B(\bftheta^{'}, \delta)}^j(y) = \sup_{\|\bftheta - \bftheta^{'}\|_2 < \delta} \log f_{\bftheta, j}(y), 
\end{equation*}
and
\begin{equation*}
l_{B(\bftheta^{'}, \delta)}^j(y) = \inf_{\|\bftheta - \bftheta^{'}\|_2 < \delta} \log f_{\bftheta, j}(y).
\end{equation*}
Then, from the uniform Lipschitz condition, for any $\epsilon > 0$, by choosing $\delta = \frac{\epsilon}{2 \sup_{j \in \mathcal{J}_{\infty}, \bftheta \in \bfTheta} \mathbb{E}[| \Psi(Y)|]}$, combined with the existence of the integrable uniform envelope $F(Y)$, we obtain
\begin{equation*}
\sup_{j \in \mathcal{J}_{\infty}} \mathbb{E}_{Y \sim f_{\bftheta^{\ast}, j}} \Big[ u_{B(\bftheta^{'}, \delta)}^j(Y) - l_{B(\bftheta^{'}, \delta)}^j(Y) \Big] < \epsilon.
\end{equation*}
From the compactness of $\bfTheta$, there exists $(\bftheta_1, \delta), \cdots, (\bftheta_m, \delta)$ such that for each $\bftheta \in \bfTheta$, choose $k(\bftheta)$ such that $\bftheta \in B(\bftheta_{k(\bftheta)}, \delta)$. Then,
\begin{equation*}
l_{B(\bftheta_{k(\bftheta)}, \delta)}^j(y) \le \log f_{\bftheta, j}(y) 
\le u_{B(\bftheta_{k(\bftheta)}, \delta)}^j(y),
\end{equation*}
for some $1 \le k \le m$. This indicates the bracketing numbers $N_{[]}(\epsilon, \mathcal{H}^j, L_1(\mathbb{P}_{\ast,j})) < \infty$ for all $\epsilon > 0$, and $j \in \mathcal{J}_{\infty}$. Thus, for each $\bftheta \in \bfTheta$, we can choose finitely many $\epsilon$-brackets $[l_{B(\bftheta_k, \delta_k)}^j(y), u_{B(\bftheta_k, \delta_k)}^j(y)]$ whose union contains $\mathcal{H}^j$. Also, for each $\bftheta \in \bfTheta$, and each $j \in \mathcal{J}_{\infty}$, 
\begin{equation*}
\mathbb{E}_{Y \sim f_{\bftheta^{\ast}, j}}[ l_{B(\bftheta_{k(\bftheta)}, \delta_k)}^j(Y)] \leq \mathbb{E}_{Y \sim f_{\bftheta^{\ast},j}}[ \log f_{\bftheta, j}(Y)] \leq \mathbb{E}_{Y \sim f_{\bftheta^{\ast}, j}}[ u_{B(\bftheta_{k(\bftheta)}, \delta_k)}^j(Y)].
\end{equation*}
Now, observe that for fixed $\bftheta \in \bfTheta$,
\begin{align*}
\ell_{T_r}(\bftheta; \mathbf{j}_{T_r}) - M(\bftheta; \bar{\bfpi}_{T_r}) 
&= \frac{1}{T_r} \sum_{s=1}^{T_r} \Big\{ 
\log f_{\bftheta, j_s}(Y_s) - \mathbb{E}[ \log f_{\bftheta, j_s}(Y_s) \mid \mathcal{F}_{s-1}] \Big\}
\\ &\le \frac{1}{T_r} \sum_{s=1}^{T_r} \Big\{ u_{B(\bftheta_k, \delta_k)}^{j_s}(Y_s) - \mathbb{E}[ u_{B(\bftheta_k, \delta_k)}^{j_s}(Y_s) \mid \mathcal{F}_{s-1}]  \Big\} + \epsilon.
\end{align*}
Taking $\sup_{\bftheta \in \bfTheta}$ on both sides, we obtain
\begin{equation*}
\sup_{\bftheta \in \bfTheta}( \ell_{T_r}(\bftheta; \mathbf{j}_{T_r}) - M(\bftheta; \bar{\bfpi}_{T_r}) ) 
\leq \max_{k \in \{1, \cdots, m\}} \frac{1}{T_r} \sum_{s=1}^{T_r} \Big\{ u_{B(\bftheta_k, \delta_k)}^{j_s}(Y_s) - \mathbb{E}[ u_{B(\bftheta_k, \delta_k)}^{j_s}(Y_s) \mid \mathcal{F}_{s-1}]  \Big\} + \epsilon.
\end{equation*}
Note that $\sup_{j \in \mathcal{J}_{\infty}} | u_{B(\bftheta_k, \delta_k)}^{j}(y)| \leq F(y)$, $F(Y)$ has bounded $(2 + \eta)$ moment, from Lemma \ref{lemma:lemma4} $\frac{1}{T_r} \sum_{s=1}^{T_r} \Big\{ u_{B(\bftheta_k, \delta_k)}^{j_s}(Y_s) - \mathbb{E}[ u_{B(\bftheta_k, \delta_k)}^{j_s}(Y_s) \mid \mathcal{F}_{s-1}]  \Big\}$ almost surely converges to $0$. This yields 
\begin{equation*}
\limsup_{r \to \infty} \sup_{\bftheta \in \bfTheta}( \ell_{T_r}(\bftheta; \mathbf{j}_{T_r}) - M(\bftheta; \bar{\bfpi}_{T_r}) ) \leq \epsilon \quad \mathbb{P}_{\ast} \, \text{a.s},
\end{equation*}
and a similar argument gives 
\begin{equation*}
\liminf_{r \to \infty} \inf_{\bftheta \in \bfTheta}( \ell_{T_r}(\bftheta; \mathbf{j}_{T_r}) - M(\bftheta; \bar{\bfpi}_{T_r}) ) \geq -\epsilon \quad \mathbb{P}_{\ast} \, \text{a.s}.
\end{equation*}
Taking $\epsilon \to 0$, we prove the statement.
\end{proof}

Using results from Lemma \ref{lemma:lemma3}, we show under Assumption \ref{ass:r2}, the modified regularity Conditions  \hyperref[c:c2']{B.10.},
\hyperref[c:c3']{B.11.}, \hyperref[c:c4']{B.12.}, \hyperref[c:c6a']{B.13.}, \hyperref[c:c7a']{B.14.}, \hyperref[c:c7b']{B.15.}, and regularity Conditions \hyperref[c:c1]{B.1.}, \hyperref[c:c5]{B.5.} are satisfied. Note that we do not require Condition \hyperref[c:c6b]{B.8.} to hold under Assumption \ref{ass:r2}.

\begin{lemma} \label{lemma:lemma6}
Suppose Assumptions from \ref{ass:r2} -- \ref{ass:information3} hold. Under model \eqref{eq:model}, Conditions \hyperref[c:c1]{B.1.}, \hyperref[c:c2']{B.10.}, \hyperref[c:c3']{B.11.}, \hyperref[c:c4']{B.12.}, \hyperref[c:c5]{B.5.}, \hyperref[c:c6a']{B.13.}, \hyperref[c:c7a']{B.14.}, \hyperref[c:c7b']{B.15.} are satisfied.
\end{lemma}
\begin{proof}
First, the regularity Conditions \hyperref[c:c1]{B.1.} and \hyperref[c:c5]{B.5.} are satisfied from the Assumption \ref{ass:compact} and Lemma \ref{lemma:lemma1}. For modified regularity Condition \hyperref[c:c2']{B.10.}, the existence of $C_{\Psi_1}, C_{\Psi_2}$ is satisfied by fourth, and fifth statements of the Lemma \ref{lemma:lemma3}. The bounds in expectation are also satisfied from the second and seventh statements of the Lemma \ref{lemma:lemma3}. The first part of the modified regularity Condition \hyperref[c:c3']{B.11.} holds from the smoothness of the $\mathcal{I}_j(\bftheta)$ stated in Corollary 11 of \citeA{li2025globally}. For the second statement, we can take $c_0 = \frac{m_1}{m_2} c$ by applying the eighth statement of Lemma \ref{lemma:lemma3} together with Assumption \ref{ass:information1}.
The modified Condition \hyperref[c:c4']{B.12.} holds by applying Lemma \ref{lemma:lemma5}, where the conditions in the Lemma hold from the first and the third statement of Lemma \ref{lemma:lemma3}.
The modified Conditions \hyperref[c:c6a']{B.13.}, and \hyperref[c:c7a']{B.14.} are satisfied following the identical reasoning as Corollary 11 of \citeA{li2025globally}, using $m_1$ as the lower bound for $\inf_{j \in \mathcal{J}_{\infty}, \bftheta \in \bfTheta} B_j( \bfxi_j )$, and letting $C_{\bfxi} = m_1/2$. The last modified regularity Condition \hyperref[c:c7b']{B.15.} is satisfied from the arguments following Lemma 25 in \citeA{li2025globally}. Specifically, under Conditions \hyperref[c:c6a']{B.13.} and \hyperref[c:c7a']{B.14.}, we obtain
\begin{equation*}
\mathrm{D}_{\mathrm{KL}}( f_{\bftheta^{\ast}, j} || f_{\bftheta, j}) \ge \frac{C_{\bfxi}}{m_2} (\bftheta - \bftheta^{\ast})^{T} \mathcal{I}_j(\bftheta^{\ast}) (\bftheta - \bftheta^{\ast}).
\end{equation*}
The constant $\frac{C_{\bfxi}}{m_2}$ does not depend on $j \in \mathcal{J}_{\infty}$, taking summation over any subset of size $T_r$ or greater preserves the inequality. 
\end{proof}

\begin{proof}[Proof of Theorem \ref{thm:normality}]
Throughout the proof, we note under Assumption~\ref{ass:r1}, Assumption \ref{ass:dimension} holds from Assumption \ref{ass:information1}.
\\ 1. Under Assumption~\ref{ass:r1}, combined with Assumptions~\ref{ass:compact}, 
\ref{ass:dimension}, and Lemma \ref{lemma:lemma2}, the statement follows directly 
by applying Theorem 4 in \citeA{li2025globally}. We now prove the result under 
Assumption \ref{ass:r2}. Define
\begin{equation*}
M(\bftheta; \bar{\bfpi}_{T_r})
= \frac{1}{T_r} \sum_{s=1}^{T_r} 
  \mathbb{E}\bigl[ \log f_{\bftheta,j_s}(Y_{s}) \mid \mathcal{F}_{s-1} \bigr].
\end{equation*}
To apply Theorem 4 in \citeA{li2025globally}, it suffices to show that there exists $r_0$ such that for all 
$r \ge r_0$,
\begin{equation*}
\sup_{\bftheta:\,\|\bftheta - \bftheta^{\ast}\| \ge \epsilon}
M(\bftheta; \bar{\bfpi}_{T_r})
\;\le\;
M(\bftheta^{\ast}; \bar{\bfpi}_{T_r}) - \eta,
\end{equation*}
for some $\eta>0$. From Lemma \ref{lemma:lemma6}, modified regularity conditions are satisfied under Assumptions \ref{ass:r2} -- \ref{ass:information3}. Since Condition \hyperref[c:c7b']{B.15.} is satisfied for the selected set $A_r=\{j_1,\ldots,j_{T_r}\}\in\mathcal{G}_{T_r}$, we have
\begin{align*}
M(\bftheta^{\ast}; \bar{\bfpi}_{T_r}) 
- M(\bftheta; \bar{\bfpi}_{T_r})
&= \frac{1}{T_r} \sum_{s = 1}^{T_r} \mathrm{D}_{\mathrm{KL}}( f_{\bftheta^{\ast}, j_s} || f_{\bftheta, j_s}) 
\\ &= \frac{1}{T_r} \sum_{s=1}^{T_r} \mathbb{E} \Big[ \log \frac{f_{\bftheta^{\ast},j_s}(Y_s)}{f_{\bftheta,j_s}(Y_s)} \Big| \mathcal{F}_{s-1} \Big]
\\ & \ge \tilde{C}_{\mathrm{KL}} \cdot 
\frac{1}{T_r} \sum_{s = 1}^{T_r} 
(\bftheta - \bftheta^{\ast})^\top 
\mathcal{I}_{j_s}(\bftheta^{\ast})
(\bftheta - \bftheta^{\ast}).
\end{align*}
Rearranging and taking $\sup_{\bftheta:\,\|\bftheta - \bftheta^{\ast}\|_2 \ge \epsilon}$, 
we obtain for all $r \geq r_0$
\begin{align*}
\sup_{\bftheta:\,\|\bftheta - \bftheta^{\ast}\|_2 \ge \epsilon}
M(\bftheta; \bar{\bfpi}_{T_r})
&\le
M(\bftheta^{\ast}; \bar{\bfpi}_{T_r})
- \tilde{C}_{\mathrm{KL}} \cdot 
\inf_{\bftheta:\,\|\bftheta - \bftheta^{\ast}\|_2 \ge \epsilon}
\Bigl\{
\frac{1}{T_r} (\bftheta - \bftheta^{\ast})^T
\Bigl(\sum_{s=1}^{T_r} \mathcal{I}_{j_s}(\bftheta^{\ast})\Bigr)
(\bftheta - \bftheta^{\ast})
\Bigr\} 
\\&\le
M(\bftheta^{\ast}; \bar{\bfpi}_{T_r})
- \tilde{C}_{\mathrm{KL}} \cdot 
\lambda_{\min} \Bigl(
\frac{1}{T_r} \sum_{s=1}^{T_r} \mathcal{I}_{j_s}(\bftheta^{\ast})
\Bigr)
\inf_{\bftheta:\,\|\bftheta - \bftheta^{\ast}\|_2 \ge \epsilon}
\|\bftheta - \bftheta^{\ast}\|_2^2 
\\&\le
M(\bftheta^{\ast}; \bar{\bfpi}_{T_r})
- \tilde{C}_{\mathrm{KL}} \cdot c \cdot \frac{m_1}{m_2}\,\epsilon^2,
\end{align*}
where the first inequality follows from $\lambda_{min}(\sum_{s=1}^{T_r} \mathcal{I}_{j_s}(\bftheta^{\ast})) I_d \preceq \sum_{s=1}^{T_r} \mathcal{I}_{j_s}(\bftheta^{\ast})$, and the last inequality follows from Assumption \ref{ass:information1} and the last
statement of Lemma \ref{lemma:lemma3}, which together imply
\begin{equation*}
\lambda_{\min} \Bigl(
\frac{1}{T_r} \sum_{s=1}^{T_r} \mathcal{I}_{j_s}(\bftheta^{\ast})
\Bigr)
\;\ge\;
\frac{m_1}{m_2}\,c.
\end{equation*}
Setting $\eta = \tilde{C}_{\mathrm{KL}} \cdot c \cdot \frac{m_1}{m_2}\,\epsilon^2$ yields
\begin{equation*}
\sup_{\bftheta:\,\|\bftheta - \bftheta^{\ast}\| \ge \epsilon}
M(\bftheta; \bar{\bfpi}_{T_r})
\;\le\;
M(\bftheta^{\ast}; \bar{\bfpi}_{T_r}) - \eta,
\end{equation*}
and Theorem 4 in \citeA{li2025globally} then applies, which proves the statement under Assumption \ref{ass:r2}.

2. Under Assumption \ref{ass:r1}, combined with Assumptions \ref{ass:compact}, \ref{ass:dimension}, and Lemma \ref{lemma:lemma2}, the statement follows directly
by applying Theorem 5 in \citeA{li2025globally}. We now prove the result under Assumption \ref{ass:r2}. To apply Theorem 5 in \citeA{li2025globally}, we start by showing that
\begin{equation} \label{e:target1}
\lim_{r \to \infty}
\Big\|
\frac{1}{T_r} \sum_{s=1}^{T_r}
\bigl\{-\nabla_{\bftheta}^2 \log f_{\bftheta^{\ast}, j_s}(Y_{s})\bigr\}
- \mathcal{I}_{\infty}
\Big\|_{\mathrm{op}}
= 0,
\end{equation}
\begin{equation} \label{e:target2}
\lim_{r \to \infty}
\Big\|
\frac{1}{T_r} \sum_{s=1}^{T_r}
\nabla_{\bftheta} \log f_{\bftheta^{\ast}, j_s}(Y_{s})
\nabla_{\bftheta} \log f_{\bftheta^{\ast}, j_s}(Y_s)^\top
- \mathcal{I}_{\infty}
\Big\|_{\mathrm{op}}
= 0.
\end{equation}

By Lemma \ref{lemma:lemma6}, the Condition \hyperref[c:c3']{B.11.} is satisfied. In particular, for all $s \ge 1$ and
$\bftheta \in \bfTheta$, we have the conditional Fisher identity
\begin{equation*}
\mathbb{E}\bigl[
 -\nabla_{\bftheta}^2 \log f_{\bftheta, j_s}(Y_{s})
 \,\big|\, \mathcal{F}_{s-1}
\bigr]
=
\mathbb{E}\bigl[
 \nabla_{\bftheta} \log f_{\bftheta, j_s}(Y_{s})
 \nabla_{\bftheta} \log f_{\bftheta, j_s}(Y_{s})^\top
 \,\big|\, \mathcal{F}_{s-1}
\bigr]
= \mathcal{I}_{j_s}(\bftheta).
\end{equation*}
From the second and seventh statements of the Lemma \ref{lemma:lemma3}, 
\begin{equation*}
\sup_{j \in \mathcal{J}_{\infty}, \bftheta \in \bfTheta} \|\nabla_{\bftheta} \log f_{\bftheta, j}(Y^j)\|_2 \leq C_{\alpha},
\quad
\sup_{j \in \mathcal{J}_{\infty}, \bftheta \in \bfTheta} \|\nabla_{\bftheta}^2 \log f_{\bftheta, j}(Y^j)\|_{\mathrm{op}} \le C_{op},
\end{equation*}
implying elementwise uniform bound of $(\nabla_{\bftheta} \log f_{\bftheta, j}(Y^j) \nabla_{\bftheta} \log f_{\bftheta, j}(Y^j)^T)$ and $(\nabla_{\bftheta}^2 \log f_{\bftheta, j}(Y^j))$.
Therefore, by applying Lemma \ref{lemma:lemma4} elementwise, we have
\begin{equation} \label{e:target3}
\lim_{r \to \infty}
\Big\|
\frac{1}{T_r} \sum_{s=1}^{T_r}
\Big(
 -\nabla_{\bftheta}^2 \log f_{\bftheta^{\ast}, j_s}(Y_s)
 - \mathcal{I}_{j_s}(\bftheta^{\ast})
\Big)
\Big\|_{\mathrm{op}}
= 0,
\end{equation}
\begin{equation} \label{e:target4}
\lim_{r \to \infty}
\Big\|
\frac{1}{T_r} \sum_{s=1}^{T_r}
\Big(
 \nabla_{\bftheta} \log f_{\bftheta^{\ast}, j_s}(Y_s)
 \nabla_{\bftheta}\log f_{\bftheta^{\ast}, j_s}(Y_s)^\top
 - \mathcal{I}_{j_s}(\bftheta^{\ast})
\Big)
\Big\|_{\mathrm{op}}
= 0.
\end{equation}
Next, from the fifth statement of Lemma \ref{lemma:lemma3}, for $\bftheta_{s,1}, \bftheta_{s,2} \in \bfTheta$ for $s = 1,\cdots,T_r$, we have
\begin{equation*}
\Big\|
\frac{1}{T_r} \sum_{s=1}^{T_r}
\bigl( \mathcal{I}_{j_s}(\bftheta_{s,1}) - \mathcal{I}_{j_s}(\bftheta_{s,2}) \bigr)
\Big\|_{\mathrm{op}}
\le \frac{C_{\Psi_2}}{T_r} \sum_{s=1}^{T_r} \|\bftheta_{s,1} - \bftheta_{s,2}\|_2.
\end{equation*}
Taking $\bftheta_{s,1} = \hat{\bftheta}_{T_r}^{\mathrm{ML}}$ and $\bftheta_{s,2} = \bftheta^{\ast}$, and using the first part of the theorem (which shows $\hat{\bftheta}_{T_r}^{\mathrm{ML}} \to \bftheta^{\ast}$
almost surely), we obtain
\begin{equation*}
\lim_{r \to \infty}
\Big\|
\frac{1}{T_r} \sum_{s=1}^{T_r}
\bigl(
  \mathcal{I}_{j_s}(\hat{\bftheta}_{T_r}^{\mathrm{ML}})
  - \mathcal{I}_{j_s}(\bftheta^{\ast})
\bigr)
\Big\|_{\mathrm{op}}
= 0.
\end{equation*}
Combining this with the Assumption \ref{ass:information2} and using the triangle inequality, we obtain
\begin{equation} \label{e:fisherlimit}
\lim_{r \to \infty} \Big \| \frac{1}{T_r}\sum_{s=1}^{T_r}
\mathcal{I}_{j_s}(\bftheta^{\ast}) - \mathcal{I}_{\infty} \Big \|_{op} = 0. 
\end{equation}
Now, we can apply the Part I to Part III proof in \citeA{li2025globally} with a few fixes. We substitute $\sum_{a \in \mathcal{A}} \pi(a) \mathcal{I}_a(\bftheta^{\ast})$ with $\mathcal{I}_{\infty}$ throughout Parts I to III. In Part I, we use the second statement of Lemma \ref{lemma:lemma3} to argue the asymptotic normality. Specifically, define $\psi_{T_r,s} = \frac{1}{\sqrt{T_r}} \mathbf{b}^T \nabla_{\bftheta} \log f_{\bftheta^{\ast},j_s}(Y_s)$, with any $\mathbf{b} \in \mathbb{R}^d$ satisfying $\|\mathbf{b}\|_2 = 1$. Then, we can prove 
\begin{equation*}
\sum_{s=1}^{T_r} \mathbb{E}[ \psi_{T_r, s}^2 \mathbf{1}_{(| \psi_{T_r,s}| > \epsilon)} \mid \mathcal{F}_{s-1} ] 
\le \frac{C_{\alpha}^2}{T_r} \sum_{s=1}^{T_r} P( |\mathbf{b}^T \nabla_{\bftheta} \log f_{\bftheta^{\ast}, j_s}(Y_s)| \ge \sqrt{T_r} \epsilon \mid \mathcal{F}_{s-1}) 
\to 0,
\end{equation*}
as $r \to \infty$ since $|\mathbf{b}^T \nabla_{\bftheta} \log f_{\bftheta^{\ast}, j_s}(Y_s)| \leq C_{\alpha}$. 

For Part II, we have proved the convergence in \eqref{e:target3}, \eqref{e:target4}, and \eqref{e:fisherlimit}. Lemma \ref{lemma:lemma4} is used instead of Lemma 15 in \citeA{li2025globally}. Note that from the seventh statement of Lemma \ref{lemma:lemma3}, we have almost sure bound for the $\sup_{j \in \mathcal{J}_{\infty}, \bftheta \in \bfTheta} \| \nabla_{\bftheta}^2 \log f_{\bftheta, j}(Y^j) \|_{op}$, which indicates finite $(2+\eta)$ moment, allowing to apply Lemma \ref{lemma:lemma4}.   

For Part III, we replace $\Psi_2^{a_j}$ in the Lipschitz bound for $\|\nabla_{\bftheta}^2 \log f_{\bftheta_1, j}(y^j) - \nabla_{\bftheta}^2 \log f_{\bftheta_2, j}(y^j)\|_{op}$ in Condition \hyperref[c:c3]{B.3.} with stronger bound $C_{\Psi_2}$ defined in the Condition \hyperref[c:c2']{B.10.}, which follows from the fifth statement of Lemma \ref{lemma:lemma3}. Then the statement follows from the identical argument in \citeA{li2025globally}.

3. Under Assumption \ref{ass:r1}, combined with Assumptions \ref{ass:compact},
\ref{ass:dimension}, and Lemma \ref{lemma:lemma2}, the statement follows directly by applying the second statement of Theorem 2 in \citeA{li2025globally}. We devote the rest to proving the result under Assumption \ref{ass:r2}. To apply the second statement of Theorem 2 in \citeA{li2025globally} under Assumption \ref{ass:r2}, we
need to first replace the Condition \hyperref[c:c6b]{B.8.} to find $\underline{C} > 0$ satisfying
\begin{equation*}
\frac{1}{T_r} \sum_{s=1}^{T_r} B_{j_s}(\bfalpha_{j_s}^T \bftheta) \bfalpha_{j_s} \bfalpha_{j_s}^T \succeq \underline{C} I_d.
\end{equation*}
From the last statement of Lemma \ref{lemma:lemma3} combined with Assumption \ref{ass:information1}, we can set $\underline{C} = c \cdot \frac{m_1}{m_2}$. We also need to show some $\eta > 0$,
\begin{equation} \label{e:target5}
\sup_{j \in \mathcal{J}_{\infty}, \bftheta \in \bfTheta} \mathbb{E}_{Y \sim f_{\bftheta^{\ast}, j}} \| \nabla_{\bftheta} \log f_{\bftheta, j}(Y) \|_2^{2 + \eta} < \infty. 
\end{equation}
This holds from the second statement of Lemma \ref{lemma:lemma3}. Then, following the same reasoning as in \citeA{li2025globally}, we prove the statement. 
\end{proof}

\end{document}